\documentclass[a4paper]{amsart}
\usepackage[utf8]{inputenc}
\usepackage[english]{babel}
\usepackage{amsmath}
\usepackage{amsfonts}
\usepackage{amssymb}
\usepackage{amsthm}
\usepackage{wasysym}
\usepackage{amscdx}
\usepackage{graphicx}
\usepackage{chngcntr}
\usepackage{tikz-cd}
\usepackage{xcolor}
\usepackage{marvosym}
\usepackage{enumitem}
\usepackage{stix}
\usepackage{todonotes}
\usepackage{comment}

\usepackage{hyperref}
\hypersetup{colorlinks=true}

\makeatletter
\renewcommand\@mkboth[2]{\markboth{#1}{}}
\makeatother

\newtheorem{thm}{Theorem}[section]
\newtheorem{lem}[thm]{Lemma}
\newtheorem{prop}[thm]{Proposition}
\newtheorem{claim}{Claim}[thm]
\newtheorem{cor}[thm]{Corollary}
\newtheorem{fact}[thm]{Fact}
\theoremstyle{definition}
\newtheorem{defn}[thm]{Definition}
\newtheorem{question}{Question}
\theoremstyle{remark}
\newtheorem{exmp}[thm]{Example}
\newtheorem{remark}[thm]{Remark}

\newenvironment{enumerate-(a)}{\begin{enumerate}[label={\upshape (\alph*)}, leftmargin=2pc]}{\end{enumerate}}
\newenvironment{enumerate-(a)-r}{\begin{enumerate}[label={\upshape (\alph*)}, leftmargin=2pc,resume]}{\end{enumerate}}
\newenvironment{enumerate-(a)-5}{\begin{enumerate}[label={\upshape (\alph*)}, leftmargin=2pc,start=5]}{\end{enumerate}}
\newenvironment{enumerate-(A)}{\begin{enumerate}[label={\upshape (\Alph*)}, leftmargin=2pc]}{\end{enumerate}}
\newenvironment{enumerate-(A)-r}{\begin{enumerate}[label={\upshape (\Alph*)}, leftmargin=2pc,resume]}{\end{enumerate}}
\newenvironment{enumerate-(i)}{\begin{enumerate}[label={\upshape (\roman*)}, leftmargin=2pc]}{\end{enumerate}}
\newenvironment{enumerate-(i)-r}{\begin{enumerate}[label={\upshape (\roman*)}, leftmargin=2pc,resume]}{\end{enumerate}}
\newenvironment{enumerate-(I)}{\begin{enumerate}[label={\upshape (\Roman*)}, leftmargin=2pc]}{\end{enumerate}}
\newenvironment{enumerate-(I)-r}{\begin{enumerate}[label={\upshape (\Roman*)}, leftmargin=2pc,resume]}{\end{enumerate}}
\newenvironment{enumerate-(1)}{\begin{enumerate}[label={\upshape (\arabic*)}, leftmargin=2pc]}{\end{enumerate}}
\newenvironment{enumerate-(1)-r}{\begin{enumerate}[label={\upshape (\arabic*)}, leftmargin=2pc,resume]}{\end{enumerate}}
\newenvironment{itemizenew}{\begin{itemize}[leftmargin=2pc]}{\end{itemize}}
\newenvironment{enumerate-(star)}{\begin{enumerate}[label={\upshape{(\( \star_{ \arabic*} \))}}, leftmargin=2pc]}{\end{enumerate}}

\newcommand{\Lin}{\mathsf{Lin}}
\newcommand{\Scat}{\mathsf{Scat}}
\newcommand{\ALin}{\mathsf{Lin}_{\infty}}
\newcommand{\AScat}{\mathsf{Scat}_{\infty}}
\newcommand{\QScat}{\Q_\kappa \text{-}\mathsf{Scat}_{\kappa}}
\newcommand{\Fin}{\mathsf{Fin}}
\newcommand{\LO}{\mathsf{LO}}
\newcommand{\WO}{\mathsf{WO}}
\newcommand{\DLO}{\mathsf{DLO}}
\newcommand{\ADLO}{\mathsf{DLO}_\infty}

\DeclareMathOperator{\rk}{rk_H}
\DeclareMathOperator{\Int}{Int}

\DeclareMathOperator{\Supp}{Supp}
\DeclareMathOperator{\id}{id}

\newcommand{\emb}{\preceq_{\LO}}

\newcommand{\cvx}{\trianglelefteq_{\LO}}
\newcommand{\ncvx}{\ntrianglelefteq_{\LO}}
\newcommand{\cvxeq}{\underline{\bowtie}_{\LO}}
\newcommand{\iso}{\cong_{\LO}}

\newcommand{\qo}{\trianglelefteq^{\mathcal{L}}_\LO}
\newcommand{\qok}{\trianglelefteq^{\mathcal{L}}_\kappa}
\newcommand{\qoa}{\trianglelefteq^{\mathcal{L}}_{\aleph_1}}
\newcommand{\qon}{\trianglelefteq^{\mathcal{L}}_{\LO_\N}}
\newcommand{\nqo}{\ntrianglelefteq^{\mathcal{L}}_{\LO}}
\newcommand{\qog}{\trianglelefteq^{\L_{\prec \boldsymbol{\gamma}+\boldsymbol{1}}}}
\newcommand{\qolg}{\trianglelefteq^{{\mathcal{L}}_{\prec \boldsymbol{\gamma}}}}
\newcommand{\eq}{\mathrel{\underline{\bowtie}}^{\mathcal{L}}_\LO}
\newcommand{\eqk}{\mathrel{\underline{\bowtie}}^{\mathcal{L}}_\kappa}

\DeclareMathOperator{\ot}{\text{ot}}

\newcommand{\z}{\zeta}
\renewcommand{\o}{\omega}
\newcommand{\g}{\gamma}
\renewcommand{\a}{\alpha}
\renewcommand{\b}{\beta}
\renewcommand{\d}{\delta}

\renewcommand{\aa}{\boldsymbol{\alpha}}
\newcommand{\bb}{\boldsymbol{\beta}}
\renewcommand{\gg}{\boldsymbol{\gamma}}

\newcommand{\Z}{\mathbb{Z}} 
\newcommand{\Q}{\mathbb{Q}} 
\newcommand{\R}{\mathbb{R}} 
\newcommand{\N}{\mathbb{N}} 

\newcommand{\1}{\mathbf{1}}
\newcommand{\n}{\mathbf{n}}

\renewcommand{\L}{\mathcal{L}}

\DeclareFontEncoding{LS1}{}{}
\DeclareFontSubstitution{LS1}{stix}{b}{it}
\DeclareSymbolFont{symbols4}{LS1}{stixbb}{b}{it}
\DeclareMathSymbol{\csube}{\mathrel}{symbols4}{"9F}
\DeclareFontEncoding{LS1}{}{}
\DeclareFontSubstitution{LS1}{stix}{b}{it}
\DeclareSymbolFont{symbols4}{LS1}{stixbb}{b}{it}
\DeclareMathSymbol{\csub}{\mathrel}{symbols4}{"9D}

\title{Piecewise convex embeddability on linear orders}
\date{December 2, 2023}
\keywords{Linear orders, convex embeddability, piecewise convex embeddability, Borel reducibility}

\author[M.~Iannella]{Martina Iannella}
\address{Institute of Discrete Mathematics and Geometry, Vienna University of Technology, 1040 Vienna --- Austria}
\email{martina.iannella@tuwien.ac.at}

\author[A.~Marcone]{Alberto Marcone}
\address{Dipartimento di scienze matematiche, informatiche e fisiche, Universit\`a di Udine, Via delle Scienze 208, 33100 Udine --- Italy}
\email{alberto.marcone@uniud.it}

\author[L.~Motto Ros]{Luca Motto Ros} \address{Dipartimento di matematica \guillemotleft{Giuseppe Peano}\guillemotright, Universit\`a di Torino, Via Carlo Alberto 10, 10121 Torino --- Italy} \email{luca.mottoros@unito.it}

\author[V.~Weinstein]{Vadim Weinstein} 
\address{Center for Ubiquitous Computing, Erkki Koiso-Kanttilan katu 3, door E P.O Box 4500, 90014 University of Oulu --- Finland} 
\email{vadim.weinstein@iki.fi}

\subjclass[2020]{Primary: 03E15, 06A05, 06A07.}
\thanks{Iannella, Marcone, and Motto Ros were partially supported by the Italian PRIN 2017 Grant ``Mathematical Logic: models, sets, computability", prot. 2017NWTM8R. 
Marcone and Motto Ros were also supported by the Italian PRIN 2022 ``Models, sets and classifications'', prot. 2022TECZJA.
We thank Thilo Weinert for bringing \cite{Orr95}, which led to the proof of Theorem \ref{thm:ccs_implies_star}, to our attention.}

\begin{document}

\maketitle

\begin{abstract}
Given a nonempty set $\L$ of linear orders, we say that the linear order $L$ is $\L$-convex embeddable into the linear order $L'$ if it is possible to partition $L$ into convex sets indexed by some element of $\L$ which are isomorphic to convex subsets of $L'$ ordered in the same way.
This notion generalizes convex embeddability and (finite) piecewise convex embeddability (both studied in \cite{IMMRW22}), which are the special cases $\L = \{\1\}$ and \( \L = \Fin \).
We focus mainly on the behavior of these relations on the set of countable linear orders, first characterizing when they are transitive, and hence a quasi-order. 
We then study these quasi-orders from a combinatorial point of view, and analyze their complexity with respect to Borel reducibility.
Finally, we extend our analysis to uncountable linear orders.
\end{abstract}

\tableofcontents

\section{Introduction}

Given two linear orders $L$ and $L'$, we say that $L$ \textbf{embeds} into $L'$ (in symbols $L \preceq L'$) if there exists a function $f \colon L \to L'$ which is an embedding, i.e.\ it is such that $n \leq_L m \iff f(n) \leq_{L'} f(m)$ for every $n,m \in L$. 
The embeddability relation \( \preceq \) and its restriction \( \emb \) to the Polish space \( \LO \) of countable linear orders are very well-studied. 
The quasi-order \( \trianglelefteq \) of \textbf{convex embeddability} is obtained from \( \preceq \) by considering only embeddings with convex range, namely, $L \trianglelefteq L'$ if there exists an embedding $f$ from $L$ to $L'$ such that $f(L)$ is a convex subset of $L'$. 
Although this notion has been mentioned in older literature, going back at least to \v{C}ech's 1936 monograph (see the English version \cite{Ce69}), the first systematic study of $\trianglelefteq$ appears in \cite{IMMRW22}, which again mostly concentrates on its restriction \( \cvx \) to the space \( \LO \), and applies it to obtain anticlassification results for arcs and knots.

Pursuing the research started in \cite{IMMRW22}, here we prove a number of additional facts about \( \trianglelefteq \). Moreover,  we extend both the old and the new results on \( \trianglelefteq \) to a new 
family of binary relations \( \trianglelefteq^\L \) on linear orders, called \textbf{$\L$-convex embeddability}, which is parametrized by certain classes $\L$ of linear orders. 
Although most of our results work in a more general setup, for the sake of simplicity in this introduction we restrict the attention to classes \( \L \subseteq \Lin \), where $\Lin = \Fin \cup \LO$ is the standard Borel space of finite or countably infinite linear orders.
Special cases of $\L$-convex embeddability are embeddability for countable linear orders (when $\L = \Lin$) and convex embeddability (when $\L = \{\1\}$).
However, our definition includes many other natural situations, some of which are motivated by the counterpart of convex embeddability for circular orders%
\footnote{Indeed, all the results in the present paper could easily be transferred to the context of circular orders as well. However, since the results and techniques would essentially remain the same, for the sake of conciseness and clarity we decided to stick to the notationally simpler case of linear orders.}
called piecewise convex embeddability in \cite{IMMRW22}; in our setup,  that notion essentially corresponds to \( \trianglelefteq^\Fin\). 
We also extend our analysis to the restriction of \( \trianglelefteq^\L \) (and hence, as a special case, of \( \trianglelefteq \)) to \emph{uncountable} linear orders, complementing (and contrasting) the celebrated result by Justin Moore~\cite{Moo06} on the embeddability relation \(  \preceq \) on such orders.

Let us give more details and explain the main results of the paper.
Let \( \L \subseteq \Lin \) be nonempty and downward \( \preceq \)-closed.
Given linear orders \( L \) and \( L' \), we write \( L \trianglelefteq^\L L'\) if there is \( K \in \L \) such that \( L \) can be written as a \( K \)-sum \( L = \sum_{k \in K} L_k\) of convex pieces so that for some embedding \( f \colon L \to L' \) each \( f(L_k) \) is a convex subset of \( L' \).
Notice that 
$\trianglelefteq^\L$ coincides with \( \trianglelefteq \) if $\L=\{\1\}$, while if $\L=\Lin$ (and we restrict to countable linear orders) then $\trianglelefteq^\L$ coincides with \( \preceq \). 
However, it is not always the case that \( \trianglelefteq^\L \) is a quasi-order, as such relation might fail to be transitive for some families \( \L \). 
In Theorem \ref{thm:ccs_trans} we prove that \( \trianglelefteq^\L \) is transitive, and hence a quasi-order, exactly when $\L$ satisfies an additional combinatorial property that we call \textbf{ccs} (for \textbf{closed under convex sums}, see Definition \ref{def:appropriate}).
The ccs property is rather technical, but is satisfied by every nonempty class \( \L \subseteq \Lin \) which is closed under products and suborders.

Given a ccs family \( \L \), our primary goal is to study the combinatorial properties and the complexity with respect to Borel reducibility of the quasi-order \( \qo \) of \( \L\)-convex embeddability on \( \LO \), and of its induced equivalence relation \( \eq \). 

It turns out that for any ccs $\L \subsetneq \Lin$, the quasi-order $\qo$ shares with $\cvx$ most of the combinatorial properties that were established in \cite{IMMRW22} and which are quite different from those of $\emb$, i.e.\ of \( \trianglelefteq^\Lin_\LO \). 
More in detail, recall that a quasi-order \(\precsim\) on a set \(X\) is a well quasi-order (briefly: a wqo) if for each sequence \((x_n)_{n \in \o}\) of elements of \(X\), there exist \(n<m\) such that \(x_n \precsim x_m\), or equivalently, if \(\precsim\) has no infinite descending chain and no infinite antichain (\cite{Ros82}). 
In 1948 Fra\"{i}ss\'{e} conjectured that $\LO$ is well quasi-ordered by \( \emb \), and Laver in 1971 showed that this is indeed the case~\cite{La71}.
Moreover $\LO$ has a maximum with respect to the partial order induced by \(\preceq_{\LO}\) (namely, the equivalence class of non-scattered linear orders), and it has a basis of size \( 2 \) constituted by $\o$ and $\o^*$. 
In contrast, we obtain the following results when $\L \subsetneq \Lin$ is ccs:
\begin{itemizenew}  
\item 
In \( \qo \) there are chains of order type \( (\R, {<}) \) and antichains of size \( 2^{\aleph_0} \), thus \( \qo \) is far from being a wqo
(Theorem~\ref{thm:chains_qo}).
\item 
\(\LO\) does not have maximal elements with respect to \(\qo\), and the dominating number of \(\qo\) is \(2^{\aleph_0}\), the largest possible value (Theorem~\ref{no_max}).  
\item 
If \( \L \subseteq \Fin \) (which by ccs amounts to \( \L = \{ \1 \} \) or \( \L = \Fin \)), then there are \( 2^{\aleph_0} \)-many \( \qo\)-incomparable \( \qo \)-minimal elements in \( \LO \), and thus every basis for \( \qo \) has maximal size \( 2^{\aleph_0} \) (Theorem~\ref{thm:basis2}); if instead \( \Fin \subsetneq \L \), then \( \{ \omega, \omega^* \} \) is a basis for \( \qo \) (Theorem~\ref{thm:twoelementsbasis}), as it was the case for \( \emb \).
\item 
The unbounding number of \(\qo\) is $\aleph_1$ (Theorem~\ref{thm:unbounding_number}).
\item 
\( \qo \) has the fractal property with respect to its upper cones (Theorem~\ref{thm:fractal}).
\end{itemizenew} 
The first four results generalize those obtained for $\cvx$ in \cite{IMMRW22}, while the last one is new also in that case, and shows in particular that \( \cvx \) exhibits a high degree of complexity everywhere.

We next move to the analysis of the descriptive set-theoretic complexity of \( \qo \).
Recall that both $\emb$ and $\cvx$ are proper analytic quasi-orders.
In contrast, the complexity of the quasi-order $\qo$ as a subset of \( \LO \times \LO\) depends on the complexity of the class $\L$, and  although it is always \( \boldsymbol{\Sigma}^1_2 \) and non-Borel, it can 
fail to be analytic (Proposition~\ref{prop:complexityinthesquare} and Theorem~\ref{cor:complexityinthesquare}). 

We also analyze the complexity with respect to Borel reducibility \( \leq_B \) of the equivalence relation \( \eq \),
which is called \textbf{\( \L \)-convex biembeddability}.
When $\L = \Lin$, the relation $\eq$ is the biembeddability relation $\equiv_{\LO}$ on $\LO$. 
Since \( \emb \) is a wqo, it follows that $\equiv_{\LO}$ is an analytic equivalence relation with $\aleph_1$-many equivalence classes, and thus ${\id(X)} \nleq_B {\equiv_{\LO}}$ for any uncountable Polish space $X$ (see \cite[Lemma 3.17]{CamMarMot2018} for a proof). 
In particular, $\equiv_{\LO}$ is far from being complete for analytic equivalence relations. 
When $\L = \{\1\}$, instead, $\eq$ is the relation of convex biembeddability $\cvxeq$ on $\LO$ studied in \cite[Section 3.2]{IMMRW22}, where it is shown that \( {\iso} \leq_B {\cvxeq} \) and \( {\cvxeq} \leq_{\text{\scriptsize \textit{Baire}}} {\iso} \). Thus ${E_1} \nleq_{\text{\scriptsize \textit{Baire}}} {\cvxeq}$ and \( \cvxeq \) is not complete for analytic equivalence relations again. In both cases, from the non-completeness of \( \eq \) it follows that \( \qo \) is not complete for analytic quasi-orders either.
Here we show that when $\L$ is ccs and different from both $\Lin$ and $\{\1\}$, we are in a rather different scenario:
\begin{itemizenew}
    \item 
    ${E_1} \mathrel{\leq_B} {\eq}$ (Theorem \ref{thm:red_E1_eq}).
    \item 
    ${\iso} <_B {\eq}$, and in fact ${\eq} \nleq_{\text{\scriptsize \emph{Baire}}} {\iso}$ (Corollary \ref{cor:bor_compl_eq}).
\end{itemizenew}
We currently do not know if for some \( \{ \1 \} \subsetneq \L \subsetneq \Lin \) the relation \( \eq \) is complete for analytic equivalence relations.

If one considers the natural variant of embeddability \( \emb \) for \emph{coloured} linear orders, denoted by \( \preceq_{\LO_\N} \), then it is no longer a wqo and it actually has maximal complexity, namely, it is a complete analytic quasi-order (\cite{MR04}). 
In contrast, we show that for all other ccs $\L \subsetneq \Lin$ the quasi-order $\qo$ is Borel bi-reducible with its natural version \( \trianglelefteq^\L_{\LO_\N} \) for coloured linear orders (Theorem~\ref{thm:coloured}), hence no complexity is gained by adding colours. 
In particular, taking \( \L = \{ \1 \} \) this shows that, in contrast with \( \preceq_{\LO_\N} \), the convex embeddability relation \( \trianglelefteq_{\LO_\N} \) on countable coloured linear orders is not a complete analytic quasi-order because its induced equivalence relation is Baire bi-reducible with \( \iso \). 

Having shown the richness of the quasi-orders \( \qo \), it is natural to ask how many different ccs classes \( \L \subseteq \Lin \) are there. 
Besides the ones mentioned so far (namely: \( \{ \1 \} \), \( \Fin\) and \( \Lin \)) and the collections \( \WO \) of well-orders and \( \Scat \) of scattered linear orders, we show that there are many other natural examples (Propositions~\ref{prop:classes_zeta} and~\ref{prop:classes_gamma}) and we establish some Borel reductions among the various $\qo$ (Theorems \ref{thm:red_qog_qolg},~\ref{thm:red_threshold} and \ref{thm:red_fin_zeta}).

Finally, we observe that most of the combinatorial techniques developed to obtain the above results actually work for uncountable linear orders as well, sometimes with minor variations. 
For an infinite cardinal $\kappa$, we hence consider classes $\L$ included in the collection $\Lin_\kappa$ of linear orders of size at most $\kappa$, and study the restriction $\qok$ of $\trianglelefteq^\L$ to the set $\LO_\kappa$ of linear orders of size $\kappa$.
Working in the framework of generalized descriptive set theory when appropriate, we obtain the following results:
\begin{itemizenew}
\item 
It is consistent with \( \mathsf{ZFC} \) that for all infinite cardinals \( \kappa \) which are successors of a regular cardinal and for every class $\L \subsetneq \Lin_\kappa$ such that $\qok$ is a $\kappa$-analytic quasi-order, the relation \( \eqk \) of $\L$-convex biembeddability over \( \LO_\kappa \) is complete for \( \kappa \)-analytic equivalence relations (Theorem~\ref{thm:completenessforuncountableLO}).
\item 
If \( \kappa \) is an infinite cardinal and $\L$ is included in the set of linear orders of size smaller than the cofinality of $\kappa$, then there are continuum many incomparable minimal elements in $\LO_\kappa$ with respect to $\qok$ (Theorem~\ref{thm:basisforuncountable}).
\item 
For all infinite cardinals \( \kappa \) and all classes $\L \subseteq \Lin_\kappa$ containing only scattered orders and such that $\qok$ is transitive, the quasi-order $\qok$ contains chains isomorphic to any linear order of size $\kappa$, and it has the fractal property with respect to its upper cones (Theorems~\ref{thm:chains_qo_unctbl} and~\ref{thm:fractal_unctbl}).
\item 
For all infinite cardinals \( \kappa \) and classes $\L \subsetneq \Lin_\kappa$ such that $\qok$ is transitive, the quasi-order $\qok$ has dominating number $2^\kappa$ and unbounding number $\kappa^+$ (Theorems~\ref{no_max_unctbl} and~\ref{thm:unbounding_number_unctbl}).
\end{itemizenew}
The first result is in contrast with the situation for countable linear orders, while the second one contrasts 
the five-elements basis theorem for the embeddability relation over uncountable linear orders~\cite{Moo06}, as it implies that there is no finite or countable basis for the $\L$-convex embeddability relation on such class when \( \L \) consists of small enough linear orders.
This notably includes the case of convex embeddability.

Finally, since in the results on uncountable linear orders we often assume that \( \qok \) is transitive, it is natural to ask for a characterization of such a property. We show in Theorem~\ref{thm:ccs_trans_kappa} that if the infinite cardinal $\kappa$ satisfies a very mild condition, namely $\kappa^{<\kappa} = \kappa$, then  for any nonempty downward $\preceq$-closed $\L \subseteq \Lin_\kappa$, the transitivity of $\L$-convex embeddability over \( \LO_\kappa \) is again equivalent to the ccs property.


\section{Preliminaries}

When $L$ is a linear order we denote by $\leq_L$ the order coded by \(L\), by \( <_L \) its strict part, and by \(L\) its domain. 
We denote by \( \min L \) and \( \max L \) the minimum and maximum of \( L \), if they exist.
If \(L_0,L_1 \subseteq L\), we write \(L_0 \leq_L L_1\) (resp.\ \(L_0 <_L L_1\)) if and only if \(n \leq_L m\) (resp.\ \(n <_L m\)) for every \(n \in L_0\) and \(m \in L_1\). 
Notice that if \(L_0 \leq_L L_1\) then either \(L_0\) and \(L_1\) are disjoint and \( L_0 <_L L_1 \), or else \( \max L_0 \) and \( \min L_1 \) exist, they are equal, and the only element in \( L_0 \cap L_1 \) is \(\max L_0=\min L_1\).
	
We recall some isomorphism-invariant operations on the class of linear orders. 

\begin{itemizenew}
\item The reverse \( L^* \) of a linear order \( L \) is the linear order on   the domain of \( L \) defined by setting $x \leq_{L^*} y \iff y \leq_L x$.
\item If \( L \) and \( K \) are linear orders, their sum \( L+K \) is the linear order defined on the disjoint union of \( L \) and \( K \) by setting ${x} \leq_{L+K} {y}$ if and only if either $x \in L$ and $y\in K$, or $x, y \in L$ and $x \leq_{L} y$, or \( x,y \in K \) and \( x \leq_K y \).
\item In a similar way, given a linear order \(K\) and a sequence of linear orders \((L_k)_{k \in K}\) we can define the \(K\)-sum $\sum_{k \in K} L_k$ on the disjoint union of the \( L_k \)'s by setting \( x \leq_{\sum_{k \in K} L_k} y \) if and only if there are \( k <_K k' \) such that \( x \in L_k \) and \( y \in L_{k'} \), or \( x,y \in L_k \) for the same \( k \in K \) and \( x \leq_{L_k} y \). 
Formally, \( \sum_{k \in K} L_k \) is thus defined on the set \( \{ (x,k) \mid k \in K , x \in L_k \} \) 	by stipulating that \( (x,k) \leq_{\sum_{k \in K} L_k} (x',k') \) if and only if \( k <_K k' \) or else \( k = k' \) and \( x \leq_{L_k} x' \).
\item The product \(LK\) of two linear orders \( L \) and \( K \) is the cartesian product \( L \times K\) ordered antilexicographically. Equivalently, \( LK  = \sum_{k \in K} L \). 	
\end{itemizenew}

\begin{defn}
  A subset \( I \) of the domain of a linear order \( L \) is
  \textbf{(\( L \)-)convex} if \( x \leq_L y \leq_L z \) with
  \( x,z \in I \) implies \( y \in I \). In this case we write \(I \csube L\).
\end{defn}

If \(m,n \in L\), we adopt the notations $[m,n]_L$, $(m,n)_L$, $(-\infty,n]_L$, $(-\infty,n)_L$, $[n,+\infty)_L$, and $(n,+\infty)_L$ (suppressing the subscript $L$ when it is clear from the context) to indicate the obvious \( L \)-convex sets. 
Notice however that in general not all $L$-convex sets are of one of these forms.
	
\begin{defn}
Let $L$ and $L'$ be linear orders. We say that an embedding \(f\)
from $L$ to $L'$ is a \textbf{convex embedding} if $f(L)$ is an
\( L' \)-convex set. We write $L\trianglelefteq L'$ when such \(f\)
exists, and call \textbf{convex embeddability} the resulting binary relation.
\end{defn}

\begin{remark}\label{rem:cvx_Q}
If a linear order \( L \) is isomorphic to \( (\Q,{\leq})\) and \(L_0 \trianglelefteq L\), then \(L_0 \) is isomorphic to one of \( \1 \), \( (\Q,{\leq}) \), \( \1 + (\Q,{\leq}) \), \( (\Q,{\leq}) + \1 \), or \( \1 + (\Q,{\leq}) + \1\), where $\1$ is the linear order with domain $\{0\}$.
\end{remark}

When studying the combinatorial properties of our quasi-orders we use the following standard terminology.

\begin{defn}\label{def:combinatorics}
Let \(\precsim\) be a quasi-order on a set $X$. We say that \(\mathcal{F} \subseteq X\) is a \textbf{dominating family} (for \( \precsim \)) if for every \(L \in X\) there exists \(L' \in \mathcal{F}\) such that \(L \precsim L'\). 
The \textbf{dominating number} \( \mathfrak{d}(\precsim) \) of \( \precsim \) is the smallest size of a dominating family for \( \precsim \).
The \textbf{unbounding number} \( \mathfrak{b}(\precsim) \) is instead the smallest size of a subset of $X$ which is unbounded with respect to \( \precsim \).
Moreover, we say that \( \mathcal{B} \subseteq X \) is a \textbf{basis} for \( \precsim \) if for every \( L \in X \) there is \( L' \in \mathcal{B} \) such that \( L' \precsim L \).
\end{defn}
	
For descriptive set theory, we adopt the standard terminology and notations from~\cite{Kec95}. 
Each element $L$ of the Polish space $2^{\N\times\N}$ can be seen as a code of a binary relation on \(\N\), namely, the one relating \( n \) and \( m \) if and only if $L(n,m) = 1$. 
We denote by $\Lin$ the set of codes for nonempty linear orders defined either on a finite subset of $\N$ or on the whole $\N$, i.e. $\Lin = \Fin \cup \LO$ where
\[
  \Fin = \{L \in 2^{\N\times\N} \mid L \mbox{ codes a finite reflexive linear order}\}
\]
and
\[
  \LO = \{L \in 2^{\N\times\N} \mid L \mbox{ codes a reflexive linear order on } \N\}.
\]
It is easy to see that $\LO$ is a closed subset of the Polish space $2^{\N\times\N}$, thus it is a Polish space as well, while $\Lin$ is a standard Borel space because \( \Fin \) is $F_\sigma$. 
We also let \(\Scat\) and $\WO$ be the subsets of $\Lin$ consisting of scattered%
\footnote{Recall that a linear order is scattered if the rationals do not embed into it.}
linear orders and well-orders, respectively. They are not standard Borel spaces, as they are proper coanalytic subsets of \( \Lin \). 

All the operations on the class of linear orders defined at the beginning of this section can be construed as Borel maps from \( \Lin \) or \( \Lin \times \Lin \) to \( \Lin \), and their restrictions to \( \LO \) are continuous and have range contained in \( \LO \).

For every \(n \in \N\), we denote by \(\boldsymbol{n}\) the element of
\(\Fin\) with domain \(\{0,...,n-1\}\) ordered as usual. 
Similarly, for every infinite  ordinal \( \alpha < \omega_1 \) we fix a well-order \( \boldsymbol{\alpha} \in \LO \) with order type \( \alpha \).
We also fix isomorphic copies of \( (\N,{\leq}) \), \( (\Z,{\leq}) \) and \( (\Q,{\leq}) \) in \( \LO \), and denote them by \(\o\), \(\z\) and \(\eta\), respectively.

Generalizing the approach used above, given any cardinal \( \kappa \) we can form the space 
\[ 
\LO_\kappa = \{ L \in 2^{\kappa \times \kappa} \mid L \text{ codes a reflexive linear order on } \kappa \} 
\] 
of (codes for) linear orders on \( \kappa \), and set $\LO_{< \kappa} = \bigcup_{0< \lambda < \kappa} \LO_\lambda$ and \( \Lin_\kappa = \LO_\kappa \cup \LO_{< \kappa} \). 
In particular, \( \LO \), \( \Fin \) and \( \Lin \), are essentially the same of \( \LO_{\aleph_0} \), \( \LO_{<\aleph_0}\) and \( \Lin_{\aleph_0} \), respectively. Using a similar notation, we denote by \( \Scat_\kappa \) the subset of \( \Lin_\kappa \) consisting of scattered linear orders, and similarly for \( \WO_\kappa \). 
We also let \( \DLO_\kappa \) be the subset of \( \Lin_\kappa \) consisting of dense linear orders without first and last element. 
By Cantor's theorem, \( \DLO_{\aleph_0} \) contains just \( \eta \), up to isomorphism, but when \( \kappa > \aleph_0 \) the class \( \DLO_\kappa \) is much richer. 
Finally, we let \( \ALin = \bigcup \{ \Lin_\kappa \mid \kappa \text{ is a cardinal} \} \), and define \( \WO_\infty \), \( \AScat\) and \( \DLO_\infty \) analogously.
All the concepts and notations introduced above can be straightforwardly adapted to work within \( \Lin_\kappa \) or \( \ALin\).

Recall that we denote by $\emb$ the restriction of the quasi-order $\preceq$ of \textbf{embeddability} between linear orders to the space $\LO$, and that $\LO$ is well quasi-ordered under this relation and has a maximum, so that $\mathfrak{d}(\emb)=1$.

We denote by $\cvx$ the restriction to $\LO$ of the quasi-order $\trianglelefteq$ of convex embeddability between linear orders. 
The following proposition collects some combinatorial properties of $\cvx$ obtained in \cite{IMMRW22} that we need (and generalize) in the following sections.

\begin{prop}\label{prop:comb_prop_cvx}
\begin{enumerate-(a)}
    \item \label{open_int_cvx} 
    There is an embedding from \((\Int(\R),\subseteq)\) into \((\LO,\cvx)\), where \(\Int(\R)\) is the set of open intervals of \(\R\).
    
    \item \label{prop:WO_unbounded_cvx} \(\WO \cap \LO\) is a maximal \(\omega_1\)-chain without an upper bound in \(\LO\)
  with respect to \(\cvx\), and \( \mathfrak{b}(\cvx) = \aleph_1 \).
    
    \item \label{no_max_cvx}
    \(\LO\) does not have maximal elements with respect to \(\cvx\).
    
    \item \label{prop:dom_fam_cvx}
    Every dominating family with respect
    to \(\cvx\) has size \(2^{\aleph_0}\), i.e.\ \( \mathfrak{d}(\cvx) = 2^{\aleph_0}\).
    
    \item \label{prop:basisforcvx}
    There exist \(2^{\aleph_0}\)-many \( \cvx \)-incomparable \(\cvx\)-minimal elements in \( \LO \). In particular, if \( \mathcal{B} \) is 
    a basis with respect to \( \cvx \) then \( |\mathcal{B}| = 2^{\aleph_0} \).
\end{enumerate-(a)}
\end{prop}

In descriptive set theory, binary relations are compared using Borel reducibility, denoted by \( \leq_B \). 
We refer to \cite{Gao09} for definitions, notations, and basic results, and in particular we use $\sim_B$ for the equivalence relation induced by $\leq_B$.
We deal also with the following weaker notion. 
Let \( E \) and \( F \) be binary relations on the topological spaces \( X \) and \( Y \). We say that \(E\) is \textbf{Baire reducible} to \(F\), and we write \({E} \leq_{\text{\scriptsize\textit{Baire}}} {F}\), if there exists a Baire measurable map \(\varphi\colon  X \to Y\) reducing \(E\) to \(F\) (i.e.\ ${x_1}\mathrel{E}{x_2}\iff {\varphi(x_1)}\mathrel{F}{\varphi(x_2)}$, for all $x_1, x_2 \in X$).   

An important class of analytic equivalence relations consists of those induced by a Borel action of a Polish group on a standard Borel space. In this paper we often consider the isomorphism relation on $\LO$, denoted by \( \iso \), which is $S_{\infty}$-complete (i.e.\ complete for the class of equivalence relations induced by a Borel action of the infinite symmetric group \(S_\infty\) --- see \cite{FS89}) and proper analytic. 

Let \(E_1\) be the equivalence relation defined on \(\R^\N\) by \((x_n)_{n \in \N} \mathrel{E_1} (y_n)_{n \in \N}\) if and only if there exists $m$ such that $x_n = y_n$ for all $n \geq m$.
%
The following result of Shani about $E_1$ generalizes a classical
theorem by Kechris and Louveau \cite{KL97}. (The additional part follows from the fact that by~\cite[Theorem 8.38]{Kec95} every Baire measurable map between Polish spaces is continuous on a comeager set.)

\begin{thm}[{\cite[Theorem 4.8]{Sha21}}]\label{E1_orbit}
The restriction of \(E_1\) to any comeager subset of \(\R^\N\) is not Borel reducible to any orbit equivalence relation. Thus in particular ${E_1} \not\leq_{\text{\scriptsize \textit{Baire}}} {\iso}$.
\end{thm}

We finally recall some results about the complexity with respect to Borel reducibility of the equivalence relation $\cvxeq$ induced by $\cvx$ (see \cite[Corollaries 3.13 and 3.27]{IMMRW22}):

\begin{thm}\label{thm:complexity_cvxeq}
    \begin{enumerate-(a)}
        \item\label{thm:complexity_cvxeq-a} ${\iso} \mathrel{\leq_B} {\cvxeq}$.
        \item\label{thm:complexity_cvxeq-b} ${E_1} \mathrel{\nleq_{\text{\scriptsize \textit{Baire}}}} {\cvxeq}$.
    \end{enumerate-(a)}
\end{thm}

\section{The ccs property}\label{sec:ccs_property}
	
In this section we introduce a binary relation among linear orders which captures the idea of ``piecewise''
convex embeddability, where the pieces are orderly indexed by an element of a fixed class $\L \subseteq \ALin$. 
Unless otherwise stated, \emph{from now on we let \( \L \) be a nonempty downward \( \preceq \)-closed subset of \( \ALin \)}. 
Among such classes we find those of the form \(\L_{\preceq L_0} = \{L \in \ALin \mid L \preceq L_0\}\) and \(\L_{\prec L_0} = \{L \in \ALin \mid L \prec L_0\}\), for some \( L_0 \in \ALin \).

    \begin{defn}
        Given linear orders $K$ and $L$, a \textbf{$K$-convex partition of $L$} is a partition $(L_k)_{k \in K}$ of $L$ such that $k <_K k'$ if and only if $L_k <_L L_{k'}$ for every $k,k' \in K$.
    \end{defn}

Notice that if $(L_k)_{k \in K}$ is a $K$-convex partition of $L$, each $L_k$ is a convex subset of $L$. 
	
\begin{defn}\label{def:QO}
Given $\L \subseteq \ALin$ as above and linear orders \( L,L' \), we write $L \trianglelefteq^\L L'$ if and only if there exist $K\in \L$, a $K$-convex partition $(L_k)_{k \in K}$ of $L$, and an embedding $f$ of $L$ into $L'$ such that $f(L_k)\csube L'$ for all $k \in K$. 
The binary relation \( \trianglelefteq^\L \) is called \textbf{\( \L \)-convex embeddability}. 
\end{defn}

Equivalently, \( L \trianglelefteq^\L L' \) if and only if there exist \( K \in \L\) and a family \( (L_k)_{k \in K} \) of nonempty linear orders such that, up to isomorphism, \( L = \sum_{k \in K} L_k \) and there is an embedding \( f \colon L \to L' \) such that \( f(L_k) \csube L' \) for all \( k \in K\). 
Yet another equivalent reformulation of \( L \trianglelefteq^\L L'\) is the following: there are \( K \in \L \), \( K' \in \ALin \), an embedding \( f \colon K \to K' \), a \( K \)-convex partition \( (L_k)_{k \in K} \) of \( L \), and a \( K' \)-convex partition \( (L'_k)_{k \in K'} \) of \( L' \) such that \( L_k \cong L'_{f(k)}\) for all \( k \in K \).

Although in general \( \trianglelefteq^\L \) needs not to be a quasi-order, we also consider its ``strict part'' \( \triangleleft^\L \) defined by \( L \triangleleft^\L L' \) if \( L \trianglelefteq^\L L' \) but \( L' \not\trianglelefteq^\L L \), and write \( L \mathrel{\underline{\bowtie}}^{\mathcal{L}} L' \) if both \( L \trianglelefteq^\L L'\) and \( L' \trianglelefteq^\L L\).
As usual, we denote by \(\qo\) the restriction of \( \trianglelefteq^\L\) to the set \(\LO\) of (codes for) linear orders on the whole \(\N\), and similarly for \( \triangleleft^\L_\LO \) and \( \underline{\bowtie}^\L_\LO \).

If \( \L = \{ \1 \} = \L_{\preceq \1} \), then \( \trianglelefteq^\L \) is simply the relation of convex embeddability \( \trianglelefteq \).
Moreover, if \( \L \subseteq \L' \) then \( {L \trianglelefteq^\L L'} \Rightarrow {L \trianglelefteq^{\L'}} L'\) for all linear orders \( L,L'\).
Since each \( \L \) is tacitly assumed to be nonempty and downward \( \preceq \)-closed, it follows that \( \{ \1 \} \subseteq \L \), and hence ${L \trianglelefteq L'} \Rightarrow {L \trianglelefteq^{\L} L'}$. 

When \( \L \subseteq \Lin \), at the other extreme we have the case \( \L = \Lin = \L_{\preceq \eta} \) (equivalently: \( \L \nsubseteq \Scat\)). 
In this case, if \( L \) is countable and \( L' \) is an arbitrary linear order, then \( {L \preceq L'} \Rightarrow {L \trianglelefteq^\L L'} \), as we can always partition $L$ in singletons. 
More generally, by the same reasoning we have the following useful fact, which applies to arbitrary families \( \L \).

\begin{fact} \label{fct:basic}
If $L \in \L$ and \( L' \) is arbitrary, then $L \trianglelefteq^{\L} L'$ if and only if $L \preceq L'$.
\end{fact}

Another useful fact is the following:

\begin{prop} \label{prop:crucial}
For every \( \L \subseteq \Lin \) and \( L \in \AScat\), we have \( L \trianglelefteq^{\L} \eta\) if and only if \( L \in \L \). 
\end{prop}
\begin{proof}
Assume that \( K \in \L \), \( (L_k)_{k \in K} \), and \( f \colon L \to \eta \) witness \( L \trianglelefteq^\L \eta \). 
By Remark~\ref{rem:cvx_Q} and \( L \in \Scat \), each \( L_k \) is a singleton and hence \( L \iso K \in \L \). 
The other direction is immediate by Fact \ref{fct:basic} and countability of $L \in \L \subseteq \Lin$.
\end{proof}

Combining the above observations, one can determine the mutual relationships among the relations \( \trianglelefteq^\L \). More precisely, say that \( \trianglelefteq^\L \) \textbf{refines} \( \trianglelefteq^{\L'} \) if \( {\trianglelefteq^\L} \subseteq {\trianglelefteq^{\L'}} \), i.e.\ \( {L \trianglelefteq^\L L'} \Rightarrow {L \trianglelefteq^{\L'} L'}\) for all linear orders \( L \) and \( L' \).%
\footnote{The result would not change if one restricts this definition to \emph{countable} linear orders.}

\begin{prop}\label{prop:refine}
Let \( \L, \L' \subseteq \Lin \). 
Then the relation \( \trianglelefteq^\L \) refines \( \trianglelefteq^{\L'} \) if and only if \( \L \subseteq \L' \).
\end{prop}

\begin{proof}
As observed, one direction is obvious, so let us assume that \( \trianglelefteq^\L \) refines \( \trianglelefteq^{\L'} \). 
Proposition~\ref{prop:crucial} implies that \( \L \cap \Scat \subseteq \L' \cap \Scat \), so we only need to show that if \( \L = \Lin \) then \( \L' = \Lin \) too. 
But if \( \eta \in \L \) then \( \eta \trianglelefteq^\L \mathbf{2} \eta \) by Fact~\ref{fct:basic}, which by our initial assumption implies \( \eta \trianglelefteq^{\L'} \mathbf{2} \eta \). 
Assume towards a contradiction that \( \L' \neq \Lin \), i.e.\ \( \L' \subseteq \Scat \). 
Let \( K \in \L' \), \( (L'_k)_{k \in K} \) and \( f \colon \eta \to \mathbf{2} \eta \) witness \( \eta \trianglelefteq^{\L'} \mathbf{2} \eta \).
Since \( K \in \Scat \), at least one of the convex sets \( L'_k \) contains a copy of \( \eta \) by Remark~\ref{rem:cvx_Q}, hence \( \eta \trianglelefteq \mathbf{2} \eta \), which is not the case.
\end{proof}

Since \( \trianglelefteq \) refines \( \trianglelefteq^\L\) for all the families \( \L \) under consideration, it easily follows that the relation \( \trianglelefteq^\L \) is always reflexive. However, the next example shows that \( \trianglelefteq^\L\) might fail to be transitive.
		
\begin{exmp}\label{qo_no_trans}
Consider $\L = \L_{\preceq \mathbf{2}}$. It is immediate that $\z \mathbf{3} \trianglelefteq^{\L} \z +\1 + \z \mathbf{2} \trianglelefteq^{\L} (\z +\1) \mathbf{3}$, but $\z \mathbf{3} \ntrianglelefteq^{\L} (\z +\1) \mathbf{3}$ because to find an embedding as in Definition \ref{def:QO} we need to have a linear order $K\in \L $ with three elements, which is not the case. 
More generally, if $\L = \L_{\preceq \n}$ with $n>1$, we have that $\z (\mathbf{2n-1}) \trianglelefteq^{\L} (\z +\1) (\mathbf{n-1}) + \z \n \trianglelefteq^{\L} (\z +\1) (\mathbf{2n-1})$, but $\z (\mathbf{2n-1}) \ntrianglelefteq^{\L} (\z +\1) (\mathbf{2n-1})$. Hence transitivity fails for all binary relations \( \trianglelefteq^{\L_{\preceq \mathbf{n}}}\) with \( n > 1 \).
\end{exmp}

Since we want to work with quasi-orders, we thus have to first determine when \( \trianglelefteq^\L \) is transitive.
Consider linear orders \(L, L', L''\) such that \(L \trianglelefteq^{\L} L'\) with witnesses \(K \in \L\), \((L_k)_{k \in K}\) and \(f \colon L \to L'\), and \(L' \trianglelefteq^{\L} L''\)  with witnesses \(K' \in \L\), \((L'_{k'})_{k' \in K'}\) and \(f'\colon L' \to L''\). We would like to have that \(L \trianglelefteq^{\L} L''\). To this aim, for every \(k \in K\) define the set 
\[
K'_k=\{k' \in K' \mid f(L_k) \cap L'_{k'} \neq \emptyset\}.
\] 
Notice that each \(K'_k\) is a nonempty convex subset of \(K'\), and that \(\forall k_0,k_1 \in K\ (k_0 <_K k_1 \Rightarrow K'_{k_0} \leq_{K'} K'_{k_1})\) because $f(L_k) \csube L'$ for each \( k \in K \) by choice of \( f \). 
Now, consider the linear order
\[
M = \sum_{k \in K} K'_k , 
\]
i.e.\ $M$ is the set $\{(k',k) \mid k \in K \text{ and } k' \in K'_k\}$ ordered antilexicographically. For every \((k',k) \in M\), let 
\[
L_{(k',k)}=\{n \in L \mid n \in L_k \text{ and } f(n) \in L'_{k'}\}.
\]
Notice that $L_{(k',k)}$ is a nonempty convex subset of \( L_k\), and hence of \(L\), and that \( f(L_{(k,k')}) \csube L'_{k'} \), hence $(f'\circ f)(L_{(k',k)}) \csube L''$. Thus, \emph{if \(M \) were a member of  \(\L\)}, then $M$, \(\big(L_{(k',k)}\big)_{(k',k) \in M}\) and $f' \circ f$ would witness \(L \trianglelefteq^{\L} L''\). 
This motivates the following definition.
	
\begin{defn}\label{def:appropriate}
Let $\L$ be \( \preceq \)-downward closed.
We say that $\L$ is \textbf{closed under convex sums}, or \textbf{ccs} for short, if for every $K, K' \in \L$ and for every \((K'_k)_{k \in K}\) such that each \(K'_k\) is a nonempty convex subset of $K'$ and
\[
\forall k_0,k_1 \in K\ (k_0 <_K k_1 \Rightarrow K'_{k_0} \leq_{K'} K'_{k_1}),
\]
we have that \(\sum_{k \in K} K'_k \in \L\).
\end{defn}

Many natural classes are ccs, for example: \(\{ \1 \}\), $\Fin$, $\WO$, $\Scat$, and $\Lin$. 
Moreover, it is immediate to see that if $\L$ is ccs then so is $\L^*= \{L^* \mid L \in \L\}$, and since $\sum_{k \in K} K'_k$ is a suborder of $K' K$ then any downward \( \preceq \)-closed $\L$ which is closed under products is ccs.
In Section~\ref{sec:examples} we however exhibit examples of ccs classes that are not closed under products.
On the other hand, notice  that the ccs property does not hold for all $\L$ which are downward \( \preceq \)-closed.
Indeed, a crucial property of the convex sums involved in Definition~\ref{def:appropriate} is that if \(K'_{k_0} \cap K'_{k_1} =\{k'\}\)  for some distinct \(k_0,k_1 \in K\), then \(k'\) ``appears'' at least twice in \(\sum_{k \in K} K'_k\) and the latter is not necessarily isomorphic to a suborder of \(K'\). 
This observation allows us to show that the classes considered in Example~\ref{qo_no_trans} are not ccs, and hence there is no ccs class between \( \{ \1 \} \) and \( \Fin\).
	

\begin{exmp}\label{ex:fin_wo_scat}
Every class $\L_{\preceq \n}$ with $n >1$ is not ccs. Indeed, it is enough to consider $K= \mathbf{2}$ and $K'= \n$ and define $K'_0=\n$ and $K'_1=\{n-1\}$ to obtain that $\sum_{k \in K} K'_k = K'_0+K'_1 \cong \mathbf{n+1}$ does not belong to $\L_{\preceq \n}$. 
\end{exmp}

We now show that when \( \L \subseteq \Lin \), the ccs property is not only sufficient to obtain the transitivity of \( \trianglelefteq^\L \), but it is also necessary, and thus characterizes those \( \L \subseteq \Lin\) for which \( \trianglelefteq^\L \) is a quasi-order.
	
\begin{thm}\label{thm:ccs_trans}
Let $\L \subseteq \Lin$ be nonempty and downward \( \preceq \)-closed. Then the following are equivalent:
\begin{enumerate-(i)}
    \item\label{thm:ccs_trans-1} $\L$ is ccs;
    \item\label{thm:ccs_trans-2} $\trianglelefteq^{\L}$ is transitive;
    \item\label{thm:ccs_trans-3} $\qo$ is transitive.
\end{enumerate-(i)}
\end{thm}
	
\begin{proof}
We already showed that \ref{thm:ccs_trans-1} \( \Rightarrow \) \ref{thm:ccs_trans-2} in the discussion preceding Definition~\ref{def:appropriate}, while \ref{thm:ccs_trans-2} \( \Rightarrow \) \ref{thm:ccs_trans-3} is obvious, so let us prove \ref{thm:ccs_trans-3} $\Rightarrow$ \ref{thm:ccs_trans-1}.    
If \( \L = \Lin \) or \( \L = \{ \1 \} \) then \( \L \) is trivially ccs, while if \( \L \neq \{ \1 \} \) but  $\L$ does not contain all finite linear orders, then $\L = \L_{\preceq \n}$ for some $n > 1$, and hence \( \qo \) is not transitive by Example~\ref{qo_no_trans}. 
We can thus assume without loss of generality that $\L$ is such that \(\Fin \subseteq \L \subseteq \Scat\). 

Suppose that \( \qo \) is transitive:
given \(K,K' \in \L\) and \((K'_k)_{k \in K}\) such that \( \emptyset \neq K'_k \csube K'\) and \(\forall k_0,k_1 \in K\ (k_0 <_K k_1 \Rightarrow K'_{k_0} \leq_{K'} K'_{k_1})\), we want to show that $L = \sum_{k \in K} K'_k \in \L$.
If $L$ is finite then it belongs to $\L$ because \( \Fin \subseteq \L \), hence we can further assume that $L$ is infinite, i.e.\  \( L \in \LO \). 
Let \( L' = \sum_{k \in K} (K'_k+Q_k) \), where
\[  
Q_k = 
\begin{cases}
\emptyset & \text{if } K'_k \cap K'_j = \emptyset \text{ for all } j >_K k; \\
\eta & \text{otherwise}.
\end{cases}
\]
Then \( L' \in \LO \) as well, and we claim that \( L' \qo \eta \). 
To see this, let \( K'' = \bigcup_{k \in K} K'_k \subseteq K' \) (notice that in general this is \textbf{not} a disjoint union), so that \( K'' \in \L \) because \( \L \) is downward \( \preceq \)-closed. 
For each \( k' \in K'' \) let \( A_{k'} = \{ k \in K \mid k' \in K'_k \} \) and let \( L'_{k'} \) be the smallest \( L' \)-convex subset containing \( \{ (k',k) \mid k \in A_{k'} \} \), that is, the set of all \( \ell \in L' \) such that \( a_0 \leq_{L'} \ell \leq_{L'} a_1 \) for some (possibly equal) \( a_0,a_1 \in \{ (k',k) \mid k \in A_{k'} \} \). 
Then \( L'_{k'} \csube L' \) is isomorphic to \( \1 \) (if \( A_{k'} \) is a singleton), or to one of \( \eta \), \( \1 + \eta \), \( \eta+\1\), \( \1+\eta+\1 \) (if \( A_{k'} \) is not a singleton, the four cases depending on whether \( A_{k'} \) has a minimum or a maximum). 
It is easy to verify that \( (L'_{k'})_{k' \in K''} \) is a \( K'' \)-convex partition of \( L' \), and since \( L'_{k'} \trianglelefteq \eta \) because of its order type (Remark~\ref{rem:cvx_Q}), it is easy to recursively construct an embedding \( f \colon L' \to \eta \) which, together with \( K'' \in \L \) and \( (L'_{k'})_{k' \in K''}\), witnesses \( L' \qo \eta \).

Clearly \( L \qo L' \), as witnessed by \( K \in \L \) and \( (K'_k)_{k \in K} \) themselves, hence by transitivity of \( \qo \) we get \( L \qo \eta \). 
But \( L \in \Scat \) because it is a scattered sum of scattered linear orders (see \cite[Proposition 2.17]{Ros82}), thus \( L \in \L \) by Proposition~\ref{prop:crucial}.
\end{proof}

We will need the following property of ccs classes \( \L \subseteq \Lin \)
to obtain some combinatorial properties of \(\trianglelefteq^\L_{\LO}\) (Section~\ref{sec:comb}) and some lower bounds for its Borel complexity (Section~\ref{sec:borel_compl}).

\begin{thm}\label{thm:ccs_implies_star}
Let $\L \subseteq \Lin$ be ccs. If \( K \in \L \cap \LO \) and \( L_k \in \Fin \) for all \( k \in K \),
then \( \sum_{k \in K} L_k \in \L \).
\end{thm}

\begin{proof}
The statement is vacuously true if \( \L = \{ \1 \} \) or \( \L = \Fin \), hence we can assume that \( \Fin \subsetneq \L\).
Let \(K \in \L \cap \LO\) and \((L_k)_{k \in K}\) be as in the statement.
%
Applying~\cite[Proposition 2]{Orr95}, we have that \(K \iso \sum_{k \in K} L'_k\), where \(L'_k\) is nonempty for all \(k \in K\), and \(|L'_k|\geq |L_k|\) for all but finitely many \(k\). 
By the downward \( \preceq \)-closure of \(\L\) there exists \(A \subseteq K\) finite such that \(\sum_{k \in K} L''_k \in \L\), where \(L''_k=\1\) if \(k \in A\) and \(L''_k=L_k\) otherwise.
If $A = \emptyset$ then \(\sum_{k \in K} L_k = \sum_{k \in K} L''_k \in \L\) and we are done.

If $A \neq \emptyset$, let \(a_0\), \dots, \(a_n\) be the elements of \(A\) listed increasingly with respect to \( \leq_K \), and let $b_i = |L_{a_i}| \geq 1$.
Let \(K'=\big( \sum_{i=0}^n b_i \big) + \1\) so that \(K' \in \Fin \subseteq \L\). 
Consider the following convex subsets of \(\sum_{k \in K} L''_k\), indexed by \(K'\):
\begin{align*}
N_{(0,0)} & = \sum_{k \leq_K a_0} L''_k\\ 
N_{(0,i+1)} & = \sum_{a_{i}<_K k \leq_K a_{i+1}} L''_k \quad  \text{ for all } i <n,\\
N_{(j,i)} & = L''_{a_i} = \1 \qquad \text{ for all } i \leq n \text{ and } 0<j< b_i,\\
N_{\max K'} & = \sum_{a_n <_K k} L''_k.
\end{align*}
The set \( N_{\max K'}\) might be empty, so consider \( K'' = \{ k' \in K' \mid N_{k'} \neq \emptyset \} \in \L \).
Applying the ccs property to \(K''\) and the subsets \((N_{k''})_{k'' \in K''}\) of \(\sum_{k \in K} L''_k\), we obtain that \(\sum_{k'' \in K''} N_{k''} \in \L\).
Since \(\sum_{k'' \in K''} N_{k''} \iso \sum_{k \in K} L_k \), we again have \(\sum_{k \in K} L_k \in \L\).
\end{proof}

\section{Some technical lemmas} \label{sec:technical}

In this section, we prove some technical results that will be useful later on. 
Although we will mostly apply them when \( \L \) is ccs, we prove them in full generality.
A subset \( A \subseteq M \) of a linear order \( M \) is \textbf{inherently cofinal} if for every embedding \( f \colon A \to M \)  the image of \( f(A) \) is cofinal in \( M \). 
Notice that if \( M \) is either \( \z \) or a well-order of length an infinite cardinal \( \kappa \), then every tail \( [m_0,+\infty)_M\) of \( M \) is inherently cofinal. 
The following proposition was already noticed in~\cite{IMMRW22} for the special case \( M = \z \) and \( \L = \{ \1 \} \).

\begin{prop}\label{prop:fin_zeta_l}
Suppose that the linear order \( M \) has an inherently cofinal tail \( [m_0,+\infty)_M \). Then for every downward \( \preceq \)-closed \( \L \subseteq \ALin \) and all linear orders \( L\) and \( L' \) we have $M L \trianglelefteq^\L M L'$ if and only if $L \trianglelefteq^\L L'$.
\end{prop}

\begin{proof}
For the nontrivial direction, suppose that $M L \trianglelefteq^\L M L'$ as witnessed by $K \in \L$, the $K$-convex partition $(L_k)_{k \in K}$ of $M L$ and $f\colon M L \to M L'$. 
For every $k \in K$, let $\Tilde{L}_k = \{\ell \in L \mid (m_0,\ell) \in L_k\}$. Let also $\Tilde{K} = \{k \in K \mid \Tilde{L}_k \neq \emptyset\}$, so that $\Tilde{K} \in \L$ because the latter is downward \( \preceq \)-closed. 
Define the map $g\colon L \to L'$ by setting $g(\ell)=\ell'$ if and only if $\ell' \in L'$ is such that  $f(m_0,\ell) \in M \times \{\ell'\}$. 
We claim that $\Tilde{K}$, $(\Tilde{L}_k)_{k \in \Tilde{K}}$ and $g \colon L \to L'$ witness $L \trianglelefteq^\L L'$.

It is easy to see that \( (\Tilde{L}_k)_{k \in \Tilde{K}} \) is a \( \Tilde{K}\)-convex partition of \( L \), and that \( g \) is order-preserving since \( f \) was. 
To see that \( g \) is also injective, consider any $\ell_0,\ell_1 \in L$ with $\ell_0<_L \ell_1$. If \( g(\ell_0) = g(\ell_1)  \), then \( f([m_0,+\infty)_M \times \{ \ell_0 \}) \) would be a non-cofinal subset of \( M \times \{ g(\ell_0) \} \) (as witnessed by \( f(m_0,\ell_1) \)), contradicting the fact that \( [m_0,+\infty)_M \) was inherently cofinal in \( M \). 
This shows that \( g \) is an embedding. It remains to show that $g(\Tilde{L}_k) \csube L'$ for all \( k \in \Tilde{K} \). Fix $\ell_0,\ell_1 \in \Tilde{L}_k$ and $\ell' \in L'$ such that $g(\ell_0) <_{L'} \ell' <_{L'} g(\ell_1)$: our goal is to show that \( \ell' = g(\ell) \) for some \( \ell \in \Tilde{L}_k\).
Since $f(L_k) \csube M L'$, there is \( \ell \in [ \ell_0 , \ell_1]_L \csube \Tilde{L}_k \) such that \( f^{-1}(m_0,\ell') \in M \times \{ \ell \}\): we claim that \( g(\ell) = \ell' \). 
Suppose towards a contradiction that \( \ell' <_{L'} g(\ell) \), which together with \( \ell_0 \leq_L \ell \leq_L \ell_1 \) implies \( (m_0,\ell_0) \leq_{ML} f^{-1}(m_0,\ell') <_{ML} (m_0, \ell) \leq_{ML} (m_0,\ell_1) \). Since \( [ (m_0,\ell_0) , (m_0,\ell_1)]_L \csube L_k \) and \( f(L_k) \csube M L' \), we get that \( f^{-1} \restriction ([m_0,+\infty)_M \times \{ \ell' \})\) is a well-defined embedding of \( [m_0,+\infty)_M \times \{ \ell' \} \) into \( M \times \{ \ell \} \) with a non-cofinal range (as witnessed by \( (m_0,\ell) \)), against the fact that \( [m_0,+\infty)_M \) was inherently cofinal.
The case \( g(\ell) <_{L'} \ell' \) is symmetric: in this case the range of the embedding obtained by restricting \( f^{-1} \) to \( [m_0,+\infty)_M \times \{ g(\ell) \}\) would not be cofinal in \( M \times \{ \ell \} \) (as witnessed by \( f^{-1}(m_0,\ell') \)), a contradiction. 
Therefore we must conclude that \( g(\ell) = \ell' \), as desired.
\end{proof}

The next result plays a crucial role in transferring some of the properties of $\cvx$ uncovered in~\cite{IMMRW22} to the more general context of an arbitrary $\qo$.

\begin{prop}\label{prop:nqo}
Let $\L \subseteq \ALin$ be downward \( \preceq \)-closed, and let \(L \), \( L' \) and \( M \) be linear orders with $M \notin \L$. If \(LM \trianglelefteq^\L L'\) then \(L \trianglelefteq L'\).
\end{prop}

 \begin{proof}
Suppose that \( K \in \L \), \( (L_k)_{k \in K} \) and \( f \colon LM \to L' \) witness \(LM \trianglelefteq^\L L'\). 
For each \( m \in M \) set \( K_m = \{ k \in K \mid L_k \cap (L \times \{ m \}) \neq \emptyset \} \). If one of the sets \( K_m \) is a singleton \( \{ k \} \), then \( L \times \{ m \} \subseteq L_k \), hence $L \cong L \times \{m\} \csube L_k \trianglelefteq L'$ and we are done.
Otherwise, each \( K_m \) has at least two elements. 
In particular, this entails that $K_{m_0} \leq_{K} K_{m_1}$ if and only if $m_0 <_M m_1$. 
Now define $g \colon M \to K$ by letting $g(m)$ be an element of $K_m$ distinct from its maximum (if the latter exists). 
It is easy to see that $g$ is an embedding, which is against the hypothesis \( M \notin \L \) because \( K \in \L \).
\end{proof}

\begin{cor}\label{cor:eta}
Let $\L \subseteq \AScat$ be downward \( \preceq \)-closed. For all linear orders \(L \) and \( L' \), we have that \(L \eta \trianglelefteq^\L L'\) if and only if \(L \eta \trianglelefteq L'\).
\end{cor}

\begin{proof}
Since $\eta \eta \cong \eta$, if \(L \eta \trianglelefteq^\L L'\) then also $(L \eta) \eta \trianglelefteq^\L L'$, hence \(L \eta \trianglelefteq L'\) by Proposition~\ref{prop:nqo}. The other direction is trivial.
\end{proof}
 
The following proposition will be used to generalize part~\ref{thm:complexity_cvxeq-a} of  Theorem~\ref{thm:complexity_cvxeq}.

\begin{prop}\label{thm:cong_qo}
    Let $\L \subsetneq \ALin$ be downward $\preceq$-closed, and let $M \notin \L$. 
    The map $\varphi$ sending each linear order $L$ to $\varphi(L)=(\1+\z L + \1) M$ is such that ${L \cong L'} \iff {\varphi(L) \trianglelefteq^\L \varphi(L')} \iff {\varphi(L) \mathrel{\underline{\bowtie}}^{\L}  \varphi(L')}$.
\end{prop}

\begin{proof} 
Obviously, if $L \cong L'$ then $\varphi(L) \cong \varphi(L')$ and, a fortiori, $\varphi(L) \mathrel{\underline{\bowtie}^{\mathcal{L}}} \varphi(L')$, which in turn implies \( \varphi(L) \trianglelefteq^\L \varphi(L') \). 
So it remains to show that if \( \varphi(L) \trianglelefteq^\L \varphi(L') \), then \( L \cong L'\). Since $M \notin \L$, by Proposition~\ref{prop:nqo} we obtain from \( \varphi(L) \trianglelefteq^\L \varphi(L') \) that $\1+\z L + \1 \trianglelefteq \varphi(L')$. 
Let \( g \) witness this. 
Since the $\mathbf{1}$'s are the only elements that do not have immediate predecessor and successor both in $\mathbf{1}+\z L+\mathbf{1}$ and in $\varphi(L')$, we have that the two $\mathbf{1}$'s in $\mathbf{1}+\z L+\mathbf{1}$ are mapped by \( g \) into the two \( \1 \)'s of \( (\1 + \z L' + \1) \times \{ m \} \) for some \( m \in M\), hence \( g( \1 + \z L + \1) = (\1 + \z L' + \1) \times \{ m \} \) and $\mathbf{1}+\z L+\mathbf{1} \cong \mathbf{1}+\z L'+\mathbf{1}$. But then $\z L \cong \z L'$, and so $L \cong L'$ by \cite[Lemma 2.14]{IMMRW22}.
\end{proof} 

Let \( D \in \ADLO \) and \( f \colon D \to \ALin \).
As in~\cite{IMMRW22}, the following standard construction of linear orders \( D^f \) in which one replaces each $d \in D$ with the linear order $f(d)$ plays a central role, especially when \( D = \eta \).
Notice that each $D^f$ is not scattered, and if \( D = \eta \) then \( \eta^f \) contains a copy of $\Q$ which is both coinitial and cofinal in it. 
Actually, it follows from a classic result of Hausdorff (see e.g.\ \cite[Theorem 4.9]{Ros82}) that every \emph{countable} linear order which has no scattered initial and final segments is of the form $\eta^f$ for some $f \colon \Q \to \Scat$.


\begin{defn}\label{eta_f}
Let \( D \in \ADLO \).
Given a map \(f \colon D \rightarrow \ALin\), let $D^f$ be the linear order $\sum_{d \in D} f(d)$, i.e.\ the set \(\{(\ell,d) \mid d \in D \text{ and } \ell \in f(d)\}\) ordered antilexicographically. 
When $A \csube D$ we let $D^f_A  = \sum_{d \in A} f(d)$ be the restriction of $D^f$ to \(\{(\ell,d) \in D^f \mid d \in A \}\).
\end{defn}



To simplify the notation, when \( A = (d_0,d_1)_D \) we write \( D^f_{(d_0,d_1)} \) instead of \( D^f_{(d_0,d_1)_D} \). Also, when \( D = \eta \) and $I \csube \R$, with a minor abuse of notation we write $\eta^f_I$ in place of $\eta^f_{I \cap \Q}$.
A crucial property of linear orders of the form \( D^f_A \) is that if \( L \csube D^f_A \) is such that its projection on the second coordinate has more than one element, then \( L \) is not scattered.

\begin{lem}\label{lem:loc_not_scat}
Let \( D, D_0, D_1 \in \ADLO \).

\begin{enumerate}[label={\upshape (\alph*)}, leftmargin=2pc]
\item \label{lem:loc_not_scat-a}
Let \( f \colon D \to \ALin \), \(K \in \AScat\), and let \((L_k)_{k \in K}\) be a $K$-convex partition of \(D^f\). Then there exist \(k \in K\) and \( d_0,d_1 \in D \) with \(d_0 <_D d_1\) such that $D^f_{(d_0,d_1)} \csube L_k$. 

\item \label{lem:loc_not_scat-b}
For \( i \in \{ 0,1 \} \),
let \( f_i \colon D_i \to \AScat \). Let \( h \colon D_0^{f_0} \to D_1^{f_1} \) witness \( D_0^{f_0} \trianglelefteq D_1^{f_1} \). Then there is \( A \csube D_1 \) and an order-preserving bijection \( g \colon D_0 \to A \) such that \( h(f_0(d) \times \{ d\} ) = f_1(g(d)) \times \{ g(d) \} \) for all \( d \in D_0 \). In particular, for every \( d \in D_0 \) there is \( d' \in A \subseteq D_1 \) such that \( f_0(d) \cong f_1(d') \).

\item \label{lem:loc_not_scat-c} 
Let \( \L \subseteq \AScat \), and for \( i \in \{ 0,1 \} \) let \( f_i \colon D_i \to \AScat \) be such that \( f_i^{-1}(L)\) is dense in \( D_i \) for every \( L \in f(D_i) \). 
If \( {D_0^{f_0} \trianglelefteq^\L D_1^{f_1}} \), then \( f_0(D_0) \) and \( f_1(D_1)\) contain the same linear orders up to isomorphism.
\end{enumerate}
In the special case \( D = D_0 = D_1 =  \eta \) we also have:
\begin{enumerate}[label={\upshape (\alph*)}, leftmargin=2pc, resume]
\item \label{lem:loc_not_scat-d} 
Let \( f \colon \Q \to \ALin \) be such that \( f^{-1}(L)\) is dense in \( \Q \) for every \( L \in f(\Q) \). Then \( \eta^f_{(r_0,r_1)} \cong \eta^f\) for every \( r_0,r_1 \in \R \) with \( r_0 < r_1\).
Moreover, for every \( \L \subseteq \AScat \) and every \( L \in \ALin \), we have \( \eta^f \trianglelefteq^\L L \) if and only if \( \eta^f \trianglelefteq L\).

\item \label{lem:loc_not_scat-e} 
Let \( \L \subseteq \AScat \) and \( f_0,f_1 \colon \Q \to \AScat \) be as in part~\ref{lem:loc_not_scat-c}. Then \( {\eta^{f_0} \trianglelefteq^{\L} \eta^{f_1}} \iff {\eta^{f_0} \trianglelefteq \eta^{f_1}} \iff {\eta^{f_0} \cong \eta^{f_1}} \iff f_0(\Q) \) and \( f_1(\Q)\) contain the same linear orders up to isomorphism.

\end{enumerate}
\end{lem}

\begin{proof}
\ref{lem:loc_not_scat-a}
Fix \( m_d \in f(d) \) for each \( d \in D \), and for every \( k \in K \) let \( L'_k \) be the projection of \( L_k \) on its second coordinate. 
Then each \( L'_k \) is convex (in \( D \)) because \( L_k \csube D^f \). 
If every \( L'_k \) were a singleton, then \( D \preceq K \) via the map sending \( d \in D \) to the unique \( k \in K \) such that \( (m_d,d) \in L_k \). 
This is impossible because \( D \) is dense and so \( \Q \preceq D \), while \( K \) is scattered. Hence by convexity there are \( k \in K\) and \( d_0 , d_1 \in D \) such that \( d_0 <_D d_1 \) and \( [d_0,d_1]_D \csube L'_k \). 
Thus \( D^f_{(d_0,d_1)} \csube L_k \), as required. 

\ref{lem:loc_not_scat-b}
Since \( h(D_0^{f_0}) \csube D_1^{f_1} \), its projection \( A \) on the second coordinate is \( D_1 \)-convex. Fix any \( d \in D_0 \). If the projection on the second coordinate of \( h(f_0(d) \times \{ d \} )\) was not a singleton, then \( h(f_0(d) \times \{ d \}) \) would be non-scattered, which is impossible because \( h(f_0(d) \times \{ d \}) \cong f_0(d) \times \{ d \} \cong f_0(d) \in \AScat \). 
Therefore the map \( g \colon D_0 \to A \) sending \( d \in D_0 \) to the unique \( d' \in A \) such that \( h(f_0(d) \times \{ d \}) \subseteq f_1(d') \times \{ d'\} \) is a well-defined surjection, and it is order-preserving since \( h \) was. 
Moreover, \( g \) is also injective: if \( d_0 <_{D_0} d_1 \) were such that \( g(d_0) = g(d_1) \), then \( h \restriction {D_0^{f_0}}_{(d_0,d_1)}\) would be an embedding sending the non-scattered linear order \( {D_0^{f_0}}_{(d_0,d_1)} \) into \( f_1(g(d_0)) \times \{ g(d_0) \} \in \AScat \), a contradiction. 
Thus \( g \) is an order-preserving bijection such that \( h(f_0(d) \times \{ d \} ) \subseteq f_1(g(d)) \times \{ g(d) \} \) for every \( d \in D_0 \), so we only need to show that \( h(f_0(d) \times \{ d \} ) = f_1(g(d)) \times \{ g(d) \} \) for every such \( d \). 
By injectivity of \( g \) and the fact that \( d \) is neither the minimum nor the maximum of \( D_0 \) (because \( D_0 \in \ADLO \) has neither a minimum nor a maximum), we have \( f_1(g(d)) \times \{ g(d) \} \csube h(D_0^{f_0}) \). Thus, if \( h(f_0(d) \times \{ d \} ) \subsetneq f_1(g(d)) \times \{ g(d) \} \), then there would be \( d' \neq d \) such that \( h(f_0(d') \times \{ d' \} ) \cap (f_1(g(d)) \times \{ g(d) \}) \neq \emptyset \), hence \( h(f_0(d') \times \{ d' \} ) \subseteq f_1(g(d)) \times \{ g(d) \}\) and \( g(d') = g(d) \) by definition of \( g \), against the fact that \( g \) is injective.

\ref{lem:loc_not_scat-c}
Assume that \( D_0^{f_0} \trianglelefteq^\L D_1^{f_1} \): we want to show that for every \( d \in D_0 \) there is \( d' \in D_1 \) such that \( f_0(d) \cong f_1(d') \), and vice versa. 
Fix \( K \in \L \), a \( K \)-convex partition \( (L_k)_{k \in K} \) of \( D_0^{f_0} \) and an embedding \( h \colon D_0^{f_0} \to D_1^{f_1} \) witnessing \( D_0^{f_0} \trianglelefteq^\L D_1^{f_1} \). 
By part~\ref{lem:loc_not_scat-a} there are \( k \in K \) and \( d_0<_{D_0} d_1 \) such that \( {D_0^{f_0}}_{(d_0,d_1)} \csube L_k \). 
Notice that \( D_2 = (d_0,d_1)_{D_0} \in \ADLO \), and that \( {D_0^{f_0}}_{(d_0,d_1)} = D_2^{f_2} \) where \( f_2 = f_0 \restriction D_2 \). Thus \( h \restriction D_2 \) witnesses \( D_2^{f_2} \trianglelefteq D_1^{f_1} \), and so we can find 
\( A \csube D_1\) and \( g \colon D_2 \to A \) as in part~\ref{lem:loc_not_scat-b}. Notice that, being isomorphic to \( D_2 \in \ADLO \), the linear order \( A \) is infinite. By the hypothesis on \( f_0 \), given any \( d \in D_0\) there is \(  d'' \in D_2 \) such that \( f_0(d) \cong f_0(d'') \), and hence \( f_0(d) \cong f_1(d') \) for \( d' = g(d'') \). Conversely, given \( d' \in D_1 \) there is \( d'' \in A \) such that \( f_1(d') \cong f_1(d'') \) by the hypothesis on \( f_1 \) and the fact that, being a convex subset of \( D_1 \) with at least two elements, \( A \) contains a nonempty interval of \( D_1 \). Therefore \( f_1(d') \cong f_0(d) \) for \( d = g^{-1}(d'')\) and we are done.

\ref{lem:loc_not_scat-d}
Use a back-and-forth argument to find an order-preserving bijection \( g \colon (r_0,r_1) \cap \Q \to \Q \) such that \( f(q) = f(g(q)) \) for all \( q \in (r_0,r_1) \cap \Q \) --- this can be done by the hypothesis on \( f \). Then the map sending \( (\ell,q) \) to \( (\ell,g(q)) \) is the desired isomorphism. 
For the nontrivial implication of the additional part, assume that \( \eta^f \trianglelefteq^\L L \) as witnessed by \( K \in \L \) and the \( K \)-convex partition \( (L_k)_{k \in K} \) of \( \eta^f\). 
By part~\ref{lem:loc_not_scat-a} there are \( k \in K \) and \( q_0 < q_1 \) such that \( \eta^f_{(q_0,q_1)} \csube L_k \), hence \( \eta^f \cong \eta^f_{(q_0,q_1)} \trianglelefteq L\) and we are done.

\ref{lem:loc_not_scat-e}
If \( f_0(\Q) \) and \( f_1(\Q) \) contain the same linear orders up to isomorphism, then using a back-and-forth argument as in part~\ref{lem:loc_not_scat-d} one can easily show that \( \eta^{f_0} \cong \eta^{f_1} \); this implies \( \eta^{f_0} \trianglelefteq \eta^{f_1} \), which in turn implies \( \eta^{f_0} \trianglelefteq^\L \eta^{f_1} \).
Finally, if \( \eta^{f_0} \trianglelefteq^\L \eta^{f_1} \), then \( f_0(\Q) \) and \( f_1(\Q)\) contain the same linear orders up to isomorphism by part~\ref{lem:loc_not_scat-c}. 
\end{proof}

\begin{remark} \label{rmk:loc_not_scat}
Let \( D \in \ADLO\) and consider any \( \emptyset \neq A \csube D\) without minimum and maximum. 
Then \( A \in \ADLO \), and hence Lemma~\ref{lem:loc_not_scat} can be applied to \( D^f_A\) because \( D^f_A = A^{f \restriction A} \). 
This observation applies e.g.\ to linear orders of the form \( \eta^f_{(r_0,r_1)} \) with \( r_0,r_1 \in \R \) such that \( r_0 < r_1 \).
\end{remark}

\section{Combinatorial properties of \(\qo\)}\label{sec:comb}
	
In this section, we explore the combinatorial properties of \(\L\)-convex embeddability on countable linear orders, for a given \( \L \subseteq \Lin \). 
We always assume that $\L$ is downward \( \preceq \)-closed and ccs.
Actually, the ccs hypothesis is not used in some of our proofs but, since we employ the usual terminology for the combinatorial properties of quasi-orders, it is natural to assume that $(\LO, {\qo})$ is indeed a quasi-order.%
\footnote{If $\L$ is not ccs, we could still view $(\LO, {\qo})$ as an (oriented) graph and state e.g.\ Theorem~\ref{thm:chains_qo} (whose proof does not use the ccs hypothesis) in terms of cliques and independent sets instead of chains and antichains.}

We exclude from our analysis the case $\L = \Lin$ because \(\trianglelefteq^{\Lin}_{\LO}\) coincides with embeddability on \(\LO\), whose combinatorial properties are well known.
We thus usually assume \(\eta \notin \L\), that is, \(\L\subseteq \Scat\).
	
The following lemma generalizes Proposition \ref{prop:comb_prop_cvx}\ref{open_int_cvx}.

\begin{lem}\label{lem:open_int_qo}
For every ccs $\L \subseteq \Scat$, there is an embedding of \((\Int(\R),\subseteq)\) into \((\LO,\qo)\).
\end{lem}
	
\begin{proof}
Let $f \colon \Q \to \{\n \mid n \in \N \setminus \{0\}\}$ be injective: we claim that the map which sends the interval $(x,y) \in \Int(\R)$ to the linear order $\eta^f_{(x,y)} \in \LO$ is the desired embedding.
	
If \((x,y) \subseteq(x',y')\), then \(\eta^f_{(x,y)} \csube \eta^f_{(x',y')}\), and thus \(\eta^f_{(x,y)} \qo \eta^f_{(x',y')}\).
Vice versa, let \((x,y),(x',y') \in \Int(\R)\) be such that \((x,y) \nsubseteq (x',y')\). 
Towards a contradiction, suppose that \(\eta^f_{(x,y)} \qo \eta^f_{(x',y')}\). 
Consider the restriction \(\eta^f_{(r_0,r_1)}\) of \(\eta^f_{(x,y)}\), where \((r_0,r_1)\) is a nonempty open interval contained in \((x,y) \setminus (x',y')\), so that \(\eta^f_{(r_0,r_1)} \qo \eta^f_{(x',y')}\) because \( \eta^f_{(r_0,r_1)} \cvx \eta^f_{(x,y)}\). 
Fix a \( K \)-convex partition \((L_k)_{k \in K}\) of \( \eta^f_{(r_0,r_1)}\) witnessing \(\eta^f_{(r_0,r_1)} \qo \eta^f_{(x',y')}\), for some \( K \in \L \subseteq \Scat \).
By Lemma~\ref{lem:loc_not_scat}\ref{lem:loc_not_scat-a} (see Remark~\ref{rmk:loc_not_scat}) 
there exist \(k \in K\) and \(q_0,q_1 \in \Q \) with $r_0\leq q_0<q_1 \leq r_1$ such that $\eta^f_{(q_0,q_1)} \csube L_{k}$. Hence $\eta^f_{(q_0,q_1)} \cvx \eta^f_{(x',y')}$, and 
hence by Lemma~\ref{lem:loc_not_scat}\ref{lem:loc_not_scat-b} for any \( q_0 < q < q_1\) there is  \( x' < q' < y' \) such that \( f(q) \cong f(q') \). But this contradicts the injectivity of \( f \), as \( q \neq  q' \) because \( (q_0,q_1) \cap (x',y') = \emptyset \).
\end{proof}

\begin{thm}\label{thm:chains_qo}
For every ccs $\L \subseteq \Scat$, there are chains of order type \( (\R, {<} ) \) and antichains of size $2^{\aleph_0}$ in \( \qo \).
\end{thm}
	
\begin{proof}
By Lemma~\ref{lem:open_int_qo} the family \(\Big\{\eta^f_{(0,x)} \mid x>0\Big\}\) is a chain of order type \( (\R,{<}) \), 
while \(\Big\{\eta^f_{(x,x+1)} \mid x \in \R \Big\}\) is an antichain of size the continuum. 
Alternatively, to build a large antichain we can fix a family \( (L_\alpha)_{\alpha < 2^{\aleph_0}} \) of pairwise non-isomorphic scattered linear orders and notice that if \( f_\alpha \colon \Q \to \Scat \) is the constant function with value \( L_\alpha \), then by Lemma~\ref{lem:loc_not_scat}\ref{lem:loc_not_scat-c} the family \( \mathcal{A} = \{ \eta^{f_\alpha} \mid \alpha < 2^{\aleph_0} \} \) is a \( \qo \)-antichain.
\end{proof}

We now show that the dominating number \( \mathfrak{d}(\qo) \) of \( \qo \) (Definition \ref{def:combinatorics}) is as large as possible.

\begin{thm}\label{no_max}
For every ccs $\L \subseteq \Scat$, the quasi-order \(\qo\) does not have maximal elements, and every dominating family with respect to \(\qo\) has size \(2^{\aleph_0}\). Thus \(\mathfrak{d}(\qo)=2^{\aleph_0}\).
\end{thm}
	
\begin{proof}
Let \( L \in \LO \).
By Proposition \ref{prop:comb_prop_cvx}\ref{no_max_cvx} there exists $L'$ such that $L \triangleleft_{\LO} L'$. 
Thus, using Proposition~\ref{prop:nqo} to show that $L' \eta \not \qo L$, we have \( L \triangleleft^\L_\LO L' \eta\), which shows that \( L \) is not \( \qo \)-maximal. 

Let now \(\mathcal{F}\) be a dominating family with respect to \(\qo\): we claim that $\mathcal{F}$ is also a dominating family with respect to $\cvx$, so that $|\mathcal{F}| = 2^{\aleph_0}$ by Proposition \ref{prop:comb_prop_cvx}\ref{prop:dom_fam_cvx}. 
Fix an arbitrary \( L \in \LO \). Since \( \mathcal{F} \) is \( \qo \)-dominating, there is $L' \in \mathcal{F}$ such that $L\eta \qo L'$. But then \( L \cvx L' \) by Proposition~\ref{prop:nqo}, hence we are done. 
\end{proof}

We now look at bases and minimal elements in $\LO$ with respect to $\qo$. Recall that by Proposition~\ref{prop:comb_prop_cvx}\ref{prop:basisforcvx}, if $\L = \{\1 \}$ then there are $2^{\aleph_0}$-many \( \qo \)-incomparable $\qo$-minimal elements. In contrast, the following result extends to most ccs classes \( \L \) a basic fact about $\emb$.

\begin{thm} \label{thm:twoelementsbasis}
For every ccs $\L \subseteq \Lin$, if either \(\boldsymbol{\o}^* \in \L\) or $\boldsymbol{\o} \in \L$ then $\{\boldsymbol{\o}, \boldsymbol{\o}^*\}$ is a basis for $\qo$.
\end{thm}

\begin{proof}
Assume that $\boldsymbol{\o}^* \in \L$.
By Fact~\ref{fct:basic} we have that \( \boldsymbol{\o}^* \qo L\) for every ill-founded \( L \in \LO \). On the other hand, if $L \in \WO \cap \LO$ then trivially $\boldsymbol{\o} \cvx L$, and hence $\boldsymbol{\o} \qo L$. The case when \( \boldsymbol{\o} \in \L \) is symmetric.
\end{proof}

Since \( \L \) is downward \( \preceq \)-closed, if \( \L \) contains at least an infinite linear order then Theorem~\ref{thm:twoelementsbasis} applies and \( \qo \) has a basis of size \( 2 \).
It thus remains to consider families \( \L \) such that \(\L \subseteq \Fin\), which by the ccs property amounts to \( \L = \{\1\} \) or \( \L = \Fin\). 
In this case, we can reproduce the result obtained for \( \cvx \) in~\cite{IMMRW22} and show that there are $2^{\aleph_0}$-many \( \qo \)-incomparable $\qo$-minimal elements. 
To motivate the next technical result, notice that by Fact~\ref{fct:basic} the relation $\qo$ coincides with embeddability on $\L$, so that all \( \qo \)-antichains have finite intersection with $\L$. 
Therefore, in order to find infinite antichains (of minimal elements) we have to search in $\LO \setminus \L$. 

For every infinite \( S \subseteq \N \setminus \{ 0 \} \), fix a surjective map \(f_S\colon \Q \rightarrow \{ \n \mid n \in  S \}\) such that $f_S^{-1}(\n)$ is dense for every $n \in S$.
 
\begin{prop}\label{prop:min_el_qo}
Let \( S,S' \subseteq \N \setminus \{ 0 \}  \) be infinite, and consider any ccs \( \L \subseteq \Scat \).
\begin{enumerate-(a)}
\item \label{prop:min_el_qo-a}
If \( S \neq S' \),  then \( \eta^{f_S} \not\qo \eta^{f_{S'}} \).
\item \label{prop:min_el_qo-b}
\( \eta^{f_S} \) is \( \qo \)-minimal in \( \LO \setminus \L \). 
\item \label{prop:min_el_qo-c}
Every basis \( \mathcal{B} \) for the restriction of \( \qo \) to \( \LO \setminus \L \) contains infinite \( \qo \)-decreasing chains.
\end{enumerate-(a)}
\end{prop}
	
\begin{proof}
%
\ref{prop:min_el_qo-a} 
This is just an immediate application of Lemma~\ref{lem:loc_not_scat}\ref{lem:loc_not_scat-c}.

\ref{prop:min_el_qo-b}
Assume that \( L \in \LO \setminus \L \) is such that \( L \qo \eta^{f_S} \), as witnessed by \( K \in \L \), \( (L_k)_{k \in K} \) and \( h \colon L \to \eta^{f_S} \).
Notice that there is \( k \in K \) such that \( L_k \) is infinite: this is trivial if \( K \) is finite, while if \( K \) is infinite we can apply Theorem~\ref{thm:ccs_implies_star} and use the fact \( L \notin \L \). Thus
\( h(L_k) \) is an infinite convex subset of \( \eta^{f_S} \), which means that \( \eta^{f_S}_{(q_0,q_1)} \csube h(L_k) \) for some \( q_0 < q_1 \), and hence \( \eta^{f_S}_{(q_0,q_1)} \cvx L \) via \( h^{-1} \). 
Since \( \eta^{f_S}_{(q_0,q_1)} \cong \eta^{f_S} \) by Lemma~\ref{lem:loc_not_scat}\ref{lem:loc_not_scat-d}, it follows that \( \eta^{f_S} \cvx L \), and thus also \( \eta^{f_S} \qo L \). 
This proves that there is no \( L \in \LO \setminus \L \) such that \( L \triangleleft^\L_\LO \eta^{f_S} \), as desired.

\ref{prop:min_el_qo-c} As in the proof of Lemma~\ref{lem:open_int_qo}, let \( f \colon \Q \to \{ \n \mid n \in \N \setminus \{ 0 \} \} \) be injective and notice that \( \eta^f \in \LO \setminus \L \) because it is not scattered. 
It is enough to prove that for every \( L \in \mathcal{B} \) with \( L \qo \eta^f \) there is \( L' \in \mathcal{B} \) such that \( L' \triangleleft^\L_\LO L \), as using this fact one can then recursively construct an infinite \( \qo \)-descending chain in \( \mathcal{B} \). 
Suppose that \( K \in \L \) and the \( K \)-convex partition \( (L_k)_{k \in K} \) of \( L \) (together with an appropriate embedding) witness \( L \qo \eta^f \). 
As before, there is \( k \in K \) such that \( L_k \) is infinite
(use Theorem~\ref{thm:ccs_implies_star} together with \( L \in \mathcal{B} \subseteq \LO \setminus \L \) when \( K \) is infinite), which in turn implies that \( \eta^f_{(q_0,q_1)} \cvx L \) for some \( q_0 < q_1 \).
Pick \( q'_0, q'_1 \in \Q \) such that \( q_0 < q'_0 < q'_1 < q_1 \), so that \( \eta^f_{(q'_0,q'_1)} \triangleleft^\L_\LO \eta^f_{(q_0,q_1)} \) by (the proof of) Lemma~\ref{lem:open_int_qo}. 
Since \( \eta^f_{(q'_0,q'_1)} \) is not scattered, it cannot belong to \( \L \). 
Thus there is \( L' \in \mathcal{B} \) such that \( L' \qo \eta^f_{(q'_0,q'_1)} \), and hence \( L' \triangleleft^\L_\LO L \) by the choice of \( q'_0 \) and \( q'_1 \).
\end{proof}



\begin{thm} \label{thm:basis2}
For any ccs \( \L \subseteq \Fin \) there are \( 2^{\aleph_0} \)-many \( \qo \)-incomparable \( \qo \)-minimal elements in \( \LO \). Thus every basis \( \mathcal{B} \) for \( \qo \) has cardinality \( 2^{\aleph_0} \). Moreover, \( \mathcal{B} \) is ill-founded with respect to \( \qo \), and thus it cannot be an antichain.
\end{thm}

\begin{proof} 
Since \( \LO \cap \Fin = \emptyset \), 
we have \( \LO \setminus \L = \LO \). Therefore, by Proposition~\ref{prop:min_el_qo}\ref{prop:min_el_qo-a}--\ref{prop:min_el_qo-b} the family 
\[ 
\mathcal{C}=  \Big\{ \eta^{f_S} \mid S \subseteq \N \setminus \{ 0 \}  \text{ is infinite} \Big\} 
\] 
is a \( \qo \)-antichain of \( \qo \)-minimal elements. The fact that every basis for \( \qo \) is ill-founded follows instead from Proposition~\ref{prop:min_el_qo}\ref{prop:min_el_qo-c}.
\end{proof}

\begin{prop} \label{prop:scatantichains}
Let $\L \subseteq \Scat$ be ccs. Then any \( \qo \)-antichain of size less than $2^{\aleph_0}$ contained in \( \LO \setminus \L \) can be extended to a  \( \qo \)-antichain of size $2^{\aleph_0}$ also contained in \( \LO \setminus \L \). In particular for every \( L \in \LO \setminus \L \) there is \( M \in \LO \setminus \L \) which is \( \qo \)-incomparable with \( L \), and indeed \( L \) belongs to a \( \qo \)-antichain of size \( 2^{\aleph_0} \).
\end{prop}

\begin{proof}
Let $\mathcal{A}=\{ L_\a \mid \a < \kappa\}$ be a \( \qo \)-antichain of size $\kappa < 2^{\aleph_0}$ with \( L_\alpha \notin \L \) for all \( \alpha < \kappa \),  and let  $\mathcal{C}$ be the \( \qo \)-antichain of size \( 2^{\aleph_0} \) from the proof of Theorem~\ref{thm:basis2}. 
Given \( \alpha < \kappa \), consider the set $\mathcal{C}_\a=\{\eta^{f_S} \in \mathcal{C} \mid \eta^{f_S} \qo L_\a\}$. 
By Lemma~\ref{lem:loc_not_scat}\ref{lem:loc_not_scat-d}, if \( \eta^{f_S} \qo L_\alpha \) then also \( \eta^{f_S} \cvx L_\alpha \),
thus \( \mathcal{C}_\alpha = \{ L \in \mathcal{C} \mid L \cvx L_\alpha \} \) and so \( \mathcal{C}_\alpha \) is countable by~\cite[Claim 3.10.1]{IMMRW22}.
Therefore $\bigcup_{\a<\kappa} \mathcal{C}_\a$ has size at most $\max \{ \kappa,\aleph_0 \}$. 
From this and $\qo$-minimality over \( \LO \setminus \L \) of the linear orders $\eta^{f_S}$ (Proposition~\ref{prop:min_el_qo}\ref{prop:min_el_qo-b}), it then follows that $\mathcal{A} \cup (\mathcal{C}\setminus \bigcup_{\a<\kappa} \mathcal{C}_\a)$ is the desired \( \qo \)-antichain in \( \LO \setminus \L \) of size \( 2^{\aleph_0} \) extending \( \mathcal{A} \).
%
\end{proof}

Using again the fact that \( \LO \setminus \Fin = \LO \),
we obtain:

\begin{cor}
For every ccs \( \L \subseteq \Fin \) there are no maximal \( \qo \)-antichains of size smaller than \( 2^{\aleph_0} \).
\end{cor}	

In contrast, as observed in the proof of Theorem~\ref{thm:twoelementsbasis}, if \( \Fin \subsetneq \L \), so that either \( \boldsymbol{\omega} \in \L \) or \( \boldsymbol{\omega}^* \in \L \), then \( \{ \boldsymbol{\omega}, \boldsymbol{\omega}^* \}\) is a maximal \( \qo \)-antichain of size \( 2 \). It is thus interesting to understand for which \( \L \) there are arbitrarily large finite maximal \( \qo \)-antichains, or even countably infinite maximal \( \qo \)-antichains.

\begin{cor} \label{cor:scatantichains}
All maximal \( \trianglelefteq^\Scat_\LO\)-antichains \( \mathcal{A} \) are either finite or of size \( 2^{\aleph_0} \). More precisely:
\begin{enumerate-(a)}
\item \label{cor:scatantichains-a}
If \( \mathcal{A} \cap \Scat \neq \emptyset \), then \( \mathcal{A} \subseteq \Scat \) and \( \mathcal{A} \) is also an antichain with respect to \( \emb \), hence it is finite.
\item \label{cor:scatantichains-b}
If \( \mathcal{A} \cap \Scat = \emptyset \), then \( |\mathcal{A}| = 2^{\aleph_0} \).
\end{enumerate-(a)}
Thus there is no countably infinite maximal \( \trianglelefteq^\Scat_\LO\)-antichain.
\end{cor}

\begin{proof}
\ref{cor:scatantichains-a}
Let \( L \in \mathcal{A} \cap \Scat \). 
If \( L' \notin \Scat \), then $L \emb \eta \emb L'$ and therefore \( L \trianglelefteq^\Scat_\LO L' \) by Fact~\ref{fct:basic}. Hence \( \mathcal{A} \subseteq \Scat \). 
Moreover, on \( \Scat \) the relations \( \trianglelefteq^\Scat_\LO \) and \( \emb \) coincide again by Fact~\ref{fct:basic}, hence we are done.

\ref{cor:scatantichains-b}
Apply Proposition~\ref{prop:scatantichains}.
\end{proof}

\begin{remark}
For an arbitrary $\L$, if an antichain $\mathcal{A}$ intersects $\L$ then it is included in $\Scat$ because $L \emb L'$ whenever $L \in \LO$ and $L' \notin \Scat$. 
However, in contrast with Corollary~\ref{cor:scatantichains}\ref{cor:scatantichains-a}, this does not rule out the existence of large \( \qo \)-antichains of scattered linear orders when \( \L \subsetneq \Scat \).
For example consider for every $f \in \N^\N$ the linear order $L_f = \z \boldsymbol{\o}^* + \sum_{n \in \N} (\z + f(n))$; then $L_f \trianglelefteq_\LO^{\Fin} L_{f'}$ if and only if $L_f \underline{\bowtie}_\LO^{\Fin} L_{f'}$ if and only if $\exists n,n'\, \forall i\, f(n+i)=f'(n'+i)$; we thus have a $\trianglelefteq_\LO^{\Fin}$-antichain of size $2^{\aleph_0}$ contained in $\Scat$.

Other configurations of maximal antichains are possible as well. For example, \( \L_{\preceq \boldsymbol{\omega}} \) is ccs by Proposition~\ref{prop:classes_gamma}, and it is easy to check using Proposition~\ref{prop:scatantichains} that every maximal \( \trianglelefteq_\LO^{\L_{\preceq \boldsymbol{\omega}}} \)-antichain either is  of the form \( \{ \boldsymbol{\omega}, \boldsymbol{\alpha}^* \} \) for some infinite \( \alpha < \omega_1 \), or else has size \( 2^{\aleph_0} \).
\end{remark}

Motivated by Proposition~\ref{prop:comb_prop_cvx}\ref{prop:WO_unbounded_cvx},
we now analyse the (un)boundedness of \(\WO \cap \LO \) in \(\LO\) with respect to \(\qo\). We have to distinguish two cases.
	
\begin{prop}\label{prop:wo_boundedness}
Consider any ccs $\L \subseteq \Lin$.
\begin{enumerate-(a)}
\item\label{wo_in_L} 
If \(\WO \subseteq \L\), then \(\WO \cap \LO\) is bounded with respect to \(\qo\).
\item\label{wo_not_in_L} 
If \(\WO \nsubseteq \L\), then \(\WO \cap \LO \) is unbounded with respect to \(\qo\).
\end{enumerate-(a)}
\end{prop}

\begin{proof}
\ref{wo_in_L}
By Fact~\ref{fct:basic}, any upper \( \emb \)-bound for \( \WO \cap \LO  \) is also an upper bound with respect to \( \qo \). Thus every non-scattered linear order \( \qo \)-bounds \( \WO \cap \LO \) from above.

\ref{wo_not_in_L}
Let \( \beta < \omega_1 \) be such that  $\bb \notin \L$, and consider any \(L \in \LO\). By Proposition~\ref{prop:comb_prop_cvx}\ref{prop:WO_unbounded_cvx} there is \( \omega \leq \alpha < \omega_1 \) such that $\aa \ncvx L$, hence
 $\aa \bb \nqo L$ by Proposition~\ref{prop:nqo}. Since \( \aa \bb \in \WO \cap \LO\) and \( L \) was arbitrary, this shows that \( \WO \cap \LO \) is \( \qo \)-unbounded.
\end{proof}

Using infinite (countable) sums of linear orders, it is immediate to prove that \( \mathfrak{b}(\qo) > \aleph_0 \). Taking this into account, we show that \( \mathfrak{b}(\qo)\) is as small as possible.

\begin{thm}\label{thm:unbounding_number}
For every ccs $\L \subseteq \Scat$ there exists a family $\mathcal{F}$ of size $\aleph_1$ which is unbounded with respect to $\qo$. Thus, \( \mathfrak{b}(\qo)=\aleph_1 \).
\end{thm}

\begin{proof}
Let $\mathcal{F}=\{\aa \eta \mid \alpha < \omega_1 \}$. Since \( \aa \eta = \eta^{f_\alpha} \) where \( f_\alpha \colon \Q \to \Scat \) is the constant function with value \( \aa \), by Lemma~\ref{lem:loc_not_scat}\ref{lem:loc_not_scat-c} the family  \( \mathcal{F} \) is a \( \qo \)-antichain of size \( \aleph_1 \): we claim that it is \( \qo \)-unbounded in \( \LO \).
%
Indeed, suppose towards a contradiction that $\mathcal{F}$ is \( \qo \)-bounded from above by some $L \in \LO$. Then $\aa\eta \qo L$ for every $\alpha < \omega_1$, hence by Proposition~\ref{prop:nqo} we would have $\aa \cvx L$ for every $\alpha < \omega_1$, against  Proposition~\ref{prop:comb_prop_cvx}\ref{prop:WO_unbounded_cvx}.
\end{proof}
	

The next result shows that \( (\LO, {\qo} ) \) exhibits a high degree of self-similarity when $\L \neq \Lin$ (the statement obviously fails for $\emb$). Given \( L_0 \in \LO \), we let \({L_0}{\uparrow^{\L}}\ = \{L \in \LO\ \mid L_0 \qo L\}\) be the \( \qo \)-upper cone above \( L_0 \).	

\begin{thm}\label{thm:fractal}
For every ccs $\L \subseteq \Scat$, the partial order \( (\LO, {\qo} ) \) has the fractal property with respect to its upper cones, that is, 
\((\LO,\qo)\) embeds into \(({L_0}{\uparrow^{\L}},\qo)\) for every
\(L_0 \in \LO\).
\end{thm}

\begin{proof} 
Fix $L_0 \in \LO$ and, using Proposition \ref{prop:comb_prop_cvx}\ref{prop:WO_unbounded_cvx}, fix $\alpha < \omega_1$ such that $\aa \ncvx L_0$ (in particular, \( \alpha \geq \omega \)).
Consider the map \(\varphi \colon \LO \to L_0\uparrow^{\L}\) defined by 
\[
\varphi(L) = ( \aa \eta_0 + \eta_1 + L_0 + \eta_2)L,
\]
where to help the reader we denote by $\eta_j$ distinct copies of $\eta$: we show that $\varphi$ is an embedding from \((\LO,{\qo})\) to \((L_0\uparrow^{\L},{\qo})\). 

Clearly, if \(L \qo L'\) via \(K \in \L\), some $K$-convex partition \((L_k)_{k \in K}\) of \(L\) and an embedding \( g \colon L \to L' \), then \(K\) itself, the $K$-convex partition \( (L'_k)_{k \in K}\) of \(\varphi(L)\) given by \( L'_k = (\aa \eta_0 + \eta_1 + L_0 + \eta_2)L_k \) and the embedding \( h \colon \varphi(L) \to \varphi(L') \) defined by \( h(x,\ell)= (x,g(\ell)) \) witness that \(\varphi(L) \qo \varphi(L')\).

For the other direction, suppose that $\varphi(L) \qo \varphi(L')$ as witnessed by $K \in \L$,  the $K$-convex partition $(M_k)_{k \in K}$ of $\varphi(L)$ and the embedding $h \colon \varphi(L) \to \varphi(L')$. 
For each $\ell \in L$, consider the partition of $\aa \eta_0 \times \{\ell\} \csube \varphi(L)$ given by the nonempty sets of the form \( M_k \cap (\aa \eta_0 \times \{\ell\})\), which is a \( K' \)-convex partition for some \( K' \subseteq K \in \L \subseteq \Scat\): 
since $\aa \eta_0 \cong \eta^f$ where $f \colon \Q \to \Lin$ is the constant function with value $\aa$,
by Lemma~\ref{lem:loc_not_scat}\ref{lem:loc_not_scat-a} we can choose%
\footnote{In general, the choice of \( N_\ell \) and \( k_\ell \) is not unique. The fact that \( \eta^f_{(q_0^{(\ell)},q_1^{(\ell)})} \times \{ \ell \} \cong \aa \eta_0 \) follows from Lemma~\ref{lem:loc_not_scat}\ref{lem:loc_not_scat-d}.} 
$N_\ell = \eta^f_{(q_0^{(\ell)},q_1^{(\ell)})} \times \{ \ell \} \cong \aa \eta_0$ and $k_\ell \in K$ such that $N_\ell \csube (\aa \eta_0 \times \{\ell\}) \cap M_{k_\ell}$, so that \( h \restriction N_\ell \) witnesses \( N_\ell \cvx \varphi(L') \).
If $h(N_\ell) \cap (\eta_j \times \{\ell'\}) \neq \emptyset$ for some $j \in \{ 1,2 \}$ and $\ell' \in L'$, then \( N_\ell \) would contain a convex subset isomorphic to \( \eta \), which is not the case because \( \alpha > 1 \). 
Therefore either $h(N_\ell) \csube \aa \eta_0 \times \{\ell'\}$ or $h(N_\ell) \csube L_0 \times \{\ell'\}$ for some (necessarily unique) $\ell' \in L'$. But $\aa \cvx N_\ell$ and $\aa \ncvx L_0$, hence the second possibility cannot hold.
This shows that there is a well-defined map \( g \colon L \to L' \) such that $h(N_\ell) \csube \aa \eta_0 \times \{g(\ell)\}$ for all \( \ell \in L \): we claim that \( g \) is an embedding. 
Indeed, for every \( \ell_0,\ell_1 \in L \) we have 
 \[ 
 \ell_0 <_L \ell_1 \iff N_{\ell_0} <_{\varphi(L)} N_{\ell_1} \iff h(N_{\ell_0}) <_{\varphi(L')} h(N_{\ell_1}) 
 \] 
because \( h \) is an embedding. 
If there were \( \ell_0 <_L \ell_1\) such that \( g(\ell_0) = g(\ell_1) \), then \( h(\eta_1 \times \{ \ell_0 \} ) \subseteq \aa \eta_0 \times \{ g(\ell_0) \}\) because \( N_{\ell_0} <_{\varphi(L)} \eta_1 \times \{ \ell_0 \} <_{\varphi(L)} N_{\ell_1} \).
Let \( k \in K \) be such that \( M_k \cap (\eta_1 \times \{ \ell_0 \})\) contains an interval \( (q_0,q_1) \times \{ \ell_0 \} \) of \( \eta_1 \times \{ \ell_0 \} \), for some \( q_0 < q_1 \).
(Such a \( k \) exists by Lemma~\ref{lem:loc_not_scat}\ref{lem:loc_not_scat-a} applied to \( \eta_1 \times \{ \ell_0 \} \), which is isomorphic to \( \eta^f \) where \( f \) the constant function with value \( \1 \).)
Then \( h((q_0,q_1) \times \{ \ell_0 \})\) would be a convex subset of \( \aa \eta_0 \times \{ g(\ell_0) \} \) homeomorphic to \( \eta \), which is clearly impossible because \( \alpha > 1 \).
Thus \( g \) is injective, and hence for all \( \ell_0, \ell_1 \in L\)
\begin{multline*}
\ell_0 <_L \ell_1 \iff h(N_{\ell_0}) <_{\varphi(L')} h(N_{\ell_1}) \iff \\ 
\aa \eta_0 \times \{ g(\ell_0) \} <_{\varphi(L')} \aa \eta_0 \times \{ g(\ell_1) \} \iff g(\ell_0) <_{L'} g(\ell_1).
\end{multline*}

Now set $L_k = \{\ell \in L \mid k_\ell = k\}$ for each $k \in K$, 
%
and let $K'=\{k \in K \mid L_k \neq \emptyset\} \subseteq K$, so that \( K' \in \L \) by downward \( \preceq \)-closure of \( \L \). 
Clearly, $\bigcup_{k \in K'}L_k=L$. Moreover, for every \( k,k' \in K' \) we have 
\begin{multline*}
k <_{K'} k'  \iff M_k <_{\varphi(L)} M_{k'}  \iff \\
\forall \ell_0 \in L_k \forall \ell_1 \in L_{k'} \, ( N_{\ell_0} <_{\varphi(L)} N_{\ell_1}) \iff L_k <_L L_{k'},
\end{multline*}
and thus $(L_k)_{k \in K'}$ is a $K'$-convex partition of $L$. In particular, every \( L_k \) is \( L \)-convex.

We also claim that $g(L_k)\csube L'$ for all $k \in K'$. 
Pick arbitrary $\ell_0, \ell_1 \in L_k$ such that $g(\ell_0) <_{L'} g(\ell_1)$, and consider any $m' \in L'$ such that $g(\ell_0) <_{L'} m' <_{L'} g(\ell_1)$, so that in particular \( h(N_{\ell_0}) <_{\varphi(L')} \aa \eta_0 \times \{ m' \} <_{\varphi(L')} h(N_{\ell_1}) \) and \( \aa \eta_0 \times \{ m' \} \csube h(M_k) \). 
Since \( h \restriction M_k\) is an isomorphism between \( M_k \) and the \( \varphi(L') \)-convex set \( h(M_k) \), and since \( \aa \eta_0 \times \{ m' \} \) does not contain any \( \varphi(L') \)-convex subset isomorphic to \( \eta \) because \( \alpha > 1 \), then \( h^{-1}(\aa \eta_0 \times \{ m' \}) \cap (\eta_j \times \{ m \}) = \emptyset \) for every \( j \in \{ 1,2 \} \) and \( m \in L \). 
Since \( h^{-1}(\aa \eta_0 \times \{ m' \} \subseteq L_0 \times \{ m \} \) is impossible by choice of \( \alpha \), we conclude that there is \( m \in L \) such that $h^{-1}(\aa \eta_0 \times \{m'\}) \csube (\aa \eta_0 \times \{m\})$.
Notice that \( \ell_0 \leq_L m \leq_L \ell_1 \) because \( N_{\ell_0} <_{\varphi(L)} h^{-1}(\aa \eta_0 \times \{m'\}) <_{\varphi(L)} N_{\ell_1} \), hence \( m \in L_k \) because the latter is \( L \)-convex, and so \( k_m = k\). 
Suppose towards a contradiction that \( m = \ell_0 \). 
Then the \( h \)-preimage of \( \eta_1 \times \{ g(\ell_0) \} \), which is a \( \varphi(L') \)-convex subset of \( \varphi(L') \) between \( h(N_{\ell_0}) \) and \( \aa \eta_0 \times \{ m' \} \), would be a \( \varphi(L) \)-convex subset of \( \aa \eta_0 \times \{ \ell_0 \} \), which is impossible. 
A similar argument excludes \( m = \ell_1 \): hence \( \ell_0 <_L m <_L \ell_1 \) and \( \aa \eta_0 \times \{ m \} \csube M_k \).
By the usual argument, this entails that \( h(\aa \eta_0 \times \{ m \} ) \subseteq \aa \eta_0 \times \{ \ell' \} \) for some \( \ell' \in L' \), and necessarily \( \ell' = m' \) by choice of \( m \). Thus \( m' = g(m) \), so \( m' \in g(L_k) \). 
Since \( m' \) was arbitrary, \( g(L_k) \) is \( L' \)-convex.

This concludes the proof because we have shown that \( K' \in \L \), the \( K' \)-convex partition \( (L_k)_{k \in K} \) and \( g \) witness \( L \qo L' \), as desired.
\end{proof}	

In contrast to Theorem~\ref{thm:fractal},  it is often not possible to embed \((\LO,{\qo})\) into a lower cone \(({L_0}{\downarrow^{\L}},{\qo})\), where \({L_0}{\downarrow^{\L}} = \{L \in \LO \mid L \qo L_0\}\). 
This is trivial if we consider a \( \qo \)-minimal element in $\LO$, such as  \( \boldsymbol{\o} \) or \( \boldsymbol{\o}^* \) if \( \L \nsubseteq \Fin \), 
or the non-scattered minimal elements from Theorem~\ref{thm:basis2} if \( \L \subseteq \Fin \).  

Besides the ones determined by minimal elements, there are many other lower cones in which \( (\LO,{\qo}) \) cannot be embedded. For example, if $\L \subseteq \Fin$ and  \(L_0 \in \Scat\), then \({L_0}{\downarrow^{\L}}\) contains countably many equivalence classes under \(\eq\) (this follows from the fact that a countable scattered linear order has countably many convex subsets, \cite{Bon75}), and thus by Theorem~\ref{thm:chains_qo} there is again no embedding from \((\LO,\qo)\) into \(({L_0} {\downarrow^{\L}},\qo)\). If instead $\Fin \subsetneq \L \subseteq \Scat$, we can notice that if $L_0 \in \LO \cap \L$ then $(\LO,\qo)$ is not embeddable in $({L_0}{\downarrow^{\L}},{\qo})$ because the latter coincides with $({L_0}{\downarrow^{\L}},{\emb})$ by Fact~\ref{fct:basic}, and hence it is a wqo.
In fact, we have no examples of an \(L_0 \in \LO\) and a ccs $\L \subseteq \Scat$ such that \((\LO,{\qo})\) embeds into \(({L_0}{\downarrow^{\L}},\qo)\). If instead $\L = \Lin$ the situation is clearer: since $\trianglelefteq^\Lin_\LO$ is $\emb$, then%
\footnote{For the nontrivial direction, notice that if there were an embedding \( f \) of \((\LO,{\trianglelefteq^\Lin_\LO })\) into \(({L_0}{\downarrow^{\L}},{\trianglelefteq^\Lin_\LO})\) then \( f(L_0) \prec f(\eta) \preceq L_0 \) (recall that \( \trianglelefteq^\Lin_\LO \) is just \( \emb \)). Thus also \( f^{(2)}(L_0) = (f \circ f)(L_0) \prec f(L_0) \), and iterating the process \( f^{(n+1)}(L_0) \prec f^{(n)}(L_0) \) for every \( n \in \omega \). But then \( (f^{(n)}(L_0) )_{n \in \N} \) would be an infinite descending chain, contradicting the fact that \( \emb \) is wqo.}
\((\LO,{\qo})\) embeds into \(({L_0}{\downarrow^{\L}},{\qo})\) if and only if $L_0$ is not scattered (in which case ${L_0}{\downarrow^{\L}} = \LO$).

\section{Borel complexity of \( \qo \) and \(\eq\)}\label{sec:borel_compl}
	
In this section we analyze the descriptive set-theoretic complexity of the quasi-order $\qo$ and of its associated equivalence relation \( \eq \).
We again mostly work with ccs families \( \L \subsetneq \Lin \), as \( \underline{\bowtie}^{\Lin}_\LO \) is just the well-studied relation \( \equiv_\LO \) of biembeddability (also called equimorphism) on \( \LO \).

We first determine bounds on the complexity of \( \qo \) and \( \eq \) as subsets of \( \LO \times \LO \), in some cases using Wadge reducibility $\leq_W$ (see \cite[Definition 21.13]{Kec95}). 
Since their definition includes an existential quantification over \( \L \), it is not surprising that their complexity depends on that of \( \L \). 

\begin{prop} \label{prop:complexityinthesquare}
Let \( \L \subseteq \Lin \) be downward \( \preceq \)-closed.
\begin{enumerate-(a)}
\item \label{prop:complexityinthesquare-a}
\( \L \) is a coanalytic subset of \( \Lin \), and thus it cannot be proper analytic.
\item \label{prop:complexityinthesquare-b}
The relations \( \qo \) and \( \eq \) are both \( \boldsymbol{\Sigma}^1_2 \).
\item \label{prop:complexityinthesquare-c}
If \( \L \) is Borel, then \( \qo \) and \( \eq \) are analytic.
\item \label{prop:complexityinthesquare-e}
If $\L$ is ccs, 
then \( \L \leq_W {\qo} \) and \( \L \leq_W {\eq} \). Thus if \( \L \) is also proper coanalytic (which in particular implies \( \L \neq \Lin \) and hence \( \L \subseteq \Scat \)) then  \( \qo \) and \( \eq \) are not analytic, while if \( \L \) is \( \boldsymbol{\Pi}^1_1 \)-complete then \( \qo \) and \( \eq \) are \( \boldsymbol{\Pi}^1_1 \)-hard.\end{enumerate-(a)}
\end{prop}

Part~\ref{prop:complexityinthesquare-c} applies e.g.\ to the families \( \{ \1 \} \), \( \Fin \), \( \Lin \), and all the ccs classes considered in Section~\ref{sec:examples}; instead, 
part~\ref{prop:complexityinthesquare-e} 
applies e.g.\ to \( \WO \) and \( \Scat \).

\begin{proof}
\ref{prop:complexityinthesquare-a}
Since $\emb$ is a wqo, $\Lin \setminus \L$ is a finite union of upward \( \preceq \)-closed cones, each of which is analytic because $\emb$ is an analytic relation. Then $\Lin \setminus \L$ is analytic and  $\L$ is coanalytic.

\ref{prop:complexityinthesquare-b} 
Immediate from Definition~\ref{def:QO}, taking into account part~\ref{prop:complexityinthesquare-a}.

\ref{prop:complexityinthesquare-c}
Similar to~\ref{prop:complexityinthesquare-b}.


\ref{prop:complexityinthesquare-e}
Consider the continuous map \( \varphi \colon \LO \to \LO \) defined by \( \varphi(L) = (\eta + \boldsymbol{2}) L\). 
We claim that \( L \in \L \iff \varphi(L) \qo \eta \iff \varphi(L) \eq \eta \), which amounts to just showing the first equivalence because \( \eta \cvx \varphi(L) \) for every \( L \in \LO \). 
By Theorem~\ref{thm:ccs_implies_star}, if \( L \in \L \) then \( K = \boldsymbol{2} L \in \L \), and the \( K \)-convex partition \( (L_k)_{k \in K} \) of \( \varphi(L) \) whose first element of each pair is $\eta + \boldsymbol{1}$ and the second element is $\boldsymbol{1}$ can be used to witness \( \varphi(L) \qo \eta \) in the obvious way. 
Conversely, assume that \( \varphi(L) \qo \eta \) via some \( K \in \L \) and some \( K \)-convex partition \( (L_k)_{k \in K} \) of \( \varphi(L) \).
Notice that whenever $\ell, \ell' \in L$ are distinct then no convex subset of $\varphi(L)$ isomorphic to a convex subset of \( \eta \) contains both $(0,\ell)$ and $(0,\ell')$. Therefore the map associating to each $\ell \in L$ the unique $k \in K$ such that $(0,\ell) \in L_k$ is injective and order-preserving, so that $L \emb K$ and $L \in \L$.
\end{proof}

We now move to the classification of \( \eq \) with respect to Borel reducibility. 
Noticing that if \( M \in \Lin \) the restriction to \( \LO \) of the map \( \varphi \) from Proposition~\ref{thm:cong_qo} is Borel, we get the following generalization of part~\ref{thm:complexity_cvxeq-a} of Theorem~\ref{thm:complexity_cvxeq}, which corresponds to the case \( \L = \{ \1 \} \).

\begin{thm}\label{thm:iso_breduces_eq}
For every ccs $\L\subseteq \Scat$, we have ${\iso} \mathrel{\leq_B} {\eq}$.
\end{thm} 

\begin{proof}
Apply Proposition~\ref{thm:cong_qo} with \( M = \eta \).
\end{proof}

Theorem~\ref{thm:iso_breduces_eq} also provides lower bounds for the complexity of \( \qo \) and \( \eq \) as subsets of \( \LO \times \LO \).

\begin{thm} \label{cor:complexityinthesquare}
For every downward \( \preceq \)-closed \( \L \subseteq \Lin \), the relations \( \qo \) and \( \eq \) are \( \boldsymbol{\Sigma}^1_1 \)-hard. 
Therefore if \( \L \) is Borel then \( \qo \) and \( \eq \) are complete analytic (as subsets of \( \LO \times \LO \)); if instead \( \L \) is proper coanalytic and ccs, then they are neither analytic nor coanalytic, hence they are at least \( \boldsymbol{\Delta}^1_2 \).
\end{thm}

\begin{proof}
If \( \L = \Lin \), then the map \( L \mapsto (\eta,L) \) simultaneously reduces the \( \boldsymbol{\Sigma}^1_1 \)-complete set \( \Lin \setminus \Scat \) to \( \trianglelefteq^\Lin_\LO \) and \( \underline{\bowtie}^\Lin_\LO \) because they coincide with \( \emb \) and \( \equiv_\LO \), respectively. 
If instead \( \L \subseteq \Scat \), use Theorem~\ref{thm:iso_breduces_eq} and the well-known fact that \( \iso \) is a \( \boldsymbol{\Sigma}^1_1\)-complete subset of \( \LO \times \LO \).
\end{proof}

Part~\ref{thm:complexity_cvxeq-b} of Theorem \ref{thm:complexity_cvxeq}, namely the fact that ${E_1} \mathrel{\nleq_{\text{\scriptsize \textit{Baire}}}} {\cvxeq}$, does not instead generalize to an arbitrary $\eq$, and actually we have the opposite situation for every ccs \( \L \) different from \( \{ \1 \} \) and \( \Lin \).
(Recall that \( {E_1} \mathrel{\nleq_{\text{\scriptsize \textit{Baire}}}} {\underline{\bowtie}^\Lin_\LO}  \) because \( \emb \) is a wqo and \( \underline{\bowtie}^\Lin_\LO \) is just \( \equiv_\LO \), hence the hypotheses of Theorem~\ref{thm:red_E1_eq} are optimal.)

\begin{thm}\label{thm:red_E1_eq}
For every ccs $\L$ such that \( \Fin \subseteq \L \subseteq \Scat \) we have \({E}_1 \leq_B {\eq}\).
\end{thm}

\begin{proof}
Let \( (\R^+)^\N \) be the set of sequences of positive real numbers, whose elements will be denoted by \( (x_n)_{n \in \N} \) or, for the sake of brevity, by \( \vec{x} \). 
Consider the restriction \( E_1 \restriction (\R^+)^\N \) of $E_1$ to \( (\R^+)^\N \): applying the exponential function pointwise, one immediately sees that $E_1 \restriction (\R^+)^\N \sim_B E_1$.
Fix an injective $f \colon \Q \to \{\n \mid n \in \N \setminus \{0\}\}$, and consider once again the linear order \(\eta^f\). To simplify the notation, given any \( r \in \R \) we write \(\eta^f_r\) in place of \(\eta^f_{(r,r+1)}\). Let \(\varphi \colon (\R^+)^\N \rightarrow \LO\) be the Borel map given by 
\[
\varphi(\vec{x})= \eta^f \boldsymbol{\omega}^* + \sum_{n \in \N } (\eta^f_{-(n+1)} + \eta^f_{x_n}).
\] 
We claim that $\varphi$ reduces \( E_1 \restriction (\R^+)^\N \) to \(\eq\). 

Suppose that \( \vec{x},\vec{y} \in (\R^+)^\N \) are such that \(\vec{x} \mathrel{E}_1 \vec{y}\), and let \(n_0 \in \N\) be such that \(x_n=y_n\) for all \(n\geq n_0\). Let \( m = 2n_0+2 \), so that in particular \( \mathbf{m} \in \Fin \subseteq\L\). 
Consider the \( \mathbf{m}\)-convex partition \( (L_k)_{k < m}\) of \( \varphi(\vec{x}) \) given by
\[
L_k = 
\begin{cases}
\eta^f \boldsymbol{\omega}^* & \text{if } k = 0 \\
\eta^f_{-(i+1)} & \text{if } k = 2i+1 \text{ for some } i < n_0 \\
\eta^f_{x_i} & \text{if } k = 2i+2 \text{ for some } i < n_0 \\
\sum_{ n \geq n_0} (\eta^f_{-(n+1)} + \eta^f_{x_n}) & \text{if } k = 2n_0 + 1 .
\end{cases}
\]
We now define an embedding \( g \colon \varphi(\vec{x}) \to \varphi(\vec{y}) \) as follows. 
First send \( L_0 \) into the \( \varphi(L') \)-convex set \( \{ (\ell,j) \in \eta^f \boldsymbol{\omega}^* \mid j \leq_{\boldsymbol{\omega^*}} 2n_0 \} \csube \eta^f \boldsymbol{\omega}^*  \) of \( \varphi(\vec{y}) \) by traslating each summand of \( L_0 \) to the left by \( 2n_0\)-many places. 
Then send each \( L_k \) with \( 0 < k \leq 2n_0 \) into the summand \( \eta^f \times \{ 2n_0 - k \} \csube \eta^f \boldsymbol{\omega}^* \csube \varphi(\vec{y}) \) in the obvious way, using the fact that \( L_k \) is the restriction of \( \eta^f \) to an open interval.
Finally,  map \( L_{2n_0+1} \) identically to itself (viewed as a tail of \( \varphi(\vec{y})\)), which is possible because \( \eta^f_{x_n}= \eta^f_{y_n} \) for all \( n \geq n_0 \) by choice of \( n_0 \). 
Then \( g(L_k) \csube \varphi(\vec{y}) \) for every \( k < m \), hence \( \varphi(\vec{x}) \trianglelefteq^{\L}_\LO \varphi(\vec{y}) \) as witnessed by \( \mathbf{m} \), \( (L_k)_{k < m } \) and \( g \).

Conversely, suppose that \( \varphi(\vec{x}) \eq \varphi(\vec{y}) \), and fix some \( K \in \L \), a \( K \)-convex partition \( (L_k)_{k \in K} \) of \( \varphi(\vec{x} )\), and an embedding \( g \colon \varphi(\vec{x}) \to \varphi(\vec{y})\) witnessing \( \varphi(\vec{x}) \qo \varphi(\vec{y} )\). 
By Lemma~\ref{lem:loc_not_scat}\ref{lem:loc_not_scat-a} for each \( n \in \N \) there are \( -(n+1) \leq q_0^{(n)} < q_1^{(n)} \leq -n \) such that \( M_n =  \eta^f_{(q_0^{(n)},q_1^{(n)})} \csube \eta^f_{-(n+1)} \cap L_k \) for some \( k \in K \), so that \( g \) itself witnesses \( M_n \cvx \varphi(\vec{y}) \). 
By Lemma~\ref{lem:loc_not_scat}\ref{lem:loc_not_scat-b} and injectivity of \( f \) (and using also \( \vec{x},\vec{y} \in (\R^+)^\N\)), it easily follows that either \( g(M_n) = M_n \times \{ j_n \} \csube \eta^f \boldsymbol{\omega}^* \csube \varphi(\vec{y}) \) for some \( j_n \in \boldsymbol{\omega}^*\), or else \( g(M_n) = M_n \csube \eta^f_{-(n+1)} \csube \varphi(\vec{y}) \). 
But since \( M_n <_{\varphi(\vec{x})} M_m \) for all \( n,m \in \N \) such that \( n < m \), if the first case occur  for both \( M_n \) and \( M_m\) then \( j_n <_{\boldsymbol{\omega}^*} j_m \) (equivalently: \( j_n > j_m \)) because otherwise \( g(M_m) =  M_m \times \{ j_m \} <_{\varphi(\vec{y})} M_n \times \{ j_n\} = g(M_n) \). 
(Here we use that if \( m > n \) then \( (-(m+1),-m) <_\R (-(n+1),-n) \).) 
This means that the first case can occur only for finitely many \( n \in \N \).
On the other hand, if the second case occurs for some \( M_n \), then it also occurs for all \( M_m \) with \( m \geq n \) because \( g \) is order-preserving.
Combining these two facts, we obtain that there is \( n_0 \in \N \) such that for all \( n \geq n_0 \) the second case, namely \( g(M_n) = M_n \csube \eta^f_{-(n+1)} \csube \varphi(\vec{y}) \),  occurs: we claim that \( x_n = y_n \) for every \( n \geq n_0 \), so that \( \vec{x} \mathrel{E}_1\ \vec{y} \). 
Suppose towards a contradiction that \( x_n \neq y_n \) for some \( n \geq n_0 \). Since \( M_n <_{\varphi(\vec{x})} \eta^f_{x_n} <_{\varphi(\vec{x})} M_{n+1} \), by choice of \( n_0 \) we have that \( \eta^f_{x_n} \qo \eta^f_{-(n+1)}+ \eta^f_{y_n} + \eta^f_{-(n+2)} \csube \varphi(\vec{y}) \). 
Fix \( r_0 < r_1 \) such that \( (r_0,r_1) \subseteq (x_n,x_n+1) \setminus (y_n, y_n+1) \), so that also \( \eta^f_{(r_0,r_1)} \qo \eta^f_{-(n+1)}+ \eta^f_{y_n} + \eta^f_{-(n+2)} \). 
By Lemma~\ref{lem:loc_not_scat}\ref{lem:loc_not_scat-a} again there are \( r_0 \leq q_0 < q_1 \leq r_1 \) such that \( \eta^f_{(q_0,q_1)} \cvx \eta^f_{-(n+1)}+ \eta^f_{y_n} + \eta^f_{-(n+2)} \). 
Applying Lemma~\ref{lem:loc_not_scat}\ref{lem:loc_not_scat-b} we get the desired contradiction because \( f \) is injective and \( (q_0,q_1) \cap [(-(n+1),-n) \cup (y_n,y_n+1) \cup (-(n+2), -(n+1))] = \emptyset \).
\end{proof}

\begin{cor}\label{cor:bor_compl_eq}
Let $\L \subseteq \Scat$ be ccs and different from $\{ \1 \}$. Then ${\iso} <_B {\eq}$, and moreover ${\eq} \nleq_{\text{\scriptsize \textit{Baire}}} {\cvxeq}$ and ${\eq} \nleq_{\text{\scriptsize \textit{Baire}}} {E}$ for every orbit equivalence relation $E$.
\end{cor}

\begin{proof}
Theorem~\ref{thm:iso_breduces_eq} gives \( {\iso} \leq_B {\eq}\), while the non-reducibility results follow at once from Theorem~\ref{thm:red_E1_eq} combined with Theorems~\ref{E1_orbit} and \ref{thm:complexity_cvxeq}\ref{thm:complexity_cvxeq-b}.
\end{proof}

Along the same lines, we have:

\begin{prop}
For every ccs $\L \subseteq \Scat$ we have ${\eq} \mathrel{\nleq_B} {\equiv_{\LO}}$, and thus ${\qo} \mathrel{\nleq_B} {\emb}$.
\end{prop}

\begin{proof}
By Theorem~\ref{thm:iso_breduces_eq} the identity on $\R$ Borel reduces to $\eq$, while by Laver's classic result it does not Borel reduce to $\equiv_{\LO}$.
\end{proof}

We do not know whether $\eq$ is complete for analytic equivalence relations when $\{ \1 \} \subsetneq \L \subsetneq \Lin$ is Borel. It remains also open whether $\eq$ is proper $\mathbf{\Sigma}^1_2$ in the case of a $\mathbf{\Pi}^1_1$-complete class $\L$.\smallskip

We also notice that Lemma~\ref{lem:open_int_qo} can be turned into a Borel reducibility result concerning the quasi-order $\qo$. 
Recall that \(\Int(\R)\) is the set of all open intervals of \(\R\), which can be naturally equipped with a Polish topology: indeed, if we extend the usual order on $\R$ to $\R \cup \{\pm \infty\}$ in the obvious way, then $\Int(\R)$ can be identified with the open subset $\{(x,y) \mid x < y\}$ of the Polish space $(\R\cup \{\pm \infty\})^2$.  
The inclusion relation on $\Int(\R)$ is then closed, and since the embedding from \((\Int(\R), \subseteq)\) to \((\LO,\qo)\) defined in the proof of Lemma \ref{lem:open_int_qo} is actually a Borel reduction we get:
	
\begin{prop}\label{prop:red_reals_lo}
  For every ccs class $\L \subseteq \Scat$, \((\Int(\R),\subseteq) \leq_B (\LO,\qo)\).
\end{prop}

We conclude this section by studying what happens if we move to \emph{coloured} linear orders. Consider the Polish space \(\LO_{\N} = \LO \times \N^{\N}\).
Each element \( (L,c) \in \LO_\N \) can be interpreted as the linear order \( L \) on \( \N \) where each of its elements \( \ell \in L \) is coloured with \( c(\ell) \).

\begin{defn}\label{def:qo_for_coloured}
Let $\L \subseteq \Lin$ and \((L,c), (L',c') \in \LO_{\N}\). We say that \((L,c)\) is \textbf{$\L$-convex 
embeddable} in \((L',c')\), in symbols \((L,c) \qon (L',c')\), if and only if for some embedding \(f \colon L\to L'\) witnessing $L \qo L'$ we have \(c'(f(n))=c(n)\) for every \(n \in \N\). When \( \L = \{ \1 \} \) we just write \( \trianglelefteq_{\LO_\N}\) instead of \( \trianglelefteq^{\{ \1 \}}_{\LO_\N}\), while if \( \L = \Lin \) we write \( \preceq_{\LO_\N} \) instead of \(\trianglelefteq^{\Lin}_{\LO_\N}\).
\end{defn} 

Notice that \( \qon \) is always reflexive, and it is transitive (i.e.\ a quasi-order) if and only is so is \( \qo \), i.e.\ if and only if \( \L \) is ccs.

Marcone and Rosendal~\cite{MR04} showed that the quasi-order \(\preceq_{\LO_\N}\) of embeddability between coloured linear orders is complete for analytic quasi-orders, and thus \( {\trianglelefteq_\LO^{\Lin}} = {\emb} <_B {\preceq_{\LO_\N}} =  {\trianglelefteq^{\Lin}_{\LO_\N}} \). In contrast, when considering ccs families $\L \neq \Lin$, we have the opposite situation.

\begin{thm} \label{thm:coloured}
If $\L \subseteq \Scat$ is ccs, then \({\qon} \sim_B {\qo}\).
\end{thm}

\begin{proof} 
Clearly, \({\qo} \leq_B {\qon}\) via the reduction \( L \mapsto (L,c) \) with \( c \) the constant map with value \( 0 \).
For the converse, let \( \LO'_\N \) be the collection of those linear orders \( (L,c) \) such that \( c(\ell) > 0 \) for all \( \ell \in L \) and \( c \) is not constant on any closed interval \( [\ell_0,\ell_1]_L \) with \( \ell_0 <_L \ell_1 \). Notice that the Borel map \( (L,c) \mapsto (\mathbf{2}L,c') \) with \( c'(0,\ell) = c(\ell) + 2 \) and \( c'(1,\ell) = 1 \) for all \( \ell \in L \) reduces \( \qon \) to \( \qon \restriction \LO'_\N \), so it is enough to show that \( {\qon \restriction \LO'_\N} \leq_B {\qo} \).

Consider the Borel map \( \varphi \colon \LO'_\N \to \LO \) defined by \( \varphi(L,c) = \sum_{\ell \in L} \eta^{f_\ell} \), where \( f_\ell \) is the constant map with value \( c(\ell) \) (viewed as a finite linear order). We claim that \( \varphi \) reduces \( \qon \restriction \LO'_\N \) to \( \qo \).

One direction is obvious, so let us assume that \( \varphi(L,c) \qo \varphi(L',c') \), as witnessed by \( K \in \L \), the \( K \)-convex partition \( (M_k)_{k \in K} \) of \( \varphi(L,c) \) and the embedding \( g \colon \varphi(L) \to \varphi(L') \). 
We follow the strategy used in the proof of Theorem~\ref{thm:fractal}, although in a simplified situation. 
By Lemma~\ref{lem:loc_not_scat}\ref{lem:loc_not_scat-a}, for every \( \ell \in L \) we can fix \( k_\ell \in K \) and \( N_\ell = \eta^{f_\ell}_{(q_0^{(n)},q_1^{(n)})} \csube \eta^{f_\ell} \cap M_{k_\ell} \), so that \( N_\ell \iso \eta^{f_\ell}\) and \( N_\ell \cvx \varphi(L',c') \) as witnessed by \( g \) itself. 
Since \( (L',c') \in \LO'_\N \) is not constant on any closed interval,
by Lemma~\ref{lem:loc_not_scat}\ref{lem:loc_not_scat-b} there is a (necessarily unique) \( \ell' \in L' \) such that \( g(N_\ell) \subseteq \eta^{f_{\ell'}}\csube \varphi(L',c')\) and \( f_{\ell'} \) has the same value of \( f_\ell \): we claim that the map \( h \colon L \to L' \) defined by setting \( h(\ell) = \ell' \) is a colour-preserving embedding. 
It is clearly order-preserving because so is \( g \). 
If there were \( \ell_0, \ell_1 \in L \) with \(\ell_0 <_L \ell_1 \) and \( h(\ell_0) = h(\ell_1) \), then \( f_\ell \) would have the same value as \( f_{h(\ell_0)}\) for all \( \ell \in [\ell_0,\ell_1]_L \), contradicting \( (L,c) \in \LO'_\ell \). 
Therefore \( h \) is also injective, and it is colour-preserving because \( f_\ell \) and \( f_{h(\ell)}\) have the same value for all \( \ell \in L \).

For each \( k \in K \) set \( L_k = \{ \ell \in L \mid k_\ell = k \} \) and \( K' = \{ k \in K \mid L_k \neq \emptyset \} \in \L \). 
Observe that \( (L_k)_{k \in K'} \) is a \( K' \)-convex partition of \( L \): indeed, since \( \ell_0 <_L \ell_1 \iff N_{\ell_0} <_{\varphi(L,c)} N_{\ell_1} \), for all \( k,k' \in K' \) we have
\begin{multline*}
k <_{K'} k' \iff M_k <_{\varphi(L,c)} M_{k'} \iff \\
\forall \ell_0 \in L_k \forall \ell_1 \in L_{k'}\,(N_{\ell_0} <_{\varphi(L)} N_{\ell_1}) \iff L_k <_L L_{k'}.
\end{multline*}

We also claim that each \( h(L_k) \) is \( L' \)-convex. 
Fix \( \ell_0,\ell_1 \in L_k \) such that \( h(\ell_0) <_{L'} h(\ell_1) \). Since \( g \restriction M_k \) is an isomorphism between \( M_k\) and \( g(M_k)\), then the corresponding restriction of \( g^{-1} \) witnesses \( \sum_{h(\ell_0) <_{L'} \ell' <_{L'} h(\ell_1)} \eta^{f_{\ell'}} \cvx \sum_{\ell_0 \leq_L \ell \leq_L \ell_1} \eta^{f_\ell} \). 
By Lemma~\ref{lem:loc_not_scat}\ref{lem:loc_not_scat-b}, for each \( \ell' \in (h(\ell_0),h(\ell_1))_{L'}\) there is \( \ell \in [\ell_0,\ell_1]_L \) such that \( g^{-1}(\eta^{f_{\ell'}}) \subseteq \eta^{f_\ell} \).
Since \( (L',c') \in \LO'_\N \), we cannot have \( \ell = \ell_0 \) because otherwise \( c' \) would be constant on \( [h(\ell_0), \ell']_{L'} \), and \( \ell \neq \ell_1 \) as well because otherwise \( c' \) would be constant on \( [\ell',h(\ell_1)]_{L'} \). 
Hence \( \ell_0 <_L \ell <_L \ell_1 \), which implies \( \eta^{f_\ell} \csube M_k \), so that necessarily \( \ell \in L_k \) and \( \eta^{f_\ell} \cvx \varphi(L',c')\) via \( g \restriction \eta^{f_\ell} \). 
By Lemma~\ref{lem:loc_not_scat}\ref{lem:loc_not_scat-b} again, we have that \( g(\eta^{f_\ell}) \csube \eta^{f_{\ell''}} \) for some unique \( \ell'' \in L' \), so that in particular \( h(\ell) = \ell'' \). 
Since  \( g^{-1}(\eta^{f_{\ell'}}) \subseteq \eta^{f_\ell}\) we must have \( \ell'' = \ell' \), so that \( \ell' \in h(L_k) \), as desired.

We have shown that \( K' \), the \( K' \)-convex partition \( (L_k)_{k \in K'} \) of \( L \), and the embedding \( h \) witness \( (L,c) \qon (L',c') \), hence we are done.
\end{proof}

Recalling that $\cvx$ is not complete for analytic quasi-orders (\cite[Corollary 3.28]{IMMRW22}), we obtain the following result, which is in contrast with the situation for \( \preceq_{\LO_\N} \) (\cite{MR04}).

\begin{cor}
    The relation $\trianglelefteq_{\LO_\N}$ of convex embeddability between coloured linear orders is not complete for analytic quasi-orders.
\end{cor}

\section{Examples of ccs families of countable linear orders} \label{sec:examples}

Recall from Theorem~\ref{thm:ccs_trans} that if \( \L \subseteq \Lin \) is downward \( \preceq \)-closed, then \( \qo \) is a quasi-order precisely when \( \L \) is ccs.
In this section we provide some natural examples of ccs classes $\L \subseteq \Lin$, and compare the resulting $\L$-convex embeddability relations \( \qo \) using Borel reducibility.

We first notice that the collection of ccs classes is not closed under unions. Indeed, let $\L=\WO \cup \WO^*$. It is easy to check that $\WO$ and $\WO^*$ are ccs classes. Let now \(K=\boldsymbol{\omega}\), \(K'=\boldsymbol{\omega}^*\), and set \(K'_0 = \boldsymbol{\omega}^*\) and \(K'_k=\{\max \boldsymbol{\omega}^*\}\) for every \(k > 0\). Then \(\sum_{k \in K} K'_k \cong \boldsymbol{\omega}^* + \boldsymbol{\omega} \cong \z \notin \L\), and hence $\L$ is not ccs. 

On the other hand, it is immediate to see that the collection of ccs classes is closed under unions of chains (with respect to inclusion) and under arbitrary intersections.

\begin{remark}\label{rem:non_ccs}
In the previous discussion and in Example~\ref{ex:fin_wo_scat}, to show that certain classes are not ccs we used the following fact: If $\L$ is ccs and $L+\1, \1+L' \in \L$, then $L + \1 + L'$ belongs to $\L$ as well. 
However the latter condition is not equivalent to being ccs, as witnessed by the following example.
Let $\L=\{L \in \LO \mid \z \boldsymbol{\o} \npreceq L \land \z\boldsymbol{\o}^* \npreceq L\}$. It is immediate that if $L+\1$ and $\1+L'$ are elements of $\L$, then $L+\1+L'$ is in $\L$ as well. On the other hand $\boldsymbol{\o}^2$ and $\boldsymbol{\o}^*\boldsymbol{\o}$ belong to $\L$, but there is a convex sum using $K=\boldsymbol{\o}^2$ and $K'= \boldsymbol{\o}^*\boldsymbol{\o}$ which is isomorphic to $\z\boldsymbol{\o} \notin \L$.
\end{remark}

We already observed that \( \{ \1 \}  \subseteq  \Fin  \subseteq \WO \subseteq \Scat \subseteq \Lin \) are all ccs classes. On the other hand, there is no intermediate class between \( \Scat \) and \( \Lin \), while there are no ccs classes between \( \{ \1 \} \) and \( \Fin \) by Example~\ref{ex:fin_wo_scat}. So if we are looking for new ccs classes, it is natural to consider the well-known stratifications of \( \WO \) and \( \Scat \), and determine which of their levels are ccs. We start from the latter, exploiting an equivalent definition of Hausdorff rank which involves linear orders of the form \( \Z^\gamma\), called powers of \( \Z \).

When $\g$ is an ordinal, we can define $\Z^\g$ in two equivalent ways: by induction on $\g$ (\cite[Definition 5.34]{Ros82}), or by explicitly defining a linear order on a certain set (\cite[Definition 5.35]{Ros82}) --- the latter can actually be used to define \(\Z^L\) for any linear order \(L\). Here are the two definitions. 
	
\begin{defn} \label{def:Z^alpha}
  \begin{enumerate-(a)}
  \item \(\Z^0 = \boldsymbol{1}\),
  \item \(\Z^{\g+1}=(\Z^\g \boldsymbol{\o})^* + \Z^\g + \Z^\g \boldsymbol{\o}\),
  \item
    \(\Z^\g = \big(\sum_{\beta < \g} \Z^\beta \boldsymbol{\o}\big)^* +
    \boldsymbol{1} + \sum_{\beta < \g} \Z^\beta \boldsymbol{\o}\) if $\g$ is limit.
  \end{enumerate-(a)}
\end{defn}

\begin{defn}\label{pow_z_def_theor}
  Let \(L\) be a linear order. For any map \(f \colon L \to \Z\), we
  define the support of \(f\) as the set
  \(\Supp(f)=\{n \in L \mid f(n) \neq 0\}\) . The \(L\)-power of
  \(\Z\), denoted by \(\Z^L\), is the linear order on
  \(\{f\colon L \to \Z \mid \Supp(f) \text{ is finite}\}\) defined by \(f \leq_{\Z^L}g\) if and only if \(f=g\) or
  \(f(n_0) <_{\Z} g(n_0)\), where
  \(n_0 = \max \{n \in \Supp(f) \cup \Supp(g) \mid f(n) \neq g(n)\}\).
\end{defn}
	
We need the following properties of $\Z^L$ (see~\cite[Section~3.2]{CCM19}).
	
\begin{prop}\label{power_Z}
For all ordinals \(\a<\g\), we have
  \[
    \Z^\g \cong \bigg(\sum_{\b < \g} \Z^\b \boldsymbol{\o}\bigg)^* + \1 + \sum_{\b < \g} \Z^\b \boldsymbol{\o} \cong \bigg(\sum_{\a \leq \b < \g} \Z^\b
    \boldsymbol{\o}\bigg)^* + \bigg(\sum_{\a \leq \b < \g} \Z^\b
    \boldsymbol{\o}\bigg).
  \]
Moreover, if \(L\) is countable and not a well-order, then \(\Z^L\cong \Z^\g \eta \) for some countable ordinal \(\g\).
\end{prop}
 
Given \(L \in \LO\), we now consider the classes $\L_{\prec \Z^L}$ and $\L_{\preceq \Z^L}$. First of all, if \(L\) is not a well-order then \(\Z^L\) is not scattered  by Proposition \ref{power_Z},  and so we obtain the ccs classes $\L_{\prec \Z^L} = \Scat$ and $\L_{\preceq \Z^L}= \Lin$.
We can thus restrict our attention to the case \(L \in \WO\). Let $\g$ be the order type of $L$. 
	    
Using Proposition~\ref{power_Z}, if we let $L = \sum_{\b<\g} \Z^\b \boldsymbol{\o}$, we have $\Z^\g \cong L^* + {\1} + L$.
Then Remark~\ref{rem:non_ccs} implies that $\L_{\prec \Z^\g}$ is not ccs, because $L^* + \1$ and ${\1} + L$ belong to $\L_{\prec \Z^\g}$, yet $L^* + {\1} + L \notin \L_{\prec \Z^\g}$. To complete our analysis, it remains to prove that $\L_{\preceq \Z^\g}$ is ccs for every countable ordinal $\g$, so that $\L_{\preceq \Z^L}$ is ccs for every countable linear order $L$.
First we need a technical lemma. 

\begin{lem}\label{lem:sum_bounded_sets_z}
Let $\g>0$ be a countable ordinal. 

\begin{enumerate-(a)}
\item \label{lem:sum_bounded_sets_z-a}
If $A \subseteq \Z^\g$ is bounded from below, then $A \preceq \sum_{\b < \g} \Z^\b \boldsymbol{\o}$. Symmetrically, if $A$ is bounded from above, then $A \preceq \left(\sum_{\b < \g} \Z^\b \boldsymbol{\o}\right)^*$.
\item \label{lem:sum_bounded_sets_z-b}
Assume that \( A \subseteq \Z^\g \) is bounded (from both sides). If \( \gamma = \beta+1 \) is successor, then \( A \preceq \Z^\beta \times n \) for some \( n \in \N\setminus \{ 0 \} \); if instead \( \gamma \) is limit, then \( A \preceq \Z^\alpha \) for some \( \alpha < \gamma \).
\end{enumerate-(a)}
\end{lem}

\begin{proof}
If $\g = \b + 1$ we have $\Z^\g \cong \Z^\b \z$, and thus both part~\ref{lem:sum_bounded_sets_z-a} and part~\ref{lem:sum_bounded_sets_z-b} easily follow. 
        
Let now $\g$ be limit. If $A$ is bounded from below, then there exists $\a < \g$ such that
\begin{align*}
A \preceq \bigg(\sum_{\b < \a} \Z^\b \boldsymbol{\o}\bigg)^* + \1 + \sum_{\b < \g} \Z^\b \boldsymbol{\o} & = \bigg(\sum_{\b < \a} \Z^\b \boldsymbol{\o}\bigg)^* + \1 + \sum_{\b < \a} \Z^\b \boldsymbol{\o} + \sum_{\a \leq \b < \g} \Z^\b \boldsymbol{\o}\\
& \cong \Z^\a + \sum_{\a \leq \b < \g} \Z^\b \boldsymbol{\o} \cong \sum_{\a \leq \b < \g} \Z^\b \boldsymbol{\o} \preceq \sum_{\b < \g} \Z^\b \boldsymbol{\o},
\end{align*}
where we are using Proposition~\ref{power_Z} to justify the first and third step. The case where \( A \) is bounded from above is similar. Finally, if \( A \) is bounded from both sides, then there is \( \alpha < \gamma \) such that \( A \preceq \bigg(\sum_{\b < \a} \Z^\b \boldsymbol{\o}\bigg)^* + \1 + \sum_{\b < \a} \Z^\b \boldsymbol{\o} \cong \Z^\a \), where the last step is again justified by Proposition~\ref{power_Z}.
\end{proof}

\begin{prop}\label{prop:classes_zeta}
$\L_{\preceq \Z^\g}$ is ccs for every countable ordinal $\g$.
\end{prop}

\begin{proof}
We argue by induction on $\g$. If \(\g = 0\) we have \(\Z^\g = \1\) and thus $\L_{\preceq \Z^\g} =\{\1\}$ is ccs.
			
Now fix $\g \geq 1$ and assume that \(\L_{\preceq \Z^\b}\) is ccs for every $\b < \g$. Consider $K, K' \in \L_{\preceq \Z^\g}$ and nonempty convex subsets \((K'_k)_{k \in K}\) of \(K'\) such that \(\forall k,k' \in K\ (k <_K k' \rightarrow K'_k \leq_{K'} K'_{k'})\). We want to show that \(\sum_{k \in K} K'_{k} \preceq \Z^\g\). It is convenient to think of $K$ and $K'$ as subsets of $\Z^\g$.
        
We assume that $K$ has a minimum but no maximum: the other cases (no extrema, maximum only, and both extrema) can be treated similarly. Pick a sequence $\{k_i \mid i \in \N\}$ cofinal in $K$ with $k_0 = \min K$.
For every $i \in \N$ let $B_i = \{k \in K \mid k_i <_K k \leq_K k_{i+1}\}$. Then 
\[
\sum_{k \in K} K'_k = K'_{k_0} + \sum_{i \in \N} \bigg( \sum_{k \in B_{i}}  K'_{k}\bigg)
\]
Since $K'_{k_0}$ is bounded from above in $\Z^\g$, $K'_{k_0} \preceq \left(\sum_{\b < \g} \Z^\b \o\right)^*$  by Lemma~\ref{lem:sum_bounded_sets_z}\ref{lem:sum_bounded_sets_z-a}. 

Assume first that \( \g = \b+1\) is a successor ordinal. Since $B_i$ and $\bigcup_{k \in B_i} K'_k$ are bounded in $\Z^\g$, by Lemma~\ref{lem:sum_bounded_sets_z}\ref{lem:sum_bounded_sets_z-b} for each \( i \in \N \) we can find \( n_i \in \N \setminus \{ 0 \} \) such that both $B_i \preceq \Z^{\b} n_i$ and $\bigcup_{k \in B_i} K'_k \preceq \Z^{\b} n_i$. 
We decompose $B_i = \bigcup_{j<n_i} B_{i,j}$ so that $B_{i,j} <_K B_{i,j+1}$ and \(\emptyset \neq B_{i,j} \preceq \Z^{\b}\) for every $j<n_i$. 
For each such $j$, we can then write $\bigcup_{k \in B_{i,j}} K'_k = \bigcup_{h<\ell_{i,j}} K'_{i,j,h}$ for some $\ell_{i,j} \leq n_i$ so that $K'_{i,j,h} <_{K'} K'_{i,j,h+1}$ and $\emptyset \neq K'_{i,j,h} \preceq \Z^{\b}$ for every $h < \ell_{i,j}$. 
For every $j < n_i$ and $h<\ell_{i,j}$, let $K_{i,j,h} = \{k \in B_{i,j} \mid K'_{i,j,h} \cap K'_k \neq \emptyset\} \preceq \Z^\b$, and for each $k \in K_{i,j,h}$ set $K'_{i,j,h,k} = K'_k \cap K'_{i,j,h}$. 
It is clear that $K'_{i,j,h,k} \leq_{K'} K'_{i,j,h,k'}$ whenever $k,k' \in K_{i,j,h}$ are such that $k <_K k'$. 
We can therefore apply the induction hypothesis to $\b<\g$ and obtain
\[
\sum_{k \in K_{i,j,h}} K'_{i,j,h,k} \preceq \Z^{\b}.\tag{$\star$}
\]
Since if $k \in B_i$ every element of $K'_k$ belongs to some $K'_{i,j,h,k}$ with $k \in K_{i,j,h}$, we can write
\[
\sum_{k \in B_{i}} K'_{k} = \sum_{j < n_{i}} \bigg( \sum_{h<\ell_{i,j}} \bigg(\sum_{k \in K_{i,j,h}} K'_{i,j,h,k}\bigg) \bigg).
\]
Then by $(\star)$ and $\ell_{i,j} \leq n_i$ we have $\sum_{k \in B_{i}} K'_k \preceq \Z^{\b} n_{i}^2$, so that
 $\sum_{i \in \N} \left( \sum_{k \in B_{i}}  K'_{k}\right) \preceq  \sum_{\b < \g} \Z^\b \boldsymbol{\o}$. Therefore
\[
\sum_{k \in K} K'_k = K'_{k_0} + \sum_{i \in \N} \bigg( \sum_{k \in B_{i}}  K'_{k}\bigg) \preceq \bigg(\sum_{\b < \g} \Z^\b \o\bigg)^* + \sum_{\b < \g} \Z^\b \o \cong \Z^\g,
\]
where in the last step we are using Proposition~\ref{power_Z}.

If \( \gamma \) is limit, the proof becomes even easier. Indeed, since both $B_i$ and $\bigcup_{k \in B_i} K'_k$ are bounded in $\Z^\g$, by Lemma~\ref{lem:sum_bounded_sets_z}\ref{lem:sum_bounded_sets_z-b} for each \( i \in \N \) there is \( \a_i < \g \) such that \( B_i \preceq \Z^{\a_i} \) and \( \bigcup_{k \in B_i} K'_k \preceq \Z^{\a_i} \). By inductive hypothesis, we get \( \sum_{k \in B_i} K'_k  \preceq \Z^{\a_i} \). 
Since we can clearly assume that the ordinals \( \alpha_i \) are increasing, it follows that \( \sum_{i \in \N} \bigg( \sum_{k \in B_{i}}  K'_{k}\bigg) \preceq \sum_{\b < \g} \Z^\b \boldsymbol{\o}\) and thus by Proposition~\ref{power_Z} again
\[
\sum_{k \in K} K'_k = K'_{k_0} + \sum_{i \in \N} \bigg( \sum_{k \in B_{i}}  K'_{k}\bigg) \preceq \bigg(\sum_{\b < \g} \Z^\b \boldsymbol{\o}\bigg)^* + \sum_{\b < \g} \Z^\b \boldsymbol{\o} \cong \Z^\g. \qedhere
\]
\end{proof}

Let us now come back to the stratification of \( \Scat \) given by the Hausdorff rank.
By~\cite[Theorem 5.37]{Ros82}, 
\[
\L_{\preceq \Z^\g} = \{L \in \Lin \mid \rk(L) \leq \g\},
\]
where $\rk$ denotes the Hausdorff rank.
Using Proposition~\ref{prop:classes_zeta} and the closure of ccs classes under unions of chains, we thus obtain:

\begin{cor}
The class $\{L \in \Lin \mid \rk(L) < \g\}$ is ccs for every countable ordinal $\g>0$.
\end{cor}

Moving to the stratification of \( \WO \), we now discuss which classes of well-orders are ccs. 
By downward $\preceq$-closure, these classes must be of the form \(\L_{\prec \gg}\) for some ordinal $\g$.
Recall that an \textbf{(additively) indecomposable ordinal} \(\g\) is any nonzero ordinal such that for any \(\a , \b < \g\), we have \(\a + \b < \g\). 
The indecomposable ordinals are precisely those of the form \(\omega^\delta\) for some ordinal \(\delta\), and thus they are limit ordinals if \( \delta > 0 \). 
From the normality of addition in its right argument, it follows that \(\g\) is indecomposable if and only if \(\alpha +\g =\g\) for every $\a<\g$. We use these properties in the proof of the following proposition.
	
\begin{prop}\label{prop:classes_gamma}
Let $\g$ be an infinite countable ordinal. Then \(\L_{\prec \gg}\) is ccs if and only if \(\g\) is either an indecomposable ordinal, or the successor of an indecomposable ordinal. 
\end{prop}

\begin{proof}
\(\Rightarrow)\) We show the contrapositive.
First we consider the case \(\g = \a + 1\) assuming that \(\a\) is not indecomposable. Fix \(\b<\a\) such that \(\a<\b + \a\). 
To show that \(\L_{\prec \gg}\) is not ccs we consider \(\boldsymbol{\b}+\1, \aa \in \L_{\prec \gg}\). 
For \( k \in \boldsymbol{\b}+1 \), define the following convex subsets of \(\aa\):
\[
\aa_{k}= \{\min \aa\} \text{ if } k<_{\boldsymbol{\b+1}} \max \{\boldsymbol{\b+1}\}, \qquad \text{and} \qquad \aa_{\max \{\boldsymbol{\b+1}\}}=\aa.
\]
Then \(\sum_{k \in \boldsymbol{\b+1}} \aa_k \cong \boldsymbol{\b} + \aa \notin \L_{\prec \gg}\).
		
It remains to prove that \(\L_{\prec \gg}\) is not ccs when $\g$ is limit but not indecomposable. Fix \(\alpha, \beta<\g\) such that \(\g =\alpha + \beta\). Since $\g$ is limit, then \(\beta\) is limit as well and in particular $\beta>1$, so that \(\alpha + 1 < \g\). Consider \(\bb, \aa+\1 \in \L_{\prec \gg}\), and for \( k \in \boldsymbol{\b} \) define the following convex subsets of \(\aa+\1\):
\[
\aa_{\min \bb}= \aa+\1, \qquad \text{and} \qquad \aa_{k}=\{\max\{\aa+\1\}\} \text{ if } k>_{\bb} \min \bb.
\]
Then \(\sum_{k \in \bb} \aa_k \cong \aa + \bb \cong \gg\) does not belong to \(\L_{\prec \gg}\).
		
\medskip
    
\(\Leftarrow)\)
Suppose first that \( \gamma = \omega^\delta+1\) for some \( \delta > 0 \). Since the largest ordinal that embeds into \( \Z^\d\) is precisely \( \o^\d\), it follows that \( \L_{\prec \gg} = \L_{\preceq \boldsymbol{\o^\d}} =  \L_{\preceq \Z^\d} \cap \WO \), and hence it is ccs by Proposition~\ref{prop:classes_zeta} and closure of the family of ccs classes under intersections. 

Now suppose that \( \g = \o^\d \) is indecomposable. 
Given \( \a,\b < \o^\d \) and a family \( (\boldsymbol{\b}_k)_{k \in \boldsymbol{\a}}\) of convex subsets of \( \boldsymbol{\b}\) such that \( \boldsymbol{\b}_{k_0} \leq_{\boldsymbol{\b}} \boldsymbol{\b}_{k_1} \) for all \( k_0,k_1 \in \boldsymbol{\a} \) with \( k_0 <_{\boldsymbol{\a}} k_1 \), we want to show that \( \sum_{k \in \boldsymbol{\a}} \boldsymbol{\b}_k \) has order type smaller than \( \o^\d \). 
For each \( k \in \boldsymbol{\a} \), let \( \boldsymbol{\b}'_k = \boldsymbol{\b}_k \setminus \{ \min \boldsymbol{\b}_k \} \).
(Some \( \boldsymbol{\b}'_k\) might be empty.) 
Then \( \boldsymbol{\b}'_{k_0} <_{\boldsymbol{\b}} \boldsymbol{\b}'_{k_1}\) for all \( k_0,k_1 \in \boldsymbol{\a}\) such that \( k_0 <_{\boldsymbol{\a}} k_1 \) and \( \boldsymbol{\b}'_{k_0}, \boldsymbol{\b}'_{k_1} \neq \emptyset \), and moreover \( \bigcup_{k \in \boldsymbol{\a}} \boldsymbol{\b}'_k \subseteq \bigcup_{k \in \boldsymbol{\a}} \boldsymbol{\b}_k \subseteq \boldsymbol{\b}\). 
From these two facts it follows that if \( \b'_k \) is the order type of \( \boldsymbol{\b}'_k\) and we set \( \b' =  \sum_{k < \a} \b'_k \), we have \( \b' \leq \b < \o^\d \).
Since \( \sum_{k \in \boldsymbol{\a}} \boldsymbol{\b}_k \preceq \sum_{k \in \boldsymbol{\a}} (\1 + \boldsymbol{\b}'_k)\), it is enough to show that \( \sum_{k < \a} (1 + \b'_k) < \o^\d \).
Consider the Hessenberg (or natural) sum \( \a \oplus \b' \): since it is equal to the length of the maximum linearization of the wqo given by the disjoint union of \( \a \) and \( \b' \), we have \( \sum_{k < \a} (1 + \b'_k) \leq \a \oplus \b' \). 
On the other hand, since both \( \a \) and \( \b' \) are smaller than \( \o^\d \), then also \( \a \oplus \b' < \o^\d \), as it can be easily checked looking at their Cantor normal forms. 
Therefore \( \sum_{k < \a} (1 + \b'_k) \leq \a \oplus \b' < \o^\d \), as desired.
\end{proof}

We now establish some connections in the context of Borel reducibility among the quasi-orders $\qo$ induced by some of the ccs classes $\L$ we discussed above.
Recall that by Proposition~\ref{prop:classes_gamma} if $\g$ is an indecomposable ordinal both $\L_{\prec \boldsymbol{\gamma}}$ and $\L_{\prec \boldsymbol{\gamma+1}}$ are ccs.

\begin{thm}\label{thm:red_qog_qolg}
Let \( \gamma \) be an infinite additively indecomposable countable ordinal. Then ${\qolg_\LO} \le_B {\qog_\LO}$.
\end{thm}

\begin{proof}
We prove that the Borel map $\varphi \colon \LO \to \LO$ defined by $\varphi(L)=L+\1$ is a reduction from $\qolg_\LO$ to $\qog_\LO$.
For the nontrivial direction, suppose that $L+\1 \qog_\LO L'+\1$ with witness $\aa \in \L_{\prec \gg+\1}$, the $\aa$-convex partition $(L_k)_{k \in \aa}$ of $L+\1$, and the embedding \( h \colon L+\1 \to L' + \1 \). Let $k_0 \in \aa$ be such that $\max (L+\1) \in L_{k_0}$. Necessarily, $k_0 = \max \aa$, hence \( \alpha \) is a successor ordinal and $\a < \g$ because $\g$ is limit. Therefore the same \( \aa \in \L_{\prec \boldsymbol{\g}} \), \( (L_k)_{k \in \boldsymbol{\a}} \) and \( h \) witness $L+\1 \qolg_\LO L'+\1$, and thus \( \aa \in \L_{\prec \boldsymbol{\g}} \), \( (L_k \cap L)_{k \in \aa} \), and \( h \restriction L \) witness $L \qolg_\LO L'$.   
\end{proof}

Attempting to compare $\trianglelefteq^{\L_{\prec \bb}}$ and $\qolg$ for $\b$ and $\g$ which are far apart seems more difficult. 
We are able to show the existence of a Borel reduction only in certain cases. Recall that an ordinal $\a > 1$ is \textbf{multiplicatively indecomposable} if $\b \g < \a$ for every $\b,\g < \a$. 
It is well known that the infinite multiplicatively indecomposable ordinals are exactly those of the form $\o^{\o^\xi}$ for some ordinal $\xi$. 
Since the multiplicatively indecomposable ordinals are closed under taking suprema, for every infinite ordinal \( \gamma  \) there is a largest multiplicatively indecomposable ordinal $\b \leq \g$, that we call the \textbf{threshold} of $\g$.

\begin{remark}\label{ind_ord}
The above terminology is justified because if $\b=\o^{\o^\xi}$ is the threshold of an (additively) indecomposable ordinal \( \g \geq \b\), then $\g=\o^{\o^\xi+ \theta}$ for some ordinal $\theta$. It is then easy to check that:
\begin{enumerate-(a)}
\item\label{ind_ord_a} $\a \g = \g$, for every $0<\a < \b$;
\item\label{ind_ord_b} $\a \g > \g$, for every $\a \geq \b$.
\end{enumerate-(a)}
\end{remark}

\begin{thm}\label{thm:red_threshold}
Let $\g$ be an infinite (additively) indecomposable ordinal and let $\b$ be its threshold. 
Then ${\trianglelefteq^{\L_{\prec \bb}}_\LO} \le_B {\qog_\LO}$.
\end{thm}

\begin{proof}
Define the Borel map $\varphi \colon \LO \to \LO$ by 
\begin{equation*}
   \varphi (L) = \sum_{\a < \g} (\boldsymbol{\a} \eta + \z L).
\end{equation*}
We claim that $\varphi$ is a reduction from $\trianglelefteq^{\L_{\prec \bb}}_\LO$ to $\qog_\LO$.

Suppose that $L\trianglelefteq^{\L_{\prec \bb}}_\LO L'$ with witnesses $\boldsymbol{\xi} \in \L_{\prec \bb}$, a $\boldsymbol{\xi}$-convex partition $(L_\b)_{\b \in \boldsymbol{\xi}}$ of $L$, and an embedding $g \colon L \to L'$. 
The linear order \( K = (\1 + \boldsymbol{\xi}) \boldsymbol{\g} \) belongs to \( \L_{\prec \boldsymbol{\g}+\1}\) by Remark~\ref{ind_ord}\ref{ind_ord_a}. Let \( (M_k)_{k \in K} \) be the \( K \)-convex partition of \( \varphi(L) \) defined as follows:
\[
M_k = 
\begin{cases}
\boldsymbol{\a } \eta \times \{ \a \} & \text{if } k = ( \min (\1 + \boldsymbol{\xi}), \a) \text{ for some } \a < \g  \\
\zeta L_\b \times \{ \a \} & \text{if } k = (\b,\a) \text{ for some } \b \in \boldsymbol{\xi} \text{ and } \a < \g .
\end{cases}
\]
Finally, let \( h \colon \varphi(L) \to \varphi(L')\) be the embedding defined using \( g \) in the obvious way, that is, for every \( \alpha < \g\) and \( x \in \boldsymbol{\a} \eta + \z L \) set
\[
h(x,\a) =
\begin{cases}
(x,\a) & \text{if } x \in \boldsymbol{\a} \eta \\
((z,g(\ell)),\a) & \text{if } x = (z,\ell) \text{ for some } z \in \z \text{ and } \ell \in L.
\end{cases}
\]
Then \( K \in \L_{\prec \boldsymbol{\g}+\1} \), \((M_k)_{k \in K} \), and \( h \) witness \( \varphi(L) \qog_\LO \varphi(L') \).

Vice versa, suppose that \( \varphi(L) \qog_\LO \varphi(L') \) as witnessed by \( K \in \L_{\prec \gg+\1}\), the \( K \)-convex partition \( (M_k)_{k \in K} \) of \( \varphi(L) \), and the embedding \( h \colon\varphi(L) \to \varphi(L') \). Since \( K \) is well-ordered, for each \( \a < \g \) we can let \( k_\a \) be the \( \leq_K \)-smallest \( k \in K \) such that \( M_k \cap (\boldsymbol{\a}\eta \times \{ \a \}) \neq \emptyset \).
By choice of \( k_\a \), there is \( q \in \eta \) such that \( L_\a = \boldsymbol{\a}(-\infty,q)_\eta \times \{ \a \} \csube M_{k_\a} \). Notice that \( L_\a \cong \eta^{f_\a}\) where \( f_\a \) is constant with value \( \boldsymbol{\a} \). 

\begin{claim}\label{claim:borredord}
\( h(L_\a) \subseteq \boldsymbol{\a}\eta \times \{ \a \} \).
\end{claim}

\begin{proof}[Proof of the claim]
Suppose towards a contradiction that \( h(x) \in \z L' \times \{ \a' \} \) for some \( x \in L_\a \) and \( \a' < \g \). 
Let \( z \in \Z \) and \( \ell' \in L' \) be such that \( h(x) = ((z,\ell'),\a') \). 
Pick any \( y \in L_\a\) such that  \( y <_{L_\a} x\) and \( (y,x)_{L_\a} \) is infinite. 
Then \( h(y) \leq_{\varphi(L')} ((z',\ell'),\a') \) for every \( z' \leq_\z z \), so that \( \{ ((z',\ell'),\a') \mid z' \leq_\z z \} \subseteq h (L_\a) \) because the latter is convex in \( \varphi(L') \). 
It follows that \( x \) would be the maximum of a convex subset of \( L_\a \cong \eta^{f_\a} \) isomorphic to \( \boldsymbol{\omega}^* \), which is impossible.
Thus for every \( x \in L_\a \) there is \( \a' < \g \) such that \( h(x) \in \boldsymbol{\a}' \eta \times \{ \a' \} \), which easily implies that there is a single \( \a' < \g  \) such that \( h(L_\a) \subseteq \boldsymbol{\a}' \eta \times \{ \a' \} \cong \eta^{f_{\a'}} \). 
Hence \( \eta^{f_\a} \cong L_\a \cvx \eta^{f_{\a'}}\), which by Proposition~\ref{lem:loc_not_scat}\ref{lem:loc_not_scat-b} implies \( \a' = \a \).
\end{proof}

By construction, if \( \a < \a' \) then \( k_\a \leq_K k_{\a'} \). We distinguish two cases. If the map \( \a \mapsto k_\a \) is not injective, then there is \( \a < \g \) such that \( k_\a = k_{\a+1} \). 
It follows that \( \z L \times \{ \a \} \csube M_{k_\a} \), and thus using Claim~\ref{claim:borredord} we can obtain from \( h \) a witness \( h'\) of \( \z L \cvx \boldsymbol{\a} \eta + \z L' + (\boldsymbol{\a}+\1)\eta\). 
Each point \( x \in \z L \) is the maximum of a convex subset of \( \z L \) isomorphic to \( \boldsymbol{\o}^*\), which easily imply \( h'(x) \in \z L' \). 
Thus \( h' \) actually witnesses \( \z L \cvx \z L'\), hence \( L \cvx L' \) by Proposition~\ref{prop:fin_zeta_l}, and so \( L \trianglelefteq^{\L_{\prec \bb}}_\LO L' \).

Assume now that that the map \( \a \mapsto k_\a \) is injective, and hence strictly increasing. 
If all intervals \( [k_\a,k_{\a+1})_K\) had order type \( \geq \b \), then the order type of \( K \) would be at least \( \b \g \), and this would contradict \( K \in \L_{\prec \boldsymbol{\g}+\1}\) by Remark~\ref{ind_ord}\ref{ind_ord_b}. 
Let thus \( \a < \g \) be such that \( [k_\a,k_{\a+1})_K \) has order type \( \xi < \beta \), so that \(  [k_\a,k_{\a+1}]_K \in \L_{\prec \boldsymbol{\b}} \) because \( \b \) is limit. 
Then \( \z L \times \{ \a \} \csube \bigcup_{k \in [k_\a,k_{\a+1}]_K} M_k \), and thus by Claim~\ref{claim:borredord} we obtain \( \z L \trianglelefteq^{\L_{\prec \bb}}_\LO \boldsymbol{\a} \eta + \z L' + (\boldsymbol{\a}+\1)\eta \) as witnessed by some \( K' \preceq [k_\a,k_{\a+1}]_K \), \( (M'_{k'})_{k' \in K'} \), and \( h' \colon \z L \to \boldsymbol{\a} \eta + \z L' + (\boldsymbol{\a}+\1)\eta \). 
For each \( \ell \in L \), let \( k'_\ell \) be the smallest \( k' \in K' \) such that \( M'_{k'} \cap (\z \times \{ \ell \}) \neq \emptyset \), and set \( L'_\ell =  M'_{k'_\ell} \cap (\z \times \{ \ell \}) \). Since \( h' \restriction L'_\ell \) witnesses \( L'_\ell \cvx \boldsymbol{\a} \eta + \z L' + (\boldsymbol{\a}+\1)\eta \), and since \( L'_\ell \) is isomorphic to either \( \z \) or \( \boldsymbol{\o}^*\) by the choice of \( k'_\ell \), it follows that \( h'(L'_\ell) \subseteq \z L'\) because there is no convex subset of \( \boldsymbol{\a} \eta \) and \( (\boldsymbol{\a}+\1)\eta\) isomorphic to \( \boldsymbol{\o}^* \). Exploiting the order type of \( L'_\ell \) once again, it is easy to see that actually there is an embedding \( g \colon L \to L' \) such that \( h'(L'_\ell) \subseteq \z \times \{ g(\ell) \} \). 
For each \( k' \in K' \), let \( A_{k'} = \{ \ell \in L \mid k'_\ell = k' \} \) and \( K'' = \{ k' \in K' \mid A_{k'} \neq \emptyset \} \preceq K' \). It is not hard to see that \( g(A_{k'}) \csube L' \) for every \( k' \in K'' \), and thus \( (A_{k'})_{k' \in K''} \) is a \( K'' \)-convex partition of \( L \) which, together with \( K'' \in \L_{\prec \bb }\) and \( g \colon L \to L' \), witnesses \( L \trianglelefteq^{\L_{\prec \bb}}_\LO L' \).
\end{proof}

Using similar ideas, one can obtain other reducibility results.
To put the next result in proper perspective, notice that \( \trianglelefteq^{\Fin}_\LO \) and \( \trianglelefteq^{\L_{\preceq \z}}_\LO \) coincide with \( \L _{\prec \boldsymbol{\o}}\) and \( \L_{\preceq \Z^1 }\), respectively.

\begin{thm}\label{thm:red_fin_zeta}
\( {\trianglelefteq^{\Fin}_\LO} \leq_B {\trianglelefteq^{\L_{\preceq \z}}_\LO} \).
\end{thm}

\begin{proof}
Let \( h \colon \z \to \{ \n \mid n \in \N \setminus \{ 0 \} \} \) be any bijection. 
Imitating the proof of Theorem~\ref{thm:red_threshold}, we claim that the Borel function 
\[
\varphi (L) = \sum_{z \in \z} (h(z) \eta + \z L ).
\]
is a reduction from $\trianglelefteq^{\Fin}_\LO$ to $\trianglelefteq^{\L_{\preceq \z}}_\LO$. 

The forward direction is easy, so suppose that $\varphi(L)\trianglelefteq^{\L_{\preceq \z}}_\LO \varphi(L')$ as witnessed by $K \preceq \z$, the $K$-convex partition $(M_k)_{k \in K}$ of $\varphi(L)$, and the embedding $h \colon \varphi(L) \to \varphi(L')$.
Since each \( h(z) \eta \times \{ z \} \) is bounded from below in \( \varphi(L) \), there is a \( \leq_K \)-smallest \( k_z \in K \) such that \( M_k \cap (h(z) \eta \times \{ z \}) \neq \emptyset \). 
The rest of the proof follows closely that of Theorem~\ref{thm:red_threshold}. 
\end{proof}

\section{Uncountable linear orders}\label{sec:uncountable}

Roughly speaking, generalized descriptive set theory is obtained by replacing \( \omega \) with an uncountable cardinal \( \kappa \) such that $2^{<\kappa} = \kappa$ in all basic definitions and notions from classical descriptive set theory. 
In this setup, one considers the \textbf{generalized Cantor space} \( 2^\kappa \) of all binary \( \kappa \)-sequences equipped with the topology generated by the sets of the form \( \{ x \in 2^\kappa \mid s \subseteq x \} \), where \( s \) varies among all binary sequences of length \( < \kappa \). 
Borel sets are then replaced by \textbf{\( \kappa^+ \)-Borel sets}, i.e.\ sets in the  \( \kappa^+ \)-algebra generated by the open sets of the given topological space. 
The notions of \textbf{\( \kappa^+ \)-Borel function} and \textbf{\( \kappa^+ \)-Borel reducibility \( \leq^\kappa_B \)} are defined accordingly.
See~\cite{AndMot} for a quite comprehensive introduction to the subject.

The usefulness of this approach is that it allows us to tackle classification problems for uncountable structures with tools which resemble, to some extent, those used in the classical setting---one can look at \cite{FriHytKul,ManMot,HytKulMor} for some of the most significant results in this direction connecting classification/complexity in terms of generalized descriptive set theory with Shelah's stability theory.

In the present setting, the space \( \LO_\kappa \) 
of (codes for) linear orders on \( \kappa \) 
can be endowed with the relative topology inherited from \( 2^{\kappa \times \kappa} \), once the latter is identified in the obvious way with the generalized Cantor space \( 2^\kappa \). This is the same as the topology generated by the neighborhood base of \( L \in \LO_\kappa \) determined by the sets \( \{ L' \in \LO_\kappa \mid L' \restriction \alpha = L \restriction \alpha \} \), for \( \alpha < \kappa \).


Given \( \L \subseteq \Lin_\kappa\),
we denote by $\qok$ the restriction of $\trianglelefteq^{\L}$ to $\LO_{\kappa}$. 
As usual, when \( \L = \{ \1 \} \) we simplify the notation and just write \( \trianglelefteq_\kappa \), while when $\L = \Lin_\kappa$ we have embeddability on $\LO_\kappa$ denoted by $\preceq_\kappa$.
As discussed before Definition~\ref{def:appropriate}, if \( \L \) is ccs then \( \qok \) is transitive: in Theorem \ref{thm:ccs_trans_kappa} we will show that the converse holds under the extra hypothesis $\kappa^{<\kappa} = \kappa$, which is equivalent to the fact that \emph{\( \kappa \) is regular} and \( 2^{< \kappa} = \kappa \). 
Since we do not want to restrict ourselves to regular cardinals, and since for singular cardinals we do not have an analog of Theorem \ref{thm:ccs_trans_kappa}, some of our results will be stated under the (possibly weaker) assumption that \( \qok \) is transitive.

Denote by $\eqk$ the equivalence relation induced by $\qok$ (when the latter is transitive). 
Notice that in this generalized context every class $\L \subseteq \Lin$ is $\kappa^+$-Borel, and hence $\eqk$ is $\kappa$-analytic for any such $\L$. 
Similarly, if \( 2^{< \kappa} = \kappa \) and \( \L \subseteq \LO_{< \kappa} \), then \( \eqk \) is again \( \kappa \)-analytic. However, for the general case \( \L \subseteq \Lin_\kappa \) the relation \( \eqk \) might be more complicated.

It is easy to check that the map \( \varphi \) from the proof of Proposition~\ref{thm:cong_qo} is a \( \kappa^+ \)-Borel map from \( \LO_\kappa \) to itself that witnesses the following theorem, once we let \( M \) be any element of \( \Lin_\kappa \setminus \L \).

\begin{thm} \label{thm:reductionstoLO_kappa}
Let \( \kappa \) be an infinite cardinal such that%
\footnote{This cardinal assumption is needed only to ensure that the family of \( \kappa^+ \)-Borel sets is well-behaved, so that one can meaningfully speak of \( \kappa^+ \)-Borel reducibility. However, no cardinal assumption is actually involved in our proof. A similar comment applies to Theorem~\ref{thm:completenessforuncountableLO}, where the apparently missing condition \( 2^{< \kappa} = \kappa \) is implied by the stronger assumption \( \mathsf{V} = \mathsf{L} \), and to Theorem~\ref{thm:coloured_kappa}.}
$2^{<\kappa} = \kappa$, and let \( \L \subsetneq \Lin_\kappa\) be such that \( \qok \) is transitive.
Then the isomorphism relation \( \cong_\kappa \) on \( \LO_\kappa \) is \( \kappa^+ \)-Borel reducible to \( \eqk \).
\end{thm}

Combining this with~\cite[Theorem 1.13]{HitKulMLQ}, we immediately get the following completeness result, which in the case of $\L= \{\1\}$ is in stark contrast with the countable setting (\cite[Corollaries~3.25 and~3.27]{IMMRW22}).

\begin{thm} \label{thm:completenessforuncountableLO}
Assume \( \mathsf{V} = \mathsf{L} \), let \( \kappa = \lambda^+ \) with \( \lambda \) regular, and let
\( \L \subsetneq \Lin_\kappa\) be such that \( \qok \) is a \( \kappa \)-analytic quasi-order.
Then the relation \( \eqk \) is complete for \( \kappa \)-analytic equivalence relations (with respect to \( \kappa^+ \)-Borel reducibility).
\end{thm}

Let now \((\LO_\kappa)_\kappa = \LO_\kappa \times \kappa^{\kappa}\) be the space of linear orders on $\kappa$ coloured with $\kappa$-many colours, where \( \kappa^\kappa \) is equipped with the product topology and \( (\LO_\kappa)_{\kappa} \) is the product of the topological spaces \( \LO_\kappa \) and \( \kappa^\kappa \). 
Consider the natural generalization of $\L$-convex embeddability to $(\LO_\kappa)_\kappa$ in which the embeddings are required to preserve colours (cfr.\ Definition \ref{def:qo_for_coloured}). 
Obviously, this relation is transitive if and only if so is $\qok$, and using the same argument of the proof of Theorem~\ref{thm:coloured} one obtains the following result.

\begin{thm}\label{thm:coloured_kappa}
Let \( \kappa \) be an infinite cardinal such that $2^{<\kappa} = \kappa$, and let \( \L \subseteq \Scat_\kappa\) be such that \( \qok \) is transitive.
Then the colour-preserving $\L$-convex embeddability relation on \((\LO_\kappa)_\kappa\) is \( \kappa^+ \)-Borel bi-reducible with \( \qok \).
\end{thm}


We now move to the combinatorial properties of \( \qok \).
To carry out this analysis we do not need to assume $2^{<\kappa} = \kappa$.
Proposition~\ref{prop:min_el_qo} can be used to uncover a significant difference between embeddability \( \preceq \) and $\L$-convex embeddability \( \trianglelefteq^\L \) among uncountable linear orders. 
By the celebrated five-element basis theorem of J.\ Moore~\cite{Moo06}, assuming \( \mathsf{PFA} \) there is a finite basis (of size \( 5 \)) for the embeddability relation on uncountable linear orders. If we move to $\L$-convex embeddability, working in \( \mathsf{ZFC} \) alone we instead obtain the following result which, when $\kappa=\aleph_1$ and \( \L = \{ \1 \} \), implies that there is no finite or even countable basis for the convex embeddability relation \( \trianglelefteq \) on the class of uncountable linear orders. 

\begin{thm} \label{thm:basisforuncountable}   
For every uncountable cardinal $\kappa$ and every $\L \subseteq \LO_{< \mathrm{cof}(\kappa)}$ such that \( \qok \) is transitive, there are at least \( 2^{\aleph_0} \)-many \( \qok \)-incomparable \( \qok \)-minimal elements in $\LO_{\kappa}$.
\end{thm}

The assumption $\L \subseteq \LO_{< \mathrm{cof}(\kappa)}$ is quite strong, but it includes many important cases, most notably \( \L = \{ \1 \} \), \( \L = \Fin \), and, if \( \kappa \) has uncountable cofinality, any ccs \( \L \subseteq \Lin \). 
Notice also that this result generalizes Theorem~\ref{thm:basis2}, as \( \LO_{< \mathrm{cof}(\aleph_0)} = \Fin \).

\begin{proof}
First observe that for every \( L \in \LO \), the linear order \( L \kappa = \sum_{\alpha < \kappa} L \) is \( \qok \)-minimal in $\LO_{\kappa}$ because \( L \kappa \) is \( \kappa \)-like, i.e.\ for every \( n \in L \kappa \) the initial segment \( (-\infty, n]_{L \kappa} \) has size $<\kappa$ while any convex subset of \( L \kappa \) of size \( \kappa \) is a tail. 
Indeed, if $K \in \L$, the $K$-convex partition $(L'_k)_{k \in K}$ of \( L' \) and \( f \colon L' \to L \kappa \) witness $L' \qok L \kappa$ for some \( L' \in \LO_{\kappa} \), then there is $k \in K$ such that $L'_k$ has size $\kappa$ 
(here we are using that \( \kappa \) has cofinality greater than \( |K| \)),
and hence \( f(L'_k) \) is a tail of \( L\kappa \).
Let \( \beta < \kappa \) be such that \( \sum_{\beta \leq \alpha < \kappa} L \csube f(L'_k)\). 
Since clearly \( L \kappa \cong \sum_{\beta \leq \alpha < \kappa} L \), we then get that \( L \kappa \qok L'_k \csube L' \) and thus \( L' \mathrel{\eqk} L \kappa \).

Consider now the countable linear orders $\eta^{f_S}$ from Proposition~\ref{prop:min_el_qo}, and let 
\[
\mathcal{A}=\{\eta^{f_S}\kappa \mid S \subseteq \N \setminus \{ 0 \} \text{ is infinite}\}.
\]
Then by the previous paragraph each linear order $\eta^{f_S}\kappa$ is $\qok$-minimal. 
We claim that the elements of $\mathcal{A}$ are also pairwise $\qok$-incomparable. 
Suppose that $\eta^{f_S}\kappa \qok \eta^{f_{S'}}\kappa$ with witnesses $K \in \L$, $(L_k)_{k \in K}$ and $g\colon \eta^{f_S}\kappa \to \eta^{f_{S'}}\kappa$. 
As in the previous paragraph, since $|K| < \mathrm{cof}(\kappa)$ there is $k \in K$ such that $\eta^{f_S} \csube L_k$. 
Thus $g(\eta^{f_S}) \csube \eta^{f_{S'}}\kappa$, and since $g(\eta^{f_S})$ is countable there is $\alpha<\aleph_1$ such that $\eta^{f_S} \cong g(\eta^{f_S}) \trianglelefteq \eta^{f_{S'}}\aa \cong \eta^{f_{S'}}$. 
(In the last step we are using Lemma~\ref{lem:loc_not_scat}\ref{lem:loc_not_scat-e} together with the fact that \( \eta^{f_{S'}} \aa \) can be rewritten as \( \eta^f \) for a suitable surjection \( f \colon \Q \to S' \) satisfying the relevant hypotheses.)
Since \( \{ \1 \} \subseteq \Scat \) is ccs, then \( S = S' \) by Proposition~\ref{prop:min_el_qo}\ref{prop:min_el_qo-a} applied to \( \{ \1 \} \). 
By the fact that $|\mathcal{A}|=2^{\aleph_0}$ we finally get the desired result.
\end{proof}

\begin{remark}
Theorem~\ref{thm:basisforuncountable} is optimal in the context of Moore's theorem because under \( \mathsf{PFA} \) we have \( 2^{\aleph_0} = 2^{\aleph_1} \), and thus the family of \( \qoa \)-minimal elements that we constructed in its proof is as large as possible. 
In particular, in that setup all bases for \( \qoa \) on $\LO_{\aleph_1}$ have maximal size when \( \L \subseteq \Lin \) is ccs.
\end{remark}


Using the technical facts from Section \ref{sec:technical}, we can also extend to the generalized setup the analysis of the other combinatorial properties of $\L$-convex embeddability, obtaining results which are new also in the special case of the convex embeddability relation \( \trianglelefteq_\kappa \).

\begin{lem} \label{lem:open_int_qo_unctbl}
Let \( \kappa \) be an infinite cardinal,
and let \( \L \subseteq \Scat_\kappa \) be such that \( \qok \) is transitive.
Let \( L \) be any linear order with an infinite dense subset \( L' \subseteq L\) of size at most \( \kappa \), and let \( \Int(L) \) be the collection of open intervals of \( L \). Then there is an embedding of \( (\Int(L), {\subseteq}) \) into \( (\LO_\kappa, {\qok}) \).
\end{lem}

\begin{proof}
Let \( D = \eta L' \in \ADLO \). Fix a function \( f \colon D \to \Scat_\kappa \cap \LO_\kappa \) such that $f(d)$ and $f(d')$ are non-isomorphic for all distinct $d, d' \in D$, and consider the linear order \( D^f \) as in Definition~\ref{eta_f}. 
Given an interval \( (x,y) \in \Int(L) \), let \( A_{(x,y)} = \{ (q,\ell) \in D \mid \ell \in (x,y) \cap L' \}  \csube D\). 
Notice that \( A_{(x,y)} \) has neither a minimum nor a maximum, hence Lemma~\ref{lem:loc_not_scat} can be applied to \( D^f_{A_{(x,y)}} \) by Remark~\ref{rmk:loc_not_scat}. 

Now consider the map \( (x,y) \mapsto D^f_{A_{(x,y)}} \). Arguing as in the proof of Lemma~\ref{lem:open_int_qo}, one can easily show that this map is an embedding from \((\Int(L), {\subseteq}) \) into \( (\LO_\kappa, {\qok}) \). 
\end{proof}

Recall that the set of non-negative reals  is order isomorphic to \( ({}^\omega \omega, {\leq_{\mathrm{lex}}}) \), where \( \leq_{\mathrm{lex}}\) is the lexicographic order. 
Given an infinite cardinal \( \kappa \), let \( \mathbb{L}_k = ({}^\kappa \kappa, {\leq_{\mathrm{lex}}}) \) and \( \R_\kappa = \mathbb{L}_\kappa^* + (\mathbb{L}_\kappa \setminus \{ \min \mathbb{L}_\kappa \} )\). 
Then \( \R_{\aleph_0} \cong \R \) and,  \( \R_\kappa \) has a dense subset of size \( \kappa \) if \( \kappa^{< \kappa} = \kappa \). Thus Lemma~\ref{lem:open_int_qo_unctbl} includes as a special case the natural generalization of Lemma~\ref{lem:open_int_qo} where \( \R \) is replaced by \( \R_\kappa \) (assuming \( \kappa^{< \kappa} = \kappa \)).
The following result, instead, generalizes Theorem~\ref{thm:chains_qo}.

\begin{thm} \label{thm:chains_qo_unctbl}
Let \( \kappa \) be an infinite cardinal, and let \( \L \subseteq \Scat_\kappa\) be such that \( \qok \) is transitive. 
Then \( \qok \) has chains isomorphic to any $L \in \Lin_\kappa$ (including increasing and decreasing chains of order type any \( \alpha < \kappa^+ \)) and, if \( \kappa^{< \kappa} = \kappa \), chains of order type \( \R_\kappa \); moreover, it has antichains of size \( 2^\kappa \).
\end{thm}

\begin{proof}
For the chains, it is enough to apply Lemma~\ref{lem:open_int_qo_unctbl} to all linear orders \( L \in \Lin_\kappa\) (among which we find  \( \boldsymbol{\alpha} \) and \( \boldsymbol{\alpha}^* \)) and to \( \R_\kappa\), and observe that \( (\Int(L),{\subseteq}) \) has chains of type \( L \setminus \{ \min L \} \), namely \( ((-\infty,x)_L)_{x \in L \setminus \{ \min L \}} \).
Moving to the construction of an antichain of size \( 2^\kappa \), similarly to the proof of Theorem~\ref{thm:chains_qo}, we can fix a family \( (L_\alpha)_{\alpha < 2^\kappa} \) of pairwise non-isomorphic scattered linear orders of size \( \kappa \) and for each \( \alpha < 2^\kappa \) consider the linear order \( \eta^{f_\alpha} \), where \( f_\alpha \colon \Q \to \Scat_\kappa \) is constant with value \( L_\alpha \).
\end{proof}

In order to study dominating families with respect to \( \qok \), we first need to generalize Proposition~\ref{prop:wo_boundedness}\ref{wo_not_in_L}.

\begin{prop} \label{prop:WO_kappa}
Let \( \kappa \) be an infinite cardinal, and let \( \L \subseteq \Lin_\kappa\) with $\WO_\kappa \nsubseteq \L$ be such that \( \qok \) is transitive. 
Then \( \WO_\kappa \cap \LO_\kappa \) is unbounded with respect to \( \qok \).
\end{prop}

\begin{proof}
First we need to establish the generalization of Proposition \ref{prop:comb_prop_cvx}\ref{prop:WO_unbounded_cvx}, i.e.\ the unboundedness of $\WO_\kappa \cap \LO_\kappa$ with respect to convex embeddability \( \trianglelefteq_\kappa \). 
To this end, we argue as in the proof of \cite[Proposition 3.6]{IMMRW22}, where Proposition \ref{prop:comb_prop_cvx}\ref{prop:WO_unbounded_cvx} was originally established. 
Fix an arbitrary \(L \in \LO_k\), and for every \(x \in L\) define
\[
\alpha_{x,L} = \sup \{\ot(L')\mid L' \text{ is a well-order},\ L'\csube L, \text{ and } x = \min L'\},
\]
where \( \ot(L') \) is the order type of \( L' \), i.e.\ the unique ordinal \( \b \) such that \( L' \cong \boldsymbol{\b}\).
Since \( \alpha_{x,L} \) is attained by definition of \( \csube \), we have \( \alpha_{x,L} < \kappa^+\). 
Let \(\alpha_L = \sup_{x \in L} \alpha_{x,L}< \kappa^+\). 
By construction, if \( L' \csube L\) is well-ordered, then \( \ot(L') \leq \alpha_L \), and hence $\max \{\boldsymbol{\kappa}, \boldsymbol{\alpha}_L + \boldsymbol{1}\} \ntrianglelefteq_\kappa L$.  
Thus for every \(L \in \LO_\kappa\) there exists \(L' \in \WO_\kappa \cap \LO_\kappa\) such that \(L'\ntrianglelefteq_\kappa L\), and hence \( \WO_\kappa \cap \LO_\kappa \) is \( \trianglelefteq_\kappa \)-unbounded in \( \LO_\kappa \).

To obtain unboudedness of $\WO_\kappa \cap \LO_\kappa$ with respect to an arbitrary \( \qok \) such that \( \WO_\kappa \nsubseteq \L \), we can now follow the proof of Proposition \ref{prop:wo_boundedness}\ref{wo_not_in_L}.
\end{proof}

%
%

The next result is the generalized version of Theorem~\ref{no_max}. (Recall that \( \L \subseteq \Scat \) if and only if \( \L \subsetneq \Lin \).)

\begin{thm} \label{no_max_unctbl}
Let \( \kappa \) be an infinite cardinal, and let \( \L \subsetneq \Lin_\kappa\) be such that \( \qok \) is transitive. Then the quasi-order \( \qok \) has no maximal element, and every dominating family with respect to \( \qok \) has size \( 2^{\kappa} \). Hence \( \mathfrak{d}(\qok) = 2^\kappa \).
\end{thm}

\begin{proof}
Let \( L \in \LO_\kappa \). 
By Proposition~\ref{prop:WO_kappa} there is \( \kappa \leq \alpha < \kappa^+ \) such that \( \boldsymbol{\alpha} \not\trianglelefteq_\kappa L\), hence \( L \triangleleft_\kappa L+ \boldsymbol{\alpha}\). 
Fix \( M \in \Lin_\kappa \setminus \L \)  and set \( L' = (L+ \boldsymbol{\alpha})M\). 
Then \( L \triangleleft^\L_\kappa L' \) by Proposition~\ref{prop:nqo}.

We now move to dominating families. 
We first generalize Proposition \ref{prop:comb_prop_cvx}\ref{prop:dom_fam_cvx} by showing that any dominating family with respect to $\trianglelefteq_\kappa$ has size $2^\kappa$. 
Consider again the collection \( (\eta^{f_\alpha})_{\alpha < 2^\kappa } \) from the end of the proof of Theorem~\ref{thm:chains_qo_unctbl}.
We claim that \( |\{ \alpha < 2^\kappa \mid \eta^{f_\alpha} \trianglelefteq_\kappa L \}|=\kappa \) for every $L \in \LO_\kappa$.
Suppose that \( \eta^{f_\alpha} \trianglelefteq_\kappa L \), so that without loss of generality we can assume \( L = L_l + \eta^{f_\alpha} + L_r \) for some (possibly empty) \( L_l \) and \( L_r \). 
If \( f \) were a convex embedding of \( \eta^{f_{\alpha'}} \) into \( L \) with \( f(\eta^{f_{\alpha'}}) \cap \eta^{f_\alpha} \neq \emptyset \), then by Lemma \ref{lem:loc_not_scat}\ref{lem:loc_not_scat-d}--\ref{lem:loc_not_scat-e} we would have \( \alpha'= \alpha \). 
Thus \( f(\eta^{f_{\alpha'}}) \cap \eta^{f_\alpha} = \emptyset \) whenever \( \alpha, \alpha' < 2^\kappa \) are distinct. 
It follows that since \( L \) has size $\kappa$, there can be only $\kappa$-many distinct ordinals \( \alpha\) for which \( \eta^{f_\alpha} \trianglelefteq_\kappa L \) holds.
%
%
%
Now if \( \mathcal{F} \) were a dominating family with respect to \( \trianglelefteq_\kappa \) of size smaller than \(  2^{\kappa} \), then there would be \( M \in \mathcal{F} \) such that \( |\{ \alpha < 2^\kappa \mid \eta^{f_\alpha} \trianglelefteq_\kappa M \}|>\kappa \), contradicting our claim. This concludes the proof for the case \( \L = \{ \1 \} \).

Let now \(\mathcal{F}^\L\) be a dominating family with respect to \(\qok\). 
Using the same argument of the proof of Theorem \ref{no_max} with any $M \in \Lin_\kappa \setminus \L$ in place of $\eta$, one sees that $\mathcal{F}^\L$ is also a dominating family with respect to $\trianglelefteq_\kappa$, and hence $|\mathcal{F}^\L| = 2^{\kappa}$. 
\end{proof}


Using $\kappa$-sums of linear orders of size $\kappa$, one can easily prove that the unbounding number $\mathfrak{b}(\qok)$ of $\qok$ is greater than $\kappa$. 
We now provide the generalized version of Theorem \ref {thm:unbounding_number}, showing that \( \mathfrak{b}(\qok) \) is as small as possible.

\begin{thm} \label{thm:unbounding_number_unctbl}
Let \( \kappa \) be an infinite cardinal, and let \( \L \subsetneq \Lin_\kappa\) be such that \( \qok \) is transitive. 
Then there is a family \( \mathcal{F} \) of size \( \kappa^+ \) which is unbounded with respect to \( \qok \). 
Thus, \( \mathfrak{b}(\qok) = \kappa^+ \).
\end{thm}

\begin{proof}
First observe that since \( \L \neq \Lin_\kappa \) is downward \( \preceq\)-closed and every \( L \in \Lin_\kappa \) embeds into \( \eta L \in \DLO_\kappa \), there is \( D \in \DLO_\kappa \setminus \L \). 
Let \( \mathcal{F} = \{ \boldsymbol{\alpha} D \mid \alpha < \kappa^+ \} \). 

Using the same argument of the proof of Theorem~\ref{thm:unbounding_number} with $D$ in place of $\eta$, and by applying Proposition~\ref{prop:WO_kappa} instead of Proposition~\ref{prop:comb_prop_cvx}\ref{prop:WO_unbounded_cvx}, one can show that the family \( \mathcal{F} \) is a \( \qok \)-antichain of size \( \kappa^+\) which is \( \qok \)-unbounded in \( \LO_\kappa \).
\end{proof}

The following result is the generalization of Theorem~\ref{thm:fractal}.

\begin{thm} \label{thm:fractal_unctbl}
Let \( \kappa \) be an infinite cardinal, and let \( \L \subseteq \Scat_\kappa\) be such that \( \qok \) is transitive.
Then the partial order \( (\LO_\kappa, {\qok} ) \) has the fractal property with respect to its upper cones, that is, 
\((\LO_\kappa,\qok)\) embeds into \(({L_0}{\uparrow^{\L}},{\qok})\) for every
\(L_0 \in \LO_\kappa\).
\end{thm}

\begin{proof}
Fix $L_0 \in \LO_\kappa$ and, by applying Proposition~\ref{prop:WO_kappa}, let $\kappa \leq \alpha < \kappa^+$ be such that $\aa \ntrianglelefteq_\kappa L_0$. 
Consider the map \(\varphi \colon \LO_\kappa \to L_0\uparrow^{\L}\) defined by 
\[
\varphi(L) = ( \aa \eta + \eta + L_0 + \eta)L.
\]
Arguing as in the proof of Theorem~\ref{thm:fractal}, one can show that $\varphi$ is an embedding from \((\LO_\kappa,{\qok})\) to \((L_0\uparrow^{\L},{\qok})\). 
\end{proof}

For the remainder of the section we assume that $\kappa^{<\kappa}=\kappa$, so that in particular $\kappa$ is regular. 
Under this assumption, we show that more results of the previous sections regarding classes $\L \subseteq \Lin$ hold also when $\L \subseteq \Lin_\kappa$. 
A crucial role here is played by the existence (originally due to Hausdorff \cite{Ha08}, see \cite[\S 9.4]{Ros82}) of a (unique, up to isomorphism) linear order $L$ of size $\kappa$ which is \textbf{$\kappa$-saturated}, i.e.\ such that for any two (possibly empty) sets $A, B \subseteq L$ of size less than $\kappa$, if $A <_L B$ then there is $x \in L$ such that $A <_L \{ x \} <_L B$.
Following the notation of \cite{ABCDK12}, we refer to this order as $\Q_\kappa$.
Notice that $\Q_\kappa$ is the analogue of $\eta$ in size $\kappa$: it belongs to $\DLO_\kappa$ and is universal (i.e.\ every $L \in \Lin_\kappa$ is such that $L \preceq \Q_\kappa$).
Notice however that these two properties do not characterize $\Q_\kappa$ when \( \kappa > \aleph_0 \).
A concrete realization of $\Q_\kappa$ is provided by 
\[
\{x \in 2^\kappa \mid  \exists \a < \kappa \, (x(\a) = 1 \land \forall \b>\a\, x(\b)=0)\},
\] 
ordered lexicographically.
Notice that not every convex subset of $\Q_\kappa$ without extrema is isomorphic to $\Q_\kappa$, as e.g.\ $\{x \in \Q_\kappa \mid \exists n<\omega\, (x(n)=1)\}$ has coinitiality $\omega$.
However, the following generalization of a related weaker property of $\eta$ holds.

\begin{fact}\label{fact:cvx_Qk}
Let \( \kappa \) be an infinite cardinal such that $\kappa^{<\kappa}=\kappa$.
For every $L \csube \Q_\kappa$ of size $>1$ there exists $L' \csube L$ isomorphic to $\1+\Q_\kappa+\1$; in particular, \( L \) contains a convex subset isomorphic to the whole \( \Q_\kappa \), and $|L|=\kappa$.
\end{fact}

\begin{proof}
Pick $x_0, x_1 \in L$ with $x_0 <_{\Q_\kappa} x_1$.
Let $s \in 2^{<\kappa}$ be such that $\{x \in \Q_\kappa \mid s \subseteq x\} \subseteq (x_0,x_1)_{\Q_\kappa}$. Set $x=s^\smallfrown 1^\smallfrown 0^\kappa$ and $y=s^\smallfrown 1^\smallfrown 1^\smallfrown 0^\kappa$, where $0^\kappa$ is the constant sequence of length $\kappa$ of zeros. 
Then \( (x,y)_{\Q_\kappa} \) is \( \kappa \)-saturated (as it has both coinitiality and cofinality \( \kappa \)), thus by a back-and-forth argument it is easy to see that $L'=[x,y]_{\Q_\kappa} \cong \1+ \Q_\kappa+\1$.
\end{proof}

As in~\cite{ABCDK12}, we consider the following notion: 

\begin{defn}
    Let \( \kappa \) be an infinite cardinal such that $\kappa^{<\kappa}=\kappa$. A linear order $L$ is \textbf{$\Q_\kappa$-scattered } if $\Q_\kappa \npreceq L$. 
    We denote by $\Q_\kappa\text{-}\AScat$ the set of all $\Q_\kappa$-scattered linear orders, and by \( \QScat \) its restriction to $\Lin_\kappa$.
\end{defn}

\begin{remark}\label{rem:q_scattered}
By using the universality of $\Q_\kappa$, one can easily observe that:
\begin{enumerate-(a)}
\item\label{rem:q_scattered-a} \(\Scat_\kappa \subseteq \QScat\);
\item\label{rem:q_scattered-b} let $\L$ be nonempty and downward \( \preceq \)-closed: then $\L\subsetneq \Lin_\kappa$ if and only if $\L \subseteq \QScat$.
\end{enumerate-(a)}	
\end{remark}

The following result generalizes Proposition \ref{prop:crucial}.

\begin{prop}\label{prop:scatt_in_class}
Let \( \kappa \) be an infinite cardinal such that $\kappa^{<\kappa}=\kappa$. 
For every $\L \subseteq \Lin_\kappa$ and $L \in \Q_\kappa \text{-}\AScat$, we have $L \trianglelefteq^\L \Q_\kappa$ if and only if $L \in \L$. 
\end{prop}

\begin{proof}
For the forward direction, argue as in the proof of Proposition~\ref{prop:crucial} using Fact~\ref{fact:cvx_Qk} instead of Remark~\ref{rem:cvx_Q}. For the converse, one can use the fact that $\Q_\kappa$ is universal for linear orders of size at most \( \kappa \) together with Fact \ref{fct:basic}. 
\end{proof} 

Using Remark~\ref{rem:q_scattered}\ref{rem:q_scattered-b} and replacing $\eta$ with $\Q_\kappa$, $\Scat$ with $\QScat$, Proposition~\ref{prop:crucial} with Proposition~\ref{prop:scatt_in_class}, and Remark~\ref{rem:cvx_Q} with Fact~\ref{fact:cvx_Qk} in the proof of Proposition \ref{prop:refine}, we obtain the following:

\begin{prop}
Let \( \kappa \) be an infinite cardinal such that $\kappa^{<\kappa}=\kappa$, and let \( \L, \L' \subseteq \Lin_\kappa \). 
Then \( \trianglelefteq^\L \) refines \( \trianglelefteq^{\L'} \) if and only if \( \L \subseteq \L' \).
\end{prop}

We can also generalize Theorem \ref{thm:ccs_trans} and show that when \( \kappa^{< \kappa} = \kappa \), the ccs property for classes $\L \subseteq \Lin_\kappa$ is not only sufficient for the transitivity of $\qok$, but is also necessary. 

\begin{thm}\label{thm:ccs_trans_kappa}
Let \( \kappa \) be an infinite cardinal such that $\kappa^{<\kappa}=\kappa$, and let $\L \subseteq \Lin_\kappa$ be nonempty and downward \( \preceq \)-closed. Then the following are equivalent:
\begin{enumerate-(i)}
	\item\label{thm:ccs_trans_kappa-1} $\L$ is ccs;
	\item\label{thm:ccs_trans_kappa-2} $\trianglelefteq^{\L}$ is transitive;

    \item\label{thm:ccs_trans_kappa-3} the restriction of $\trianglelefteq^{\L}$ to $\Lin_{\kappa}$ is transitive;
	\item\label{thm:ccs_trans_kappa-4} $\qok$ is transitive.
\end{enumerate-(i)}
\end{thm}

\begin{proof}
By the argument preceding Definition~\ref{def:appropriate} we have \ref{thm:ccs_trans_kappa-1} $\Rightarrow$ \ref{thm:ccs_trans_kappa-2}, while \ref{thm:ccs_trans_kappa-2} $\Rightarrow$ \ref{thm:ccs_trans_kappa-3} and \ref{thm:ccs_trans_kappa-3} $\Rightarrow$ \ref{thm:ccs_trans_kappa-4} are obvious. 
We now show \ref{thm:ccs_trans_kappa-4} $\Rightarrow$ \ref{thm:ccs_trans_kappa-3} and \ref{thm:ccs_trans_kappa-3} $\Rightarrow$ \ref{thm:ccs_trans_kappa-1}.

For \ref{thm:ccs_trans_kappa-4} $\Rightarrow$ \ref{thm:ccs_trans_kappa-3}, we use the fact that $\kappa$ has inherently cofinal tails in the sense of Section~\ref{sec:technical}. 
Let $L, L', L'' \in \Lin_\kappa$ be such that $L \trianglelefteq^{\L} L'$ and $L'\trianglelefteq^{\L} L''$. 
Then by Proposition \ref{prop:fin_zeta_l} it follows that $\kappa L \qok \kappa L'$ and $\kappa L' \qok \kappa L''$, and since $\qok$ is transitive by~\ref{thm:ccs_trans_kappa-4}, we have $\kappa L \qok \kappa L''$. 
Applying Proposition \ref{prop:fin_zeta_l} once again, we obtain $L \trianglelefteq^{\L} L''$, as desired.

It remains to show \ref{thm:ccs_trans_kappa-3} $\Rightarrow$ \ref{thm:ccs_trans_kappa-1} and, arguing as in Theorem \ref{thm:ccs_trans}, it suffices to prove this when $\Fin \subseteq \L \subseteq \QScat$.

Suppose that the restriction of \( \trianglelefteq^{\L} \) to $\Lin_\kappa$ is transitive:
given \(K,K' \in \L\) and \((K'_k)_{k \in K}\) such that \( \emptyset \neq K'_k \csube K'\) and \(\forall k_0,k_1 \in K\ (k_0 <_K k_1 \Rightarrow K'_{k_0} \leq_{K'} K'_{k_1})\), we want to show that $L=\sum_{k \in K} K'_k \in \L$ (notice that the size of $L$ is not necessarily $\kappa$). 
Let \(K''= \bigcup_{k \in K} K'_k\), so that $K'' \in \L$ (since $\L$ is downward $\preceq$-closed and \( K'' \preceq K' \)). 
For every $k' \in K''$ we set $A_{k'}=\{k \in K \mid k' \in K'_k\}$, and we let $L'_{k'}=\1$ if \( A_{k'} \) is a singleton, or we let $L'_{k'}$ be one of \( \Q_\kappa \), \( \1 + \Q_\kappa \), \( \Q_\kappa+\1\), \( \1+\Q_\kappa+\1 \) if \( A_{k'} \) is not a singleton, the four cases depending on whether \( A_{k'} \) has a minimum or a maximum.
Let \( L'=\sum_{k' \in K''} L'_{k'} \in \Lin_\kappa \).

First notice that $L \trianglelefteq^{\L} L'$, as witnessed by $K \in \L$ and $(K'_k)_{k \in K}$.
Moreover, one can see that \( L' \trianglelefteq^{\L} \Q_\kappa \) with witnesses $K'' \in \L$ and $(L'_{k'})_{k' \in K''}$ using the fact that $\boldsymbol{3}K'' \preceq \Q_\kappa$ (so that we can discretely embed $K''$ into $\Q_\kappa$) and Fact~\ref{fact:cvx_Qk}. 
By the transitivity of $\trianglelefteq^{\L}$ on $\Lin_\kappa$ we have $L \trianglelefteq^{\L} \Q_\kappa$, and since $L \in \Q_\kappa\text{-}\AScat$ we can apply Proposition~\ref{prop:scatt_in_class} to conclude that $L \in \L$.
\end{proof}

Notice that by the previous theorem, when $\kappa^{<\kappa}=\kappa$ we can reformulate all the results of this section preceding Fact \ref{fact:cvx_Qk} replacing in their hypotheses the transitivity of $\qok$ with the ccs property for $\L$.

Using $\Q_\kappa$ instead of $\eta$ and Fact \ref{fact:cvx_Qk}, it is also easy to prove an analogue of Lemma \ref{lem:loc_not_scat} where we replace $\AScat$ with $\Q_\kappa\text{-}\AScat$ and require $D \in \ADLO$ to belong to $\Lin_\infty \setminus \Q_\kappa\text{-}\AScat$. 
Applying this, it is easy to obtain the following versions of some of the previous results under the optimal hypothesis \( \L \subsetneq \Lin_\kappa \), which is equivalent to \( \L \subseteq \QScat \) by Remark~\ref{rem:q_scattered}\ref{rem:q_scattered-b}.

\begin{thm}
Let \( \kappa \) be an infinite cardinal such that $\kappa^{<\kappa}=\kappa$,
and let \( \L \subsetneq \Lin_\kappa \) be ccs.
Then the conclusions about \( \qok \) stated in Lemma \ref{lem:open_int_qo_unctbl} and Theorems \ref{thm:chains_qo_unctbl} and \ref{thm:fractal_unctbl} hold.
\end{thm}

Using linear orders $D \in \ADLO$ belonging to $\Lin_\infty \setminus \Q_\kappa\text{-}\AScat$, and applying Lemma \ref{lem:loc_not_scat} in a suitable way, one could also generalize results about $\qok$-minimal elements and $\qok$-antichains similar to those stated in Proposition \ref{prop:min_el_qo}, Theorem \ref{thm:basis2} and Proposition \ref{prop:scatantichains}; however, we think that in the uncountable context Theorem~\ref{thm:basisforuncountable} is more significant, and thus we do not fully develop the details of such results.

Finally notice that when $\kappa^{<\kappa}=\kappa$ and $\L=\Lin_\kappa$, then the conclusions of Proposition~\ref{prop:WO_kappa} and Theorems~\ref{no_max_unctbl}--\ref{thm:unbounding_number_unctbl} fail by maximality of $\Q_\kappa$ with respect to $\preceq_\kappa$,
which means that also those results are stated under optimal hypotheses and cannot be further improved.

\section{Open problems}

In this paper we characterized the classes $\L \subseteq \Lin$ such that the binary relation $\qo$ is a quasi-order, and we described many combinatorial properties and lower bounds with respect to Borel reducibility of $\L$-convex (bi-)embeddability. 
Here we briefly discuss a few questions which in our opinion deserve further investigation.

Definition~\ref{def:appropriate} characterizes those classes \( \L \subseteq \Lin \) for which \( \trianglelefteq^\L \) is transitive (and hence a quasi-order) via a combinatorial property, namely the ccs property, whose validity is not always immediate to check.
It is thus natural to ask whether there is a reformulation of being ccs which can be used to establish the transitivity of a given \( \trianglelefteq^\L \) in an easier way.
To see some of the difficulties in finding a more manageable characterization of ccs classes, notice that while $\L_{\preceq \boldsymbol{\o}^*+\boldsymbol{\o}}$ is ccs by Proposition~\ref{prop:classes_zeta}, Remark~\ref{rem:non_ccs} implies that $\L_{\preceq \boldsymbol{\o}+\boldsymbol{\o}^*}$ is not ccs. 
Analogously, using an argument similar to that from the proof of Proposition~\ref{prop:classes_zeta} one can show that $\L_{\preceq \boldsymbol{\o}^*+\boldsymbol{\o}^2}$ is ccs, while $\L_{\preceq \boldsymbol{\o}^2+\boldsymbol{\o}^*}$ is not ccs (again by Remark~\ref{rem:non_ccs}).




In Theorem \ref{thm:red_E1_eq} we showed that $E_1$ is a lower bound with respect to Borel reducibility for $\eq$. 

\begin{question}
Given a Borel ccs class $\L\subseteq \Scat$ different from $\{\1\}$, is $\eq$ complete for analytic equivalence relations?
\end{question}

Moreover, in Section~\ref{sec:examples} we obtained a few Borel reductions among different quasi-orders $\qo$.

\begin{question}
Are the Borel reducibilities from Theorems~\ref{thm:red_qog_qolg},~\ref{thm:red_threshold}, and~\ref{thm:red_fin_zeta} strict?
\end{question}

More in general, it would be interesting to establish more Borel reducibilities and/or prove some non-reducibilities between different quasi-orders $\qo$.



\begin{thebibliography}{IMMRW23}
	
	\bibitem[ABC{\etalchar{+}}12]{ABCDK12}
	Uri Abraham, Robert Bonnet, James Cummings, Mirna D\u{z}amonja, and Katherine Thompson.
	\newblock A scattering of orders.
	\newblock {\em Trans. Amer. Math. Soc.}, 364(12):6259--6278, 2012.
	
	\bibitem[AMR22]{AndMot}
	Alessandro Andretta and Luca Motto~Ros.
	\newblock Souslin quasi-orders and bi-embeddability of uncountable structures.
	\newblock {\em Mem. Amer. Math. Soc.}, 277(1365):vii+189, 2022.
	
	\bibitem[Bon75]{Bon75}
	Robert Bonnet.
	\newblock On the cardinality of the set of initial intervals of a partially ordered set.
	\newblock In {\em Infinite and finite sets ({C}olloq., {K}eszthely, 1973; dedicated to {P}. {E}rd\H{o}s on his 60{th} birthday), {V}ol. {I}}, pages 189--198. Colloq. Math. Soc. János Bolyai, Vol. 10. North-Holland, Amsterdam, 1975.
	
	\bibitem[CCM19]{CCM19}
	Riccardo Camerlo, Rapha\"{e}l Carroy, and Alberto Marcone.
	\newblock Linear orders: when embeddability and epimorphism agree.
	\newblock {\em J. Math. Log.}, 19(1):1950003, 32, 2019.
	
	\bibitem[CMMR18]{CamMarMot2018}
	Riccardo Camerlo, Alberto Marcone, and Luca Motto~Ros.
	\newblock On isometry and isometric embeddability between ultrametric {P}olish spaces.
	\newblock {\em Adv. Math.}, 329:1231--1284, 2018.
	
	\bibitem[FHK14]{FriHytKul}
	Sy-David Friedman, Tapani Hyttinen, and Vadim Kulikov.
	\newblock Generalized descriptive set theory and classification theory.
	\newblock {\em Mem. Amer. Math. Soc.}, 230(1081):vi+80, 2014.
	
	\bibitem[FS89]{FS89}
	Harvey Friedman and Lee Stanley.
	\newblock A {B}orel reducibility theory for classes of countable structures.
	\newblock {\em J. Symbolic Logic}, 54(3):894--914, 1989.
	
	\bibitem[Gao09]{Gao09}
	Su~Gao.
	\newblock {\em Invariant descriptive set theory}, volume 293 of {\em Pure and Applied Mathematics (Boca Raton)}.
	\newblock CRC Press, Boca Raton, FL, 2009.
	
	\bibitem[Hau08]{Ha08}
	F.~Hausdorff.
	\newblock Grundz\"{u}ge einer {T}heorie der geordneten {M}engen.
	\newblock {\em Math. Ann.}, 65(4):435--505, 1908.
	
	\bibitem[HK15]{HitKulMLQ}
	Tapani Hyttinen and Vadim Kulikov.
	\newblock On {$\Sigma^1_1$}-complete equivalence relations on the generalized {B}aire space.
	\newblock {\em MLQ Math. Log. Q.}, 61(1-2):66--81, 2015.
	
	\bibitem[HKM17]{HytKulMor}
	Tapani Hyttinen, Vadim Kulikov, and Miguel Moreno.
	\newblock A generalized {B}orel-reducibility counterpart of {S}helah's main gap theorem.
	\newblock {\em Arch. Math. Logic}, 56(3-4):175--185, 2017.
	
	\bibitem[IMMRW23]{IMMRW22}
	M.~Iannella, A.~Marcone, L.~Motto~Ros, and V.~Weinstein.
	\newblock Convex embeddability and knots.
	\newblock 2023.
	\newblock arXiv:2309.09910.
	
	\bibitem[Kec95]{Kec95}
	Alexander~S. Kechris.
	\newblock {\em Classical descriptive set theory}, volume 156 of {\em Graduate Texts in Mathematics}.
	\newblock Springer-Verlag, New York, 1995.
	
	\bibitem[KL97]{KL97}
	Alexander~S. Kechris and Alain Louveau.
	\newblock The classification of hypersmooth {B}orel equivalence relations.
	\newblock {\em J. Amer. Math. Soc.}, 10(1):215--242, 1997.
	
	\bibitem[Lav71]{La71}
	Richard Laver.
	\newblock On {F}ra\"{\i}ss\'{e}'s order type conjecture.
	\newblock {\em Ann. of Math. (2)}, 93:89--111, 1971.
	
	\bibitem[MMR21]{ManMot}
	Francesco Mangraviti and Luca Motto~Ros.
	\newblock A descriptive main gap theorem.
	\newblock {\em J. Math. Log.}, 21(1):Paper No. 2050025, 40, 2021.
	
	\bibitem[Moo06]{Moo06}
	Justin~Tatch Moore.
	\newblock A five element basis for the uncountable linear orders.
	\newblock {\em Ann. of Math. (2)}, 163(2):669--688, 2006.
	
	\bibitem[MR04]{MR04}
	Alberto Marcone and Christian Rosendal.
	\newblock The complexity of continuous embeddability between dendrites.
	\newblock {\em J. Symbolic Logic}, 69(3):663--673, 2004.
	
	\bibitem[Orr95]{Orr95}
	John~Lindsay Orr.
	\newblock Shuffling of linear orders.
	\newblock {\em Canad. Math. Bull.}, 38(2):223--229, 1995.
	
	\bibitem[Ros82]{Ros82}
	Joseph~G. Rosenstein.
	\newblock {\em Linear orderings}, volume~98 of {\em Pure and Applied Mathematics}.
	\newblock Academic Press, 1982.
	
	\bibitem[Sha21]{Sha21}
	Assaf Shani.
	\newblock Classifying invariants for {$E_1$}: a tail of a generic real, 2021.
	\newblock arXiv:2112.12881.
	
	\bibitem[\v{C}69]{Ce69}
	Eduard \v{C}ech.
	\newblock {\em Point sets}.
	\newblock Academic Press, 1969.
	\newblock Translated from the Czech by Ale\v{s} Pultr.
	
\end{thebibliography}
\newcommand{\etalchar}[1]{$^{#1}$}

\end{document}